\documentclass[a4paper,11pt,fleqn]{amsart}

\usepackage{amssymb,amsmath,amsthm}
\usepackage{graphicx}
\usepackage{color}
\usepackage{thmtools}
\usepackage{thm-restate}
\usepackage[shortlabels]{enumitem}
\usepackage{mathrsfs}
\usepackage[margin=1in]{geometry}
\usepackage[utf8]{inputenc}
\usepackage[T1]{fontenc}
\usepackage{scrextend}

\newtheorem{theorem}{Theorem}[section]
\newtheorem{theoreml}{Theorem}[section]

\newtheorem{lemma}[theorem]{Lemma}
\newtheorem{proposition}[theorem]{Proposition}

\newtheorem{question}[theorem]{Question}
\newtheorem{corollary}[theorem]{Corollary}

\theoremstyle{definition}
\newtheorem{definition}[theorem]{Definition}
\newtheorem{example}[theorem]{Example}

\theoremstyle{remark}
\newtheorem{remark}[theorem]{Remark}

\numberwithin{equation}{section}

\newcommand{\frakc}{\mathfrak{c}}
\newcommand{\frakd}{\mathfrak{d}}

\newcommand{\eps}{\varepsilon}

\newcommand{\N}{\mathbb{N}}
\newcommand{\R}{\mathbb{R}}
\newcommand{\Q}{\mathbb{Q}}

\newcommand{\F}{\mathbb{F}}

\newcommand{\aA}{\mathcal{A}}
\newcommand{\bB}{\mathcal{B}}

\newcommand{\fF}{\mathcal{F}}
\newcommand{\gG}{\mathcal{G}}
\newcommand{\hH}{\mathcal{H}}
\newcommand{\iI}{\mathcal{I}}
\newcommand{\jJ}{\mathcal{J}}

\newcommand{\lL}{\mathcal{L}}

\newcommand{\nN}{\mathcal{N}}
\newcommand{\pP}{\mathcal{P}}
\newcommand{\qQ}{\mathcal{Q}}
\newcommand{\rR}{\mathcal{R}}
\newcommand{\sS}{\mathcal{S}}

\newcommand{\uU}{\mathcal{U}}

\newcommand{\xX}{\mathcal{X}}
\newcommand{\yY}{\mathcal{Y}}
\newcommand{\zZ}{\mathcal{Z}}

\newcommand{\cCs}{\mathscr{C}}
\newcommand{\cCc}{\cCs}

\DeclareMathOperator{\ba}{ba}

\newcommand{\concat}{{^\smallfrown}}
\DeclareMathOperator{\conv}{{conv}}

\DeclareMathOperator{\exh}{Exh}
\DeclareMathOperator{\fin}{Fin}

\newcommand{\ol}{\overline}

\newcommand{\rstr}{\restriction}

\newcommand{\sm}{\setminus}
\newcommand{\sub}{\subseteq}
\DeclareMathOperator{\supp}{supp}

\newcommand{\wh}{\widehat}

\newcommand{\seq}[2]{\big\langle#1\colon\ #2\big\rangle}
\newcommand{\seqn}[1]{\big\langle#1\colon\ n\io\big\rangle}

\newcommand{\seqk}[1]{\big\langle#1\colon\ k\io\big\rangle}

\newcommand{\seqm}[1]{\big\langle#1\colon\ m\io\big\rangle}

\newcommand{\seqi}[1]{\big\langle#1\colon\ i\io\big\rangle}

\newcommand{\gen}[1]{\mathopen{}\left\langle#1\mathclose{}\right\rangle}

\newcommand{\clopen}[1]{\left[#1\right]}

\newcommand{\ctblsub}[1]{\left[#1\right]^\omega}
\newcommand{\finsub}[1]{\left[#1\right]^{<\omega}}

\newcommand{\iA}{\in\aA}

\newcommand{\io}{\in\omega}%{\in\N}

\newcommand{\wo}{{\wp(\omega)}}%{{\wp(\N)}}
\newcommand{\bo}{{\beta\omega}}%{{\beta\N}}
\newcommand{\oo}{\omega^\omega}%{\N^\N}

\newcommand{\cso}{\ctblsub{\omega}}%{\ctblsub{\N}}
\newcommand{\fso}{\finsub{\omega}}%{\finsub{\N}}

\newcommand{\seqtwo}{2^{<\omega}}

\newcommand{\Cantor}{2^\omega}

\newcommand{\erdos}{{Erd\H{o}s}}

\DeclareMathOperator{\tr}{tr}

\begin{document}

\title[The Nikodym and Grothendieck properties of ideals]{The Nikodym and Grothendieck properties of Boolean algebras and rings related to ideals}

\author[D.\ Sobota]{Damian Sobota}
\address{Kurt G\"{o}del Research Center, Department of Mathematics, University of Vienna, Vienna, Austria.}
\email{ein.damian.sobota@gmail.com}
\urladdr{www.logic.univie.ac.at/~{}dsobota}

\author[T.\ \.{Z}uchowski]{Tomasz \.{Z}uchowski}
\address{Mathematical Institute, University of Wroc\l aw, Wroc\l aw, Poland.}
\email{tomasz.zuchowski@math.uni.wroc.pl}
\thanks{The first author was supported by the Austrian Science Fund (FWF), Grants I 2374-N35 and ESP 108-N. The second author was supported by the Austrian Science Fund (FWF), Grant I 5918-N}

\begin{abstract}
%We study the relations between the Nikodym properties of ideals and Boolean algebras generated by them.
%For an ideal $\iI$ in a Boolean algebra $\aA$ we show that $\iI$ has the Nikodym property if and only if the Boolean subalgebra $\aA\gen{\iI}$ of $\aA$ generated by $\iI$ has the Nikodym property, and if $\aA$ has the Nikodym property but $\aA\gen{\iI}$ does not, then $\aA\gen{\iI}$ does not have the Grothendieck property either. 
For an ideal $\iI$ in a $\sigma$-complete Boolean algebra $\aA$, we show that if the Boolean subalgebra $\aA\gen{\iI}$ of $\aA$ generated by $\iI$ does not have the Nikodym property, then $\aA\gen{\iI}$ does not have the Grothendieck property either. The converse to this statement however does not hold---we construct a family of $\frakc$ many pairwise non-isomorphic Boolean subalgebras of the power set $\wo$ of the form $\wo\gen{\iI}$ which, when thought of as subsets of the Cantor space $2^\omega$, belong to the Borel class $\F_{\sigma\delta}$ and have the Nikodym property but not the Grothendieck property, and a family of $2^\frakc$ many pairwise non-isomorphic non-analytic Boolean algebras of the form $\wo\gen{\iI}$ with the Nikodym property but without the Grothendieck property. The main tool used for proving the above results is a new characterization of the Nikodym property of Boolean algebras $\wo\gen{\iI}$ in terms of sequences of non-negative measures.

Extending a result of Hern\'{a}ndez-Hern\'{a}ndez and Hru\v{s}\'{a}k, we show that for an analytic P-ideal $\iI$ on $\omega$ the following conditions are equivalent: 1) $\iI$ is totally bounded, 2) $\iI$ has the Local-to-Global Boundedness Property for submeasures, 3) the quotient Boolean algebra $\wo/\iI$ contains a countable splitting family, 4) $\conv\le_K\iI$, i.e. $\iI$ is Kat\v{e}tov above the ideal $\conv$ generated by convergent sequences of rational numbers. Moreover, proving a conjecture of Drewnowski, Florencio, and Pa\'ul, we present examples of analytic P-ideals on $\omega$ with the Nikodym property but without the Local-to-Global Boundedness Property for submeasures (and so not totally bounded). Exploiting a construction of Alon, Drewnowski, and \L uczak, we also describe a family of $\frakc$ many pairwise non-isomorphic ideals on $\omega$, induced by sequences of Kneser hypergraphs, which all have the Nikodym property but not the Nested Partition Property---this answers a question of Stuart.

Adapting classical arguments, we show that the class $\mathcal{AN}$ of those ideals $\iI$ on $\omega$ for which the Boolean algebra $\wo\gen{\iI}$ does not have the Nikodym property and this is witnessed by a sequence of measures with supports contained in $\omega$, ordered by the Kat\v{e}tov ordering, is Tukey equivalent to the poset $(\oo,\le^*)$. Finally, we observe that $\iI\le_K\tr(\nN)$ for every $\iI\in\mathcal{AN}$, i.e. every ideal in $\mathcal{AN}$ is Kat\v{e}tov below the trace ideal $\tr(\nN)$ of the Lebesgue null ideal $\nN$.
\end{abstract}

% PRIMARY:
%   03E15   Descriptive set theory
%   28A33   Spaces of measures, convergence of measures
%   28A60   Measures on Boolean rings, measure algebras
%   28E15   Other connections with logic and set theory

% SECONDARY:
%   03E04   Ordered sets and their cofinalities; pcf theory
%   06E15   Stone spaces (Boolean spaces) and related structures
%   28A05   Classes of sets (Borel fields, $\sigma$-rings, etc.), measurable sets, Suslin sets, analytic sets
%   46E15   Banach spaces of continuous, differentiable or analytic functions
%   46E27   Spaces of measures
%   54B15   Quotient spaces, decompositions in general topology


\maketitle

\section{Introduction\label{sec:intro}}

%CO DODAĆ:
% -- nt. pytań o topologiczne charakteryzacje (N) i (G)
% -- "unifying" (N) dla pierścieni, ideałów i algebr (Drewnowski i s-ka)
% -- nowe przykłady dostajemy (wszystkiego! ;))
% -- nt. (sN) i (wN) (nowe przykłady?)
% -- Aizpuru i N.c.p.

For a Boolean algebra $\aA$ let $\ba(\aA)$ denote the Banach space of all bounded finitely additive real-valued measures on $\aA$ endowed with the variation norm (see Section \ref{sec:prelim} for all further relevant notation and notions). We start with the following standard and well-known definitions.

\begin{definition}\label{def:nikodym}
A Boolean algebra $\aA$ has the \textit{Nikodym property} if every sequence $\seqn{\mu_n}$ in $\ba(\aA)$ such that $\lim_{n\to\infty}\mu_n(A)=0$ for every $A\in\aA$ is norm bounded.
\end{definition}

Equivalently, a Boolean algebra $\aA$ has the Nikodym property if every sequence $\seqn{\mu_n}$ in $\ba(\aA)$ such that $\sup_{n\io}\big|\mu_n(A)\big|<\infty$ for every $A\in\aA$ is norm bounded (see \cite[Proposition 2.4]{SZ19}).

\begin{definition}\label{def:grothendieck}
    A Boolean algebra $\aA$ has the \textit{Grothendieck property} if every norm bounded sequence $\seqn{\mu_n}$ in $\ba(\aA)$ such that $\lim_{n\to\infty}\mu_n(A)=0$ for every $A\in\aA$ is weakly convergent to $0$.
\end{definition}

Both of the above properties belong to the most important and intensively investigated measure-theoretic properties of Boolean algebras, in particular of fields of sets, having numerous applications in Banach space theory, vector measure theory, theory of topological vector spaces, matrix analysis, etc., see e.g. \cite{AS85}, \cite{BS24},  \cite{Die84}, \cite{DU77}, \cite{DS58}, \cite{FSR04}, \cite{GK21}, \cite{Sch82}.

The class of Boolean algebras having both of the properties include all $\sigma$-complete Boolean algebras (Nikodym \cite{Nik33b}, And\^{o} \cite{And61}; Grothendieck \cite{Gro53}) or, more generally, all Boolean algebras whose Stone spaces are F-spaces (Seever \cite{See68}); see also \cite{Aiz88}, \cite{Aiz92}, \cite{Das81}, \cite{DFH76}, \cite{Fre84_vhs}, \cite{Hay81}, \cite{Mol81}, \cite{SZ19} for more examples. On the other hand, if the Stone space $St(\aA)$ of a Boolean algebra $\aA$ contains a non-trivial convergent sequence, then $\aA$ has neither of the properties (cf. Example \ref{example:strong_conv_seq}; see also \cite{MS24} and \cite{Zuc25} for more general results). It is however more difficult to find examples of algebras which have only one of the properties. The Boolean algebra $\mathscr{J}$ of Jordan measurable subsets of the unit interval $[0,1]$ was the first known example of a Boolean algebra with the Nikodym property but without the Grothendieck property (Schachermayer \cite{Sch82}; see also Graves and Wheeler \cite{GW83} and Valdivia \cite{Val13} for generalizations). Boolean algebras with the Grothendieck property but without the Nikodym property have been so far constructed only consistently---by Talagrand \cite{Tal84} under the Continuum Hypothesis, by the first author and Zdomskyy \cite{SZ23} under Martin's Axiom, and by G\l odkowski and Widz \cite{GW25} in a specific forcing extension; no \textsf{ZFC} example is known.

In the present work we provide further examples of Boolean algebras with the Nikodym property and without the Grothendieck property. Constructed examples are subalgebras of the Boolean algebra $\wo$, the power set of $\omega$, and have the form $\fF\cup\fF^*$ where $\fF$ is a filter on $\omega$ and $\fF^*$ is its dual ideal. The used filters are dual to so-called \erdos--Ulam ideals, which yields that the constructed Boolean algebras are of low Borel complexity when thought of as subsets of the Cantor space $2^\omega$. %Our first main theorems thus read as follows (here, $\frakc$ denotes the continuum).

\begin{theoreml}\label{thm:main_large_borel}
There is a family $\hH_1$ of $\frakc$ many pairwise non-isomorphic $\mathbb{F}_{\sigma\delta}$ Boolean subalgebras of $\wo$, each of them having the Nikodym property but not the Grothendieck property.
\end{theoreml}

The cardinality of $\hH_1$, i.e. the continuum $\frakc$, is optimal, as there are only $\frakc$ many Borel subsets of $2^\omega$. By using appropriate non-analytic filters we may however obtain larger families of Boolean subalgebras of $\wo$ with the Nikodym property but without the Grothendieck property.

\begin{theoreml}\label{thm:main_large_nonborel}
There is a family $\hH_2$ of $2^\frakc$ many pairwise non-isomorphic non-analytic Boolean algebras of $\wo$, each of them having the Nikodym property but not the Grothendieck property.
\end{theoreml}

The proofs of Theorems \ref{thm:main_large_borel} and \ref{thm:main_large_nonborel} are presented in Section \ref{sec:definable_large}. Here we only briefly discuss their main ingredients, each of them being in fact of its own separate interest.

\medskip

Our first step towards Theorems \ref{thm:main_large_borel} and \ref{thm:main_large_nonborel} relies on the study of so-called Nikodym concentration points of unbounded sequences of measures on Boolean algebras without the Nikodym property, initiated in \cite{Aiz92} and \cite{Sob19}. Briefly, given a Boolean algebra $\aA$, let us say that a sequence $\seqn{\mu_n}$ of measures on $\aA$ is \textit{anti-Nikodym} if $\lim_{n\to\infty}\mu_n(A)=0$ for every $A\in\aA$ and $\sup_{n\io}\big\|\mu_n\big\|=\infty$, and that in this case a point $t$ in the Stone space $St(\aA)$ of $\aA$ (i.e. an ultrafilter $t$ in $\aA$) is its \textit{Nikodym concentration point} if for each $A\in t$ we have $\sup_{n\io}\big\|\mu_n\rstr A\big\|=\infty$ (cf. Definitions \ref{def:anti_nik_alg} and \ref{def:nik_cp}). Note that if a Boolean algebra does not have the Nikodym property, then it carries an anti-Nikodym sequence, and that every anti-Nikodym sequence has a Nikodym concentration point (Lemma \ref{lemma:Ncp_exists}). %Those two notions were already studied e.g. in \cite{Aiz92}, \cite{Sob19}, and \cite{SZ17}, and appeared extremely useful. 
In Section \ref{sec:anti} we will investigate anti-Nikodym sequences of measures having \textit{strong} Nikodym concentration points, that is, Nikodym concentration points satisfying a slightly technical yet natural condition described in Definition \ref{def:strong_nik_cp}. In particular, %we observe in Lemma \ref{lemma:one_Ncp_sNcp} that if an anti-Nikodym sequence has a unique Nikodym concentration point, then it must be a strong Nikodym concentration point, as well as 
we prove in Theorem \ref{thm:sNcp_disj_supp} that if an anti-Nikodym sequence of measures on a Boolean algebra $\aA$ has a strong Nikodym concentration point, then $\aA$ carries an anti-Nikodym sequence of measures with pairwise disjoint supports (or even with supports contained in pairwise disjoint clopen subsets of $St(\aA)$), which, as one will see below, will appear to be an extremely useful feature.

%For a Boolean algebra $\aA$ and an ideal $\iI$ in $\aA$, let $\aA\gen{\iI}$ be the Boolean subalgebra of $\aA$ generated by $\iI$; it follows that $\aA\gen{\iI}=\iI\cup\iI^*$, where $\iI^*$ denotes the filter in $\aA$ dual to $\iI$. Consequently, each algebra in the families $\hH_1$ and $\hH_2$ from Theorems \ref{thm:main_large_borel} and \ref{thm:main_large_nonborel} is of the form $\wo\gen{\iI}$ for some ideal $\iI$ on $\omega$. 
In the second step of the proofs of Theorems \ref{thm:main_large_borel} and \ref{thm:main_large_nonborel} we prove in Proposition \ref{prop:ai_uniq_concentr_point} that if a Boolean algebra $\aA$ has the Nikodym property but its subalgebra $\aA\gen{\iI}$ generated by $\iI$ (i.e. $\aA\gen{\iI}=\iI\cup\iI^*$) does not, then every anti-Nikodym sequence of measures on $\aA\gen{\iI}$ has a unique strong Nikodym concentration point, namely, the point $p_{\iI^*}\in St(\aA\gen{\iI})$ corresponding to the filter $\iI^*$ (which is an ultrafilter in $\aA\gen{\iI}$). %, which then consequently must be a strong Nikodym concentration point.
This, together with aforementioned Theorem \ref{thm:sNcp_disj_supp}, enables us to characterize the Nikodym property for Boolean algebras of the form $\aA\gen{\iI}$ in terms of non-negative measures with pairwise disjoint supports---see Theorem \ref{thm:ai_nikodym_char_pos} for a general statement, below we only provide its variant for a free filter $\fF$ on $\omega$ and the algebra $\aA_\fF=\wo\gen{\fF}=\fF\cup\fF^*$ (cf. also Corollary \ref{cor:sf_an_char_pos}).

\begin{theoreml}\label{thm:main_sf_an_char_pos}
Let $\fF$ be a free filter on $\omega$. Then, the Boolean algebra $\aA_{\fF}=\fF\cup\fF^*$ has the Nikodym property if and only if there is no disjointly supported sequence $\seqn{\mu_n}$ of non-negative measures on $S_\fF=St\big(\aA_{\fF}\big)$ such that:
\begin{enumerate}
\item $p_{\fF}\not\in\supp\big(\mu_n\big)$ for every $n\io$,
\item $\sup_{n\in\omega}\mu_n\big(S_\fF\big)=\infty$,
\item $\lim_{n\to\infty}\mu_n\big(S_\fF\sm U\big)=0$ for every clopen neighborhood $U$ of $p_\fF$.
\end{enumerate}
\end{theoreml} 

For every free filter $\fF$ on $\omega$, the set of isolated points in the Stone space $St\big(\aA_\fF\big)$ of the Boolean algebra $\aA_\fF$ can be naturally associated with the (discrete) space $\omega$, and so we can put $N_\fF=\omega\cup\big\{p_\fF\big\}$, where again the point $p_\fF$ corresponds to the ultrafilter $\fF$ in $\aA_\fF$. It is easy to see that we have $St\big(\aA_\fF\big)=St\big(\aA_\fF/Fin\big)\cup N_\fF$ and $St\big(\aA_\fF/Fin\big)\cap N_\fF=\big\{p_\fF\big\}$, where $\aA_\fF/Fin$ denotes the quotient Boolean algebra of $\aA_\fF$ modulo finite subsets of $\omega$ (see \cite[Section 3]{MS24} for details). The third step in proving Theorems \ref{thm:main_large_borel} and \ref{thm:main_large_nonborel} thus relies on decomposing the Nikodym property of Boolean algebras of the form $\aA_\fF$ into two parts: the Nikodym property of the quotient Boolean algebras $\aA_\fF/Fin$ and the finitely supported Nikodym property of the spaces $N_\fF$, defined as follows.%, introduced in \cite{Zuc25} and defined as follows.

\begin{definition}\label{def:fsNp}
A zero-dimensional topological space $X$ has the \textit{finitely supported Nikodym property} if every sequence $\seqn{\mu_n}$ of Borel real-valued measures on $X$ with finite supports and such that $\lim_{n\to\infty}\mu_n(A)=0$ for every clopen subset $A\sub X$ is norm bounded.
\end{definition}
%
%\begin{definition}
%Let $\fF$ be a free filter on $\omega$. Endow the set $N_{\fF} =\omega\cup\big\{p_{\fF}\big\}$, where $p_{\fF}$ is a fixed point not belonging to $\omega$, with the topology $\tau_{\fF}$ defined in the following way:
%\begin{itemize}
%\item every point of $\omega$ is isolated in $N_{\fF}$ , i.e. $\{n\}$ is open for every $n\io$,
%\item a set $U\sub N_{\fF}$ is an open neighborhood of $p_{\fF}$ if and only if $U=A\cup\big\{p_{\fF}\big\}$ for some $A\in\fF$.
%\end{itemize}
%\end{definition}
%
%Note that for every free filter $\fF$ the space $N_{\fF}$ (with the topology inherited from $St\big(\aA_\fF\big)$) is a countable non-discrete zero-dimensional space.

The announced decomposition of the Nikodym property for Boolean algebras $\aA_\fF$ reads then precisely as follows (for its proof, using Theorem \ref{thm:main_sf_an_char_pos}, see Theorem \ref{thm:sf_nik_nf_sfs}).

\begin{theoreml}\label{thm:main_sf_nik_nf_sfs}
Let $\fF$ be a free filter on $\omega$. Then, the algebra $\aA_{\fF}$ has the Nikodym property if and only if the algebra $\aA_{\fF}/Fin$ has the Nikodym property and the space $N_{\fF}$ has the finitely supported Nikodym property.
\end{theoreml}
%
%For the proof of Theorem \ref{thm:main_sf_nik_nf_sfs}, see Theorem \ref{thm:sf_nik_nf_sfs}

The finitely supported Nikodym property of spaces of the form $N_\fF$ was thorougly studied by the second author in \cite{Zuc25}, especially in the context of the Kat\v{e}tov ordering of ideals and so-called non-pathological ideals on $\omega$ (see Section \ref{sec:nonpath} for definitions of various classes of ideals mentioned below). In particular, it was proved there that, for a density ideal $\iI$ on $\omega$, the space $N_{\iI^*}$ has the finitely supported Nikodym property if and only if $\iI$ is isomorphic to an \erdos--Ulam ideal (cf. Theorem \ref{thm:zuch}.(3)). Consequently, using Theorem \ref{thm:main_sf_nik_nf_sfs} and Theorem \ref{thm:p_ideal_nik_equivalences}, the latter asserting that for a P-filter $\fF$ the Nikodym property of the algebra $\aA_\fF$ is equivalent to the finitely supported Nikodym property of the space $N_\fF$, in Theorem \ref{thm:dens_nik_erdos_ulam} we show that for every density ideal $\iI$ the Boolean algebra $\aA_{\iI^*}$ has the Nikodym property if and only if $\iI$ is isomorphic to an \erdos--Ulam ideal. We then use this result, together with the facts that each \erdos--Ulam ideal is an $\mathbb{F}_{\sigma\delta}$ subset of $\Cantor$ and that there are $\frakc$ many pairwise non-isomorphic \erdos--Ulam ideals on $\omega$, to finally obtain the family $\hH_1$ of Boolean algebras in Theorem \ref{thm:main_large_borel}; see the proof of Theorem \ref{thm:continuum} for details. 
%this leads us to Theorem \ref{thm:main_large_borel} (cf. the proof of Theorem \ref{thm:continuum}). 

The final argument for Theorem \ref{thm:main_large_nonborel} is a bit different---instead of using \erdos--Ulam ideals, we consider a family of $2^\frakc$ many pairwise non-isomorphic non-analytic ideals of the form $\uU^*\oplus\zZ$, where each $\uU$ is a free ultrafilter on $\omega$ and $\zZ$ is the standard asymptotic density zero ideal, see the proof of Theorem \ref{thm:two_to_continuum}.

The above discussion has so far concerned only the Nikodym property of Boolean algebras of the form $\aA_\fF$. To show that the algebras $\aA_\fF$ used to establish Theorems \ref{thm:main_large_borel} and \ref{thm:main_large_nonborel} do not have the Grothendieck property, we exploit the fact that their spaces $N_\fF$ have the following property.

\begin{definition}\label{def:bjnp}
A topological space $X$ has the \textit{bounded Josefson--Nissenzweig property} (in short, \textit{(BJNP)}), if there is a sequence $\seqn{\mu_n}$ of Borel real-valued measures on $X$ with finite supports and norm $1$ and such that $\lim_{n\to\infty}\int_{X}fd\mu_n=0$ for every bounded continuous function $f\colon X\to\mathbb{R}$. 
\end{definition}

The bounded Josefson--Nissenzweig property was introduced by K\k{a}kol \textit{et al.} \cite{KMSZ} in order to investigate the celebrated Separable Quotient Problem in the class of topological vector spaces $C_p(X)$ of continuous real-valued functions on Tychonoff spaces $X$ endowed with the pointwise topology and is, of course, closely related to the classical Josefson--Nissenzweig theorem from Banach space theory (see \cite[Chapter XII]{Die84}). Based on the work \cite{KSZ23}, Marciszewski and the first author showed in \cite{MS24} that if the Stone space $St(\aA)$ of a Boolean algebra $\aA$ contains a homeomorphic copy of a space $N_\fF$ which has (BJNP), then $\aA$ does not have the Grothendieck property, as well as that for an ideal $\iI$ on $\omega$ the space $N_{\iI^*}$ has (BJNP) if and only if $\iI\le_K\zZ$, i.e. $\iI$ is Kat\v{e}tov below the aforementioned asymptotic density ideal $\zZ$. Consequently, since both \erdos--Ulam ideals as well as ideals of the form $\uU^*\oplus\zZ$ are all Kat\v{e}tov below $\zZ$, we obtain that no algebra $\aA_\fF$ from the families $\hH_1$ and $\hH_2$ can have the Grothendieck property.

\medskip

Motivated by the above results, in the context of the open problem of the existence in \textsf{ZFC} of a Boolean algebra with the Grothendieck property but without the Nikodym property, one might naturally look for such an example among the algebras of the form $\aA_\fF$. This would however appear pointless, as there is no Boolean algebra $\aA_\fF$ with the Grothendieck property and without the Nikodym property.

\begin{theoreml}\label{thm:main_af_no_nik_no_gr}
Let $\fF$ be a free filter on $\omega$. If the algebra $\aA_{\fF}$ does not have the Nikodym property, then it does not have the Grothendieck property either.
\end{theoreml}

To establish the above result, we again appeal to the notion of strong Nikodym concentration points. Namely, in Section \ref{sec:no_gr} we show that the existence of an anti-Nikodym sequence of measures on a Boolean algebra $\aA$ with a strong Nikodym concentration point causes that $\aA$ in fact has neither the Nikodym property nor the Grothendieck property.

\begin{theoreml}\label{thm:main_no_nik_no_gr}
    If a Boolean algebra $\aA$ does not have the Nikodym property and this is witnessed by an anti-Nikodym sequence with a strong Nikodym concentration point in the Stone space $St(\aA)$, then $\aA$ does not have the Grothendieck property.
\end{theoreml}

An immediate important corollary of Theorem \ref{thm:main_no_nik_no_gr} is that if a Boolean algebra $\aA$ carries an anti-Nikodym sequence of measures having a unique Nikodym concentration point, then $\aA$ does not have the Grothendieck property (Corollary \ref{cor:one_Ncp_gr}). This result may be seen as a variant of the above-mentioned standard fact asserting that if the Stone space $St(\aA)$ of a Boolean algebra $\aA$ contains a non-trivial convergent sequence, then $\aA$ has neither the Nikodym property, nor the Grothendieck property, cf. Example \ref{example:strong_conv_seq}. Consequently, as shown in Proposition \ref{prop:ai_uniq_concentr_point}, if a Boolean algebra $\aA$ has the Nikodym property and $\iI$ is an ideal in $\aA$ such that the subalgebra $\aA\gen{\iI}$ of $\aA$ generated by $\iI$ does not have the Nikodym property, then $\aA\gen{\iI}$ does not have the Grothendieck property either (cf. also Corollary \ref{cor:sigmacompl_ideal}). Theorem \ref{thm:main_af_no_nik_no_gr} is then a special case of the latter statement (for $\aA=\wo$ and $\iI=\fF^*$).

\medskip

The above-presented main theorems have mostly regarded the Nikodym and Grothendieck properties of Boolean algebras \textit{generated by} ideals. We will now focus on the Nikodym property of ideals \textit{themselves}; the property is defined in a similar way to that of Boolean algebras, see Definitions \ref{def:nikodym_ring} and \ref{def:nik_ideal}. The research on this variant of the Nikodym property was initiated by Drewnowski, Florencio, and Pa\'ul \cite{DFP94} in the context of studying properties of barrelled topological vector spaces, with the remarkable result asserting that the ideal $\zZ$ has the Nikodym property, and later continued, e.g., in  \cite{ADL09}, \cite{BBL02}, \cite{BL13}, \cite{DFP96}, \cite{DP00}, \cite{Fer02}, \cite{FLALP}, \cite{SA04}, \cite{Stu07}.

The Nikodym properties of Boolean algebras and ideals may initially appear somehow unrelated, but, as first observed by Drewnowski and Pa\'ul \cite[Remark 2.2.(a)]{DP00}, they are in fact quite the same: an ideal $\iI$ in a Boolean algebra $\aA$ has the Nikodym property if and only if the subalgebra $\aA\gen{\iI}$ of $\aA$ has the Nikodym property (cf. Theorem \ref{thm:nik_prop_ideal_algebra}). In Section \ref{sec:extending} we analyse this relation in more detail, primarily by decomposing measures on the Stone space $St(\aA\gen{\iI})$ of a Boolean algebra $\aA\gen{\iI}$ into two parts, defined on the open subset of $St(\aA\gen{\iI})$ corresponding to the ideal $\iI$ and on the singleton $\big\{p_{\iI^*}\big\}$ corresponding to the ultrafilter $\iI^*$ in $\aA\gen{\iI}$, respectively. Based on this decomposition, e.g., in Theorem \ref{thm:isomorphism} we show that the Banach space $\ba(\iI)$ of all finitely additive signed measures on the ideal $\iI$ is in a natural way isomorphic to a closed linear subspace of the Banach space $\ba(\aA\gen{\iI})$ of codimension $1$.

Let us point out that the above-described topological approach to the study of relations between the two Nikodym properties has several important advantages. First of all, it provides better insight into the general relation between ideals and the generated Boolean algebras as well as into the behavior of sequences of measures on the Stone spaces of the latter algebras. Second, it brings natural tools for merging and unifying two, so far essentially separated, lines of research on the Nikodym property: one concerning rings and one regarding Boolean algebras. Third, it creates a natural source of examples for both of the lines, see for instance Theorems \ref{thm:main_large_borel} and \ref{thm:main_large_nonborel} as well as numerous examples and constructions presented in Sections \ref{sec:definable} and \ref{sec:nonatomic}.

\medskip

Let us discuss the content of Section \ref{sec:nonatomic} in more detail. Its main point is to study relations between various notions of boundedness and non-atomicity of ideals and submeasures on $\omega$ which are closely connected to the Nikodym property. These are, in particular, the total boundedness, the Local-to-Global Boundedness Property, the Bolzano--Weierstrass Property (BWP), the (Strong) Nested Partition Property ((S)NPP), strong non-atomicity of submeasures, and the existence of countable splitting families in quotients $\wo/\iI$. Due to the numerosity of the properties, we simply refer the reader to Sections \ref{sec:definable} and \ref{sec:nonatomic} for respective definitions, here we only provide two of them: an ideal $\iI$ on $\omega$ has the \textit{Local-to-Global Boundedness Property for submeasures} (in short, \textit{(LGBPs)}), if whenever $\varphi$ is a submeasure on $\omega$ such that $\iI\sub\fin(\varphi)$, then $\sup_{A\in\iI}\varphi(A)<\infty$; and an analytic P-ideal $\iI$ on $\omega$ is \textit{totally bounded}, if whenever $\varphi$ is an lsc submeasure on $\omega$ such that $\iI\sub\exh(\varphi)$, then $\varphi(\omega)<\infty$ (cf. Proposition \ref{prop:tot_bnded_equiv_def}). The first property was defined by Drewnowski, Florencio, and Pa\'ul \cite{DFP96}, who showed that it implies the Nikodym property of ideals on $\omega$ (and, in fact, in a sense it may be considered as an analogon of the Nikodym property for a single submeasure). The total boundedness was introduced by Hern\'andez-Hern\'andez and Hru\v{s}\'ak \cite{HHH07} as an auxiliary notion in the study of destructibility of ideals on $\omega$ by forcing. It was shown in the latter work that, for an analytic P-ideal $\iI$, the existence of a countable splitting family in the quotient Boolean algebra $\wo/\iI$ implies that $\iI$ is totally bounded. In Corollaries \ref{cor:tot_bnded_full_char} and \ref{cor:tot_bnded_full_char2} we show that the reverse implication also holds as well as that the total boundedness is equivalent to several further properties. More precisely, supplementing and combining results of Filip\'ow \textit{et al.} \cite{BFMS13} and \cite{FMRS07}, Drewnowski and \L uczak \cite{DL08ii}, and Stuart \cite{Stu07}, we get that for every lsc submeasure $\varphi$ on $\omega$ and the ideal $\iI=\exh(\varphi)$ the following statements are equivalent:
\vspace{1mm}
\begin{addmargin}[1em]{2em}% 2em left, 2em right
\begin{tabular}{l p{0.5\linewidth}}
(1) $\iI$ is totally bounded, & (2) $\iI$ has (LGBPs),\\
(3) $\iI$ cannot be extended to an $\F_\sigma$-ideal, & (4) $\iI$ does not have (BWP),\\
(5) $\wo/\iI$ has a countable splitting family, & (6) $\conv\le_K\iI$,\\
(7) $\varphi^\bullet$ is strongly non-atomic, & (8) $\iI$ is non-atomic,\\
\multicolumn{2}{l}{(9) $\iI$ has (SNPP).}\\
\end{tabular}
\end{addmargin}
\vspace{1mm}
(Here, $\conv$ denotes the ideal on $\Q\cap[0,1]$ generated by convergent sequences and $\varphi^\bullet$ is the core of the submeasure $\varphi$, see Section \ref{sec:nonpath}.)

As mentioned above, it was proved in paper \cite{DFP96} that (LGBPs) implies the Nikodym property, but its authors conjectured that the converse does not hold (\cite[Conjecture on page 147]{DFP96}). Confirming their guesswork, we prove in Theorem \ref{thm:nikodym_lgbps} that several ideals on $\omega$ already described in the literature indeed have the Nikodym property but not (LGBPs). Moreover, expanding the construction of analytic P-ideals on $\omega$ due to Alon, Drewnowski, and \L uczak \cite{ADL09}, based on  sequences of finite Kneser hypergraphs, in Section \ref{sec:hypergraph} we introduce a new class of ideals called \textit{hypergraph} (see Definition \ref{def:hypergraph}) and obtain the following quantitative result.

\begin{theoreml}\label{thm:main_hypergraphs}
    There is a family of $\frakc$ many pairwise non-isomorphic hypergraph ideals which have the Nikodym property but not the Local-to-Global Boundedness Property for submeasures.   
\end{theoreml}

The above theorem is a reformulation of Theorem \ref{thm:hypergraph_nikodym_not_tot_bnd}. Let us note that the hypergraph ideals constructed in its proof are induced by sequences of hypergraphs with increasing chromatic numbers, therefore by Proposition \ref{prop:hypergraph_npp} they cannot have the Nested Partition Property. This fact yields a negative answer to an (implicit) question of Stuart \cite[page 153]{Stu07}, asking whether the Nikodym property of a ring of sets implies the Nested Partition Property, see Corollary \ref{cor:nikodym_no_npp}. 

Our results also have some applications to the study of two stronger variants of the Nikodym property, named the \textit{strong Nikodym property} and the \textit{web Nikodym property}, see Definitions \ref{def:strong_web_nik} and \ref{def:strong_web_nik_alg}, introduced by Valdivia \cite{Val79} and L\'opez-Pellicer \cite{LP97}, respectively, who proved that $\sigma$-fields of sets have both the properties (see Section \ref{sec:definable} for more information). The web Nikodym property for ideals was first studied by Ferrando \cite{Fer02}, who proved that the ideal $\zZ$ has the web Nikodym property, and later by Ferrando, L\'opez-Alfonso, and L\'opez-Pellicer \cite{FLALP}. Exploiting the latter work we show in Theorem \ref{thm:p_ideal_nik_equivalences} that for a P-ideal $\iI$ on $\omega$ the following conditions are equivalent:
\vspace{1mm}
\begin{addmargin}[1em]{2em}% 2em left, 2em right
\begin{tabular}{l p{0.5\linewidth}}
(1) $\aA_{\iI^*}$ has the Nikodym property, & (2) $\iI$ has the Nikodym property,\\
(3) $\aA_{\iI^*}$ has the strong Nikodym property, & (4) $\iI$ has the strong Nikodym property,\\
(5) $\aA_{\iI^*}$ has the web Nikodym property, & (6) $\iI$ has the web Nikodym property,\\
\multicolumn{2}{l}{(7) $N_{\iI^*}$ has the finitely supported Nikodym property.}\\
\end{tabular}
\end{addmargin}
\vspace{1mm}
%\begin{enumerate}[(i)]
%    \item $\aA_{\iI^*}$ has the Nikodym property,
%    \item $\aA_{\iI^*}$ has the strong Nikodym property,
%    \item $\aA_{\iI^*}$ has the web Nikodym property,
%    \item $\iI$ has the Nikodym property,
%    \item $\iI$ has the strong Nikodym property,
%    \item $\iI$ has the web Nikodym property,
%    \item $N_{\iI^*}$ has the finitely supported Nikodym property.
%\end{enumerate}
Moreover, by using the results of \cite{Zuc25}, we observe in Corollaries \ref{cor:density_nik_strong_web} and \ref{cor:density_nik_strong_web_totally bounded} that for a density ideal $\iI$ the following conditions are additionally equivalent:
\begin{enumerate}[(a)]
    \item $\iI$ has the (strong/web) Nikodym property,
    \item $\iI$ is isomorphic to an \erdos--Ulam ideal,
    \item $\iI$ is totally bounded.
\end{enumerate}
Some new examples of ideals with the strong and web Nikodym properties are then presented in Examples \ref{example:nonppoint_web_nikodym} and \ref{example:new_web_nikodym}, and in Remarks \ref{rem:c_many_density} and \ref{rem:hypergraph_web_nikodym}. 

\medskip

To discuss the results presented in Section \ref{sec:tukey} let us introduce the following notation, which will be used throughout the paper.

\begin{definition}\label{def:class_an}
    The class of all ideals $\iI$ on $\omega$ for which the space $N_{\iI^*}$ does not have the finitely supported Nikodym property is denoted by $\aA\nN$.
\end{definition}

%In Section \ref{sec:tukey} we study the Tukey complexity of the ordered set $(\aA\nN,\le_K)$, that is, the set $\aA\nN$ endowed with the Kat\v{e}tov ordering $\le_K$. 
In \cite{Zuc25} the second author studied the cofinal structure of the ordered set $(\aA\nN,\le_K)$, that is, the set $\aA\nN$ endowed with the Kat\v{e}tov ordering $\le_K$. In particular, he proved that $(\aA\nN,\le_K)$ does not have any maximal elements as well as that it has cofinal families of size $\frakd$, i.e. the cardinal characteristic of the continuum called the \textit{dominating number}, consisting only of density ideals. Moreover, it was showed there that the ordered set $(\aA\nN,\le_K)$ is Tukey reducible to the standard poset $(\omega^\omega,\le^*)$ and asked whether a converse Tukey reduction also exists (see Section \ref{sec:tukey} for relevant definitions). In Section \ref{sec:tukey} we observe, based on a simple measure-theoretic argument, that this is indeed the case.

\begin{theoreml}\label{thm:main_tukey}
    The ordered sets $(\aA\nN,\le_K)$ and $(\omega^\omega,\le^*)$ are Tukey equivalent.
\end{theoreml}

Theorem \ref{thm:main_tukey} supplements the results of Minami and Sakai \cite{MS16} and He, Li, and Zhang \cite{HLZ}, who proved that the classes of all $\mathbb{F}_\sigma$-ideals on $\omega$ and of all summable ideals on $\omega$, both endowed with $\le_K$, are Tukey equivalent to the poset $(\omega^\omega,\le^*)$. In fact, the most important ingredient of the proof of Theorem \ref{thm:main_tukey} is showing that every ideal in $\mathcal{AN}$ is contained in a summable ideal, which, as explained in Section \ref{sec:tukey}, greatly generalizes the theorem of Drewnowski and Pa\'ul \cite{DP00} asserting that in the class of P-ideals the Nikodym property coincides with the so-called Positive Summability Property (see Definition \ref{def:psp_asp}).

Let us note that Theorem \ref{thm:main_tukey} corresponds to the intuitive ``unbounded'' character of the Nikodym property and clearly contrasts the ``bounded'' character of the Grothendieck property, here expressed by the above-mentioned result of Marciszewski and the first author \cite{MS24} asserting that, for an ideal $\iI$ on $\omega$, the space $N_{\iI^*}$ has the bounded Josefson--Nissenzweig property (so, in a sense, does not have the \textit{finitely supported} Grothendieck property; cf. \cite{KSZ23}) if and only if $\iI\le_K\zZ$.

\medskip

The paper is organized as follows. In the next section we recall basic notions and notation concerning measures on rings, Boolean algebras, ideals, and compact spaces. In Section \ref{sec:anti} we study sequences of measures on Boolean algebras which witness simultaneously the lack of the Nikodym property and of the Grothendieck property. Section \ref{sec:ideals} is devoted to relations between the Nikodym properties of ideals and of Boolean algebras generated by ideals. In Section \ref{sec:definable} we investigate the Nikodym property of Boolean subalgebras of $\wo$ generated by ideals on $\omega$. In Section \ref{sec:nonatomic} relations between various notions of boundedness and non-atomicity of ideals on $\omega$ and the Nikodym property are studied. Section \ref{sec:tukey} is devoted to the description of the Tukey type of the ordered set $(\aA\nN,\le_K)$. In the final section of the paper, Section \ref{sec:questions}, we pose several open questions.

\section*{Acknowledgements}

The authors would like to thank Lyubomyr Zdomskyy for numerous valuable and inspiring discussions concerning the topic of this work.

\section{Preliminaries\label{sec:prelim}}

In this section we briefly describe the notation and nomenclature used in the paper.

\medskip

By $\omega$ we denote the first infinite cardinal number, i.e. the cardinality of the set $\N$ of all non-negative integers, and we write $\frakc$ for the \textit{continuum}, i.e. the cardinality of the real line $\R$. As usual, we will simply identify $\omega$ with the set $\N$. 

We will also identify each $n\io$ with the set $\{0,1,\ldots,n-1\}$. For each $k\le n\io$ we put $[k,n]=\{k,k+1,\ldots,n\}$.

The set of all functions $\omega\to\omega$ is denoted by $\omega^\omega$. For every $f,g\in\omega^\omega$ we write $f\le^*g$ if and only if there is $N\io$ such that $f(n)\le g(n)$ for every $n\ge N$.

Let $X$ be a set. By $\wp(X)$ we denote the family of all subsets of $X$. $|X|$ stands for the cardinality of $X$. If $\kappa$ is a cardinal number, then by $[X]^{<\kappa}$ and $[X]^{\kappa}$ we denote the families of all subsets of $X$ of cardinality strictly less than $\kappa$ and of all subsets of $X$ of cardinality equal to $\kappa$, respectively. By $X^{<\omega}$ we denote the set consisting of all finite sequences of elements of $X$, i.e. $X^{<\omega}=\bigcup_{n\io}X^n$. We also shortly write $Fin=\fso$.

\medskip

Let $\aA$ be a Boolean algebra (with the operations $\wedge$, $\vee$, $\setminus$, $^c$, zero element $0_\aA$, and unit element $1_\aA$). For a subset $\xX\sub\aA$ by $\aA\gen{\xX}$ we denote the Boolean subalgebra of $\aA$ generated by $\xX$; we also set $\xX^*=\big\{1_\aA\sm A\colon A\in\xX\}$ and say that $\xX^*$ is \textit{dual} to $\xX$. For an element $A\in\aA$ we set $\aA_A=\{A\wedge B\colon B\in\aA\}$; of course, $\aA_A$ is also a Boolean algebra (with the obvious operations and constants). A collection $\seq{A_i}{i\in I}$ of elements of $\aA$ is an \textit{antichain} if $A_i\wedge A_j=0_\aA$ for every $i\neq j\in I$. 

Let $\iI$ be an ideal in the Boolean algebra $\aA$. Note that, as $\iI$ is downward closed, for every $A\in\iI$ we have $\aA_A=(\aA\gen{\iI})_A$. 

Let $\xX\in\{\aA,\iI\}$. A function $\mu\colon\xX\to\R$ is a \textit{measure} on $\xX$ if $\mu$ is finitely additive and \textit{bounded}, i.e. for the \textit{variation norm}
\[\|\mu\|=\sup\big\{|\mu(A)|+|\mu(B)|\colon\ A,B\in\xX,\ A\wedge B=0_\aA\big\}\]
we have $\|\mu\|<\infty$. By $\ba(\xX)$ we denote the vector space of all measures on $\xX$ endowed with the variation norm, which makes $\ba(\xX)$ a Banach space.

More generally, for $\mu\in\ba(\xX)$ and $X\in\xX$ we set
\[|\mu|(X)=\sup\big\{|\mu(A)|+|\mu(B)|\colon\ A,B\in\xX,\ A\wedge B=0_\aA,\ A\vee B\le X\big\},\]
so, in the case $\xX=\aA$, we have $\|\mu\|=|\mu|\big(1_\aA\big)$.

For a measure $\mu\in\ba(\xX)$, an element $A\in\xX$, and a subset $\yY\sub\xX$, we define the \textit{restricted} measures $\mu\rstr A\in\ba(\xX)$ and $\mu\rstr\yY\in\ba(\yY)$ as follows: $(\mu\rstr A)(B)=\mu(A\wedge B)$ for every $B\in\xX$, and $(\mu\rstr\yY)(B)=\mu(B)$ for every $B\in\yY$. Note that $\|\mu\rstr A\|=\|\mu\rstr\aA_A\|$ for every $A\in\aA$. 

For measures $\mu\in\ba(\aA)$ and $\nu\in\ba(\iI)$ we will say that $\mu$ \textit{extends} $\nu$ if $\mu\rstr\iI=\nu$. Of course, since $\iI$ is downward closed, we have $|\nu|(A)=|\mu|(A)$ for every $A\in\iI$ and so $\|\nu\|\le\|\mu\|$.

\medskip

All topological spaces considered in this paper are assumed to be \textit{Tychonoff}. Let $X$ be a topological space. By $Bor(X)$ we denote the $\sigma$-field of all Borel subsets of $X$ and by $Clopen(X)$ the field of all clopen subsets of $X$. $C(X)$ as usual denotes the set of all continuous real-valued functions on $X$ and by $C_b(X)$ we denote its subset consisting of all bounded functions. If $X$ is a compact space, then we endow $C(X)$ with the supremum norm, which makes it a Banach space. Recall that a subspace $Y$ of $X$ is \textit{C*-embedded} if every $f\in C_b(Y)$ extends to some $f'\in C_b(X)$.

A point $x$ in $X$ is a \textit{P-point} if, for every sequence $\seqn{U_n}$ of neighborhoods of $x$ in $X$, the interior of the intersection $\bigcap_{n\io}U_n$ contains $x$.% the intersection $\bigcap_{n\io} U_n$ has non-empty interior in $X$. 

By $\beta X$ we denote the \v{C}ech--Stone compactification of $X$, i.e. a unique (up to equivalence) compactification $K$ of $X$ satisfying the following property: for every $A,B\sub K$ such that $A\sub f^{-1}(0)$ and $B\sub f^{-1}(1)$ for some $f\in C(X)$ we have $\ol{A}^K\cap\ol{B}^K=\emptyset$ (see \cite[page 167 and Corollary 3.6.2]{Eng89}). For $X=\omega$ we additionally set $\omega^*=(\beta\omega)\sm\omega$.

A function $\mu\colon Bor(X)\to\R$ is a \textit{measure} on $X$ if $\mu$ is $\sigma$-additive, regular, and \textit{bounded}, i.e. for the \textit{variation}
\[|\mu|(Y)=\sup\big\{|\mu(A)|+|\mu(B)|\colon\ A,B\in Bor(X),\ A\cap B=\emptyset,\ A\cup B\sub Y\big\},\]
where $Y\in Bor(K)$, we have $\|\mu\|=|\mu|(X)<\infty$. The vector space of all measures on $X$ endowed with the variation norm will be denoted by $M(X)$.

For a point $x\in X$ by $\delta_x$ we denote the one-point measure concentrated at $x$. A measure $\mu\in M(X)$ is said to be \textit{finitely supported} if there are a finite subset $F\sub X$ and non-zero real numbers $\alpha_x$, $x\in F$, such that $\mu=\sum_{x\in F}\alpha_x\delta_x$; note that in this case we have $\|\mu\|=\sum_{x\in F}\big|\alpha_x\big|$.

Let $K$ be a compact space. Notice that $M(K)$ is a Banach space. Recall that by the classical Riesz--Markov--Kakutani Representation Theorem $M(K)$ is isometrically isomorphic to the dual Banach space $C(K)^*$. 

For each $\mu\in M(K)$ by $\supp(\mu)$ we denote the \textit{support} of $\mu$, i.e. the smallest closed subset $F$ of $K$ such that $|\mu|(F)=\|\mu\|$. Note that if $\mu$ is finitely supported, then $\supp(\mu)=F$, where $F$ is as in the above definition of a finitely supported measure. A sequence $\seqn{\mu_n}$ of measures on $K$ is \textit{disjointly supported} if $\supp\big(\mu_n\big)\cap\supp\big(\mu_{n'}\big)=\emptyset$ for every $n\neq n'\io$.

For a Boolean algebra $\aA$, by $St(\aA)$ we denote the Stone space of $\aA$, i.e. the compact (Hausdorff) space of all ultrafilters on $\aA$. For every $A\in\aA$ the corresponding clopen subset of $St(\aA)$ is denoted by $\clopen{A}_\aA$. Recall that the mapping $\aA\ni A\mapsto\clopen{A}_\aA\in Clopen(St(\aA))$ is an isomorphism between the Boolean algebras $\aA$ and $Clopen(St(\aA))$. Note that for every $\mu\in\ba(\aA)$ there is a unique measure $\wh{\mu}\in M(St(\aA))$ such that $\mu(A)=\wh{\mu}\big(\clopen{A}_\aA\big)$ for every $A\in\aA$ and $\|\mu\|=\|\wh{\mu}\|$; we will call $\wh{\mu}$ the \textit{Radon extension} of $\mu$. Similarly, every $\mu\in M(St(\aA))$ induces a measure $\check{\mu}\in\ba(\aA)$ by setting $\check{\mu}(A)=\mu\big(\clopen{A}_\aA\big)$ for every $A\in\aA$; we also have $\|\mu\|=\|\check{\mu}\|$. Consequently, the Banach spaces $\ba(\aA)$ and $M(St(\aA))$ are isometrically isomorphic.

Let us also say that a sequence $\seqn{\mu_n}$ in $\ba(\aA)$ is \textit{disjointly supported} if the sequence $\seqn{\wh{\mu}_n}$ of the respective Radon extensions in $M(St(\aA))$ is disjointly supported.

\medskip

The Boolean algebra $\wo$ and its subalgebras of the form $\wo\gen{\iI}$ for ideals $\iI$ will be of our special interest. For an ideal $\iI$ (resp. a filter $\fF$) in $\wo$, let us briefly say that $\iI$ is an \textit{ideal on $\omega$} (resp. $\fF$ is a \textit{filter on $\omega$}). Similarly, an element $\mu$ of $\ba(\wo)$ is also called a \textit{measure on $\omega$}.

The \textit{Fr\'echet} filter on $\omega$ is $Fr=\{A\sub\omega\colon |\omega\sm A|<\infty\}$; we have $Fr^*=Fin$. A filter $\fF$ on $\omega$ is \textit{free} if $\bigcap\fF=\emptyset$, or, equivalenty, if $Fr\sub\fF$. For a set $A\in\wo$ by $A^\bullet$ we denote the equivalence class of $A$ in the quotient Boolean algebra $\wo/Fin$. Recall that the Stone spaces $St(\wo)$ and $St(\wo/Fin)$ are usually identified with $\bo$ and $\omega^*$, respectively.

A filter $\fF$ on $\omega$ is a \textit{P-filter} if for every sequence $\seqn{A_n}$ of sets from $\fF$ there is $A\in\fF$ such that $A\setminus A_n$ is finite for each $n\io$; the ideal on $\omega$ dual to a P-filter is a \textit{P-ideal}. A \textit{P-point} (on $\omega$) is a P-filter which is maximal in the sense of inclusion; of course, the notion of P-points on $\omega$ coincides with the notion of P-points in the topological space $\omega^*$.

The Kat\v{e}tov order $\le_K$ is an important tool for comparing the structural complexity of ideals on $\omega$, see e.g. \cite{Hru17}. For two ideals $\iI$ and $\jJ$ on $\omega$, we say that $\iI$ is \textit{Kat\v{e}tov below} $\jJ$, denoted $\iI\le_K\jJ$, if there is a function $f\colon\omega\to\omega$ such that $f^{-1}[A]\in\jJ$ for all $A\in\iI$; if it is not true that $\iI\le_K\jJ$ or $\jJ\le_K\iI$, then we say that $\iI$ and $\jJ$ are \textit{Kat\v{e}tov incomparable}. Also, two ideals $\iI$ and $\jJ$ on $\omega$ are \textit{isomorphic} if there is a bijection $f\colon\omega\to\omega$ such that for every $A\in\wo$ we have $A\in\iI$ if and only if $f[A]\in\jJ$, and they are \textit{Kat\v{e}tov equivalent}, denoted $\iI\equiv_K\jJ$, if $\iI\le_K\jJ$ and $\jJ\le_K\iI$. Of course, those definitions translate in the natural way via complements to filters on $\omega$.

Recall that every subset of $\omega$ can be naturally identified with a point in the Cantor space $\Cantor$ and hence subsets of $\wo$ correspond in a natural way to subsets of $\Cantor$. Consequently, when considering subsets of $\wo$ we may investigate their Borel complexity, analyticity, measurability (in the sense of the product measure on $\Cantor$), etc.

\medskip

Let $\rR$ be a ring of subsets of a set $S$. A function $\mu\colon\rR\to\R$ is a \textit{measure} on $\rR$ if $\mu$ is additive and \textit{bounded}, i.e. for the \textit{variation norm}
\[\|\mu\|=\sup\big\{|\mu(A)|+|\mu(B)|\colon\ A,B\in\rR,\ A\cap B=\emptyset\big\}\]
we have $\|\mu\|<\infty$. By $\ba(\rR)$ we denote the Banach space of all measures on $\rR$ endowed with the variation norm.

For an ideal $\iI$ in a Boolean algebra $\aA$ set $\rR(\iI)=\big\{\clopen{A}_\aA\colon A\in\iI\big\}$ and note that $\rR(\iI)$ is a ring of clopen subsets of $St(\aA)$. Similarly as above, every $\mu\in\ba(\iI)$ induces a measure $\mathring{\mu}\in\ba(\rR(\iI))$ with $\|\mu\|=\|\mathring{\mu}\|$ by setting $\mathring{\mu}\big(\clopen{A}_\aA\big)=\mu(A)$ for every $A\in\iI$, and every $\mu\in\ba(\rR(\iI))$ induces a measure $\tilde{\mu}\in\ba(\iI)$ with $\|\mu\|=\|\tilde{\mu}\|$ by setting $\tilde{\mu}(A)=\mu\big(\clopen{A}_\aA\big)$ for every $A\in\iI$. Consequently, the Banach spaces $\ba(\iI)$ and $\ba(\rR(\iI))$ are isometrically isomorphic.

\medskip

Recall that a subset $A$ of a Banach space $(X,\|\cdot\|)$ is \textit{norm bounded} if $\sup_{x\in A}\|x\|<\infty$. Given two Banach spaces $(E,\|\cdot\|_E)$ and $(F,\|\cdot\|_F)$, we will use the following standard norms on the direct sum $E\oplus F$: $\|(x,y)\|_1=\|x\|_E+\|y\|_F$ and $\|(x,y)\|_\infty=\max(\|x\|_E,\|y\|_F)$ (for $x\in E$, $y\in F$). We assume by default that $E\oplus F$ is endowed with the norm $\|\cdot\|_1$. The norm $\|T\|$ of a linear operator $T\colon E\to F$ is defined as $\|T\|=\sup\big\{\|T(x)\|_F\colon x\in E, \|x\|_E\le 1\big\}$, and we say that $T$ is \textit{bounded} if $\|T\|<\infty$.

\medskip

Let $\rR$ be a ring of subsets of a set $S$, $\aA$ a Boolean algebra, and $\iI$ an ideal in $\aA$. Let $\xX\in\{\rR,\aA,\iI\}$. A sequence $\seqn{\mu_n}$ of measures in $\ba(\xX)$ is \textit{pointwise convergent} to some measure $\mu\in\ba(\xX)$ if $\lim_{n\to\infty}\mu_n(A)=\mu(A)$ for every $A\in\xX$, and $\seqn{\mu_n}$ is \textit{weakly convergent} to $\mu$ if it is convergent in the weak topology $\sigma(\ba(\xX),\ba(\xX)^*)$ on $\ba(\xX)$. Moreover, a sequence $\seqn{\mu_n}$ of measures in $\ba(\aA)$ is \textit{weak* convergent} to $\mu\in\ba(\aA)$ if the sequence $\seqn{\wh{\mu}_n}$ is weak* convergent to $\wh{\mu}$ in the weak* topology $\sigma\big(M(St(\aA)),C(St(\aA))\big)$ on $M(St(\aA))$. Note that, by the Stone--Weierstrass theorem, a sequence $\seqn{\mu_n}$ in $\ba(\aA)$ is weak* convergent to $\mu\in\ba(\aA)$ if and only if $\seqn{\mu_n}$ is norm bounded and pointwise convergent to $\mu$, and, by \cite[Theorem 11, page 90]{Die84}, $\seqn{\mu_n}$ is weakly convergent to $\mu$ if and only if $\lim_{n\to\infty}\wh{\mu}_n(B)=\wh{\mu}(B)$ for every $B\in Bor(St(\aA))$.

\section{Concentration points and modifications of anti-Nikodym sequences\label{sec:anti}}

It is well known that if the Stone space $St(\aA)$ of a Boolean algebra $\aA$ contains a non-eventually constant convergent sequence, then $\aA$ has neither the Nikodym property, nor the Grothendieck property (cf. Example \ref{example:strong_conv_seq} below). The goal of this section is to generalize this observation by studying properties of special sequences of measures on Boolean algebras, which witness the lack of the Nikodym property, and by showing that their existence implies the lack of the Grothendieck property, too.

For the rest of this section $\aA$ will always stand for an infinite Boolean algebra. We start with the following alike definitions.% concerning sequences of measures witnessing the lack of the Nikodym property or the Grothendieck property.

\begin{definition}\label{def:anti_nik_alg}
    A sequence $\seqn{\mu_n}$ of measures on $St(\aA)$ is \textit{anti-Nikodym} if we have $\lim_{n\to\infty}\mu_n(A)=0$ for every $A\in Clopen(St(\aA))$ and $\sup_{n\io}\big\|\mu_n\big\|=\infty$.

    A sequence $\seqn{\mu_n}$ of measures on $\aA$ is \textit{anti-Nikodym} if the sequence $\seqn{\wh{\mu}_n}$ of the corresponding Radon extensions on $St(\aA)$ is anti-Nikodym.
\end{definition}

\begin{definition}
    A sequence $\seqn{\mu_n}$ of measures on $\aA$ is \textit{anti-Grothendieck} if it is weak* convergent to $0$ but not weakly convergent.
\end{definition}

Trivially by definition, $\aA$ has the Nikodym property if and only if there are no anti-Nikodym sequences on $\aA$, and, similarly, $\aA$ has the Grothendieck property if and only if there are no anti-Grothendieck sequences on $\aA$.

Anti-Nikodym sequences of measures have a very important property, namely, they ``concentrate'' around points in the Stone spaces in a sense explained by the following lemma and definition.

\begin{lemma}\label{lemma:Ncp_exists}
    Let $A\in\aA$. If $\seqn{\mu_n}$ is a sequence in $\ba(\aA)$ such that $\sup_{n\io}\big\|\mu_n\rstr A\big\|=\infty$, then there is $t\in\clopen{A}_\aA$ such that for every $B\iA$ with $t\in\clopen{B}_\aA$ we have $\sup_{n\io}\big\|\mu_n\rstr B\big\|=\infty$.
\end{lemma}
\begin{proof}
Assume that for every $t\in\clopen{A}_\aA$ there is $B_t\in\aA$ such that $t\in\clopen{B_t}_\aA$ and $\sup_{n\io}\big\|\mu_n\rstr B_t\big\|<\infty$. By the compactness of $\clopen{A}_\aA$ there are $t_1,\ldots,t_k\in\clopen{A}_\aA$ such that $\clopen{A}_\aA\subseteq\bigcup_{i=1}^k\clopen{B_{t_i}}_\aA$. We then have
\[\infty=\sup_{n\io}\big\|\mu_n\rstr A\big\|\le\sup_{n\io}\sum_{i=1}^k\big\|\mu_n\rstr B_{t_i}\big\|<\infty,\]
which is a contradiction.
\end{proof}
%
%\noindent (Note that in the above proof we did not use the fact that $\seqn{\mu_n}$ is pointwise convergent.)

Lemma \ref{lemma:Ncp_exists} justifies introducing the following notion, already studied in \cite{Sob19} and \cite{SZ17} (see also \cite{Aiz92}).

\begin{definition}\label{def:nik_cp}
If $\seqn{\mu_n}$ is an anti-Nikodym sequence of measures on $\aA$, then a point $t\in St(\aA)$ is its \textit{Nikodym concentration point} if for every $A\iA$ with $t\in\clopen{A}_\aA$ we have $\sup_{n\io}\big\|\mu_n\rstr A\big\|=\infty$.
\end{definition}

Consequently, Lemma \ref{lemma:Ncp_exists} implies that every anti-Nikodym sequence of measures has a Nikodym concentration point.

\begin{lemma}\label{lemma:Ncp_large_nhbds}
If $\seqn{\mu_n}$ is an anti-Nikodym sequence of measures on $\aA$ and $t\in St(\aA)$ is its Nikodym concentration point, then for every $\alpha>0$ and $A\in\aA$ with $t\in\clopen{A}_\aA$ there are $n\io$ and $B\le A$ such that $t\not\in\clopen{B}_\aA$ and $\big|\mu_n(B)\big|>\alpha$.
\end{lemma}
\begin{proof}
    Since $\sup_{n\io}\big\|\mu_n\rstr A\big\|=\infty$, there are $C\le A$ and $n\io$ such that
    \[\big|\mu_n(C)\big|>\alpha+\sup_{m\io}\big|\mu_m(A)\big|.\]
    If $t\not\in\clopen{C}_\aA$, then let $B=C$. Otherwise, $t\not\in\clopen{A\sm C}_\aA$ and we have
    \[\big|\mu_n(A\sm C)\big|=\big|\mu_n(A)-\mu_n(C)\big|\ge\big|\mu_n(C)\big|-\big|\mu_n(A)\big|>\]
    \[>\alpha+\sup_{m\io}\big|\mu_m(A)\big|-\big|\mu_n(A)\big|\ge\alpha,\]
    so set $B=A\sm C$.
\end{proof}

An inductive application of Lemma \ref{lemma:Ncp_large_nhbds} yields the following standard corollary.

\begin{corollary}\label{cor:antiN_antichain}
If $\seqn{\mu_n}$ is an anti-Nikodym sequence of measures on $\aA$, then there exist an antichain $\seqk{A_k}$ in $\aA$ and a strictly increasing sequence $\seqk{n_k}$ in $\omega$ such that $\big|\mu_{n_k}\big(A_k\big)\big|>k$ for every $k\io$.
\end{corollary}

%If $St(\aA)$ contains a non-trivial sequence $\seqn{x_n}$ convergent to $x\in St(\aA)$, then, for the anti-Nikodym sequence $\seqn{n\big(\delta_{x}-\delta_{x_n}\big)}$ of measures on $\aA$, the point $x$ has the following strong concentration property.

To establish a new criterion for a Boolean algebra without the Nikodym property to not have the Grothendieck property either, we will need however the following stronger notion of Nikodym concentration points.

\begin{definition}\label{def:strong_nik_cp}
    Let $\seqn{\mu_n}$ be an anti-Nikodym sequence of measures on $\aA$. If $t\in St(\aA)$ has the property that for every $N\io$ and $A\iA$ with $t\in\clopen{A}_\aA$ there are a clopen set $U\sub \clopen{A}_\aA\sm\{t\}$ and $n\io$ such that
\[\big|\wh{\mu}_n(U)\big|>N\cdot\Big(\big\|\mu_n\rstr A^c\big\|+1\Big),\]
    then $t$ is a \textit{strong Nikodym concentration point} of $\seqn{\mu_n}$.
\end{definition}

It is easy to see that every strong Nikodym concentration point is a Nikodym concentration point. The converse is however not true, see Example \ref{example:strong_no} below.

\begin{example}\label{example:strong_conv_seq}
    Assume that $\seqn{x_n}$ is a non-eventually constant sequence in the Stone space $St(\aA)$ which converges to some $x\in St(\aA)$. Then, the point $x$ is a strong Nikodym concentration point of the anti-Nikodym sequence $\seqn{\mu_n}$ of measures on $\aA$, defined for every $n\io$ via their Radon extensions as follows:
    \[\wh{\mu}_n=n\big(\delta_{x_n}-\delta_x\big).\]
    (Note that the sequence $\seqn{\mu_n/n}$ is anti-Grothendieck on $\aA$). See also Proposition \ref{prop:nf_strong} and Corollary \ref{cor:2c_many_strong} for more general but similar constructions.
\end{example}

\begin{example}\label{example:many_strong}
    Expanding Example \ref{example:strong_conv_seq}, we can easily get an example of an anti-Nikodym sequence of measures with infinitely many strong Nikodym concentration points. Indeed, assume that $\aA$ contains an antichain $\seqk{A_k}$ such that for each $k\io$ there is a point $t_k\in\clopen{A_k}_\aA$ being the limit of a non-eventually constant sequence $\seqn{x_{k,n}}$ which is entirely contained in $\clopen{A_k}_\aA$. Let $a\colon\omega\to\omega\times\omega$ be a bijection; write $a=(b,c)$, where $b,c\colon\omega\to\omega$. For each $n\io$ define the measure $\mu_n$ on $\aA$ via its Radon extension as follows:
    \[\wh{\mu}_n=n\Big(\delta_{x_{b(n),c(n)}}-\delta_{t_{b(n)}}\Big).\]
    It follows that each point $t_k$ is a strong Nikodym concentration point of $\seqn{\mu_n}$.
\end{example}

\begin{example}\label{example:strong_no}
Let $\seqn{\mu_n}$ be the sequence of measures on $Clopen(\Cantor)$ defined in the proof of \cite[Proposition 4.6]{Sob19}. Then, every $t\in \Cantor$ is a Nikodym concentration point of $\seqn{\mu_n}$, but $\seqn{\mu_n}$ has no strong Nikodym concentration points. To see this, let $t\in\Cantor$ be arbitrary and set
\[V=\big\{x\in\Cantor\colon x(0)=t(0)\big\},\]
so that $t\in V$. For every $n\io$ we have
\[\big\|\mu_n\rstr V\big\| = \big\|\mu_n\rstr V^c\big\| = 2^n.\]
Therefore, for every clopen set $U\sub V\sm\{t\}$ and $n\io$ it holds
\[\big|\mu_n(U)\big|\le\big\|\mu_n\rstr V\big\|<\big\|\mu_n\rstr V^c\big\|+1,\]
and so $t$ is not a strong Nikodym concentration point of $\seqn{\mu_n}$.
\end{example}

\begin{example}\label{example:strong_only}
Lemma \ref{lemma:one_Ncp_sNcp} below asserts that if an anti-Nikodym sequence $\seqn{\mu_n}$ on a Boolean algebra $\aA$ has a unique Nikodym concentration point $t\in St(\aA)$, then $t$ must be a strong Nikodym concentration point of $\seqn{\mu_n}$. In Example \ref{example:wo_no_nikodym} we describe a large class $\cCc$ of Boolean subalgebras of $\wo$ carrying only anti-Nikodym sequences with single Nikodym concentration points and therefore carrying only anti-Nikodym sequences with strong Nikodym concentration points. The simplest element of $\cCc$ is the Boolean algebra
\[\big\{A\in\wo\colon\ |A|<\infty\text{ or }|A^c|<\infty\big\},\]
generated by the ideal $Fin$ of finite subsets of $\omega$, whose Stone space is homeomorphic to the space $\{1/(n+1)\colon n\io\}\cup\{0\}$ endowed with the standard Euclidean topology.
\end{example}

%\begin{example}\label{example:strong_schachermayer}
%    \cite{Sch82}
%\end{example}

\begin{example}\label{example:strong_consistently_no}
    Examples \ref{example:strong_conv_seq} and \ref{example:strong_no} show that the Boolean algebra $Clopen(\Cantor)$ carries anti-Nikodym sequences with strong Nikodym concentration points as well as anti-Nikodym sequences without such points.
    Theorem \ref{thm:main_no_nik_no_gr} from Introduction, on the other hand, implies that if a Boolean algebra $\aA$ has the Grothendieck property but not the Nikodym property, then no anti-Nikodym sequence of measures on $\aA$ can have a strong Nikodym concentration point. Such Boolean algebras have been however so far constructed only consistently, e.g. under the Continuum Hypothesis (Talagrand \cite{Tal84}) or Martin's Axiom (Sobota and Zdomskyy \cite{SZ23g}); cf. Question \ref{ques:no_strong_cp}.
\end{example}

The following lemma (aforementioned in Example \ref{example:strong_only}) provides a handy criterion for Nikodym concentration points to be strong.

\begin{lemma}\label{lemma:one_Ncp_sNcp}
    Let $\seqn{\mu_n}$ be an anti-Nikodym sequence of measures on $\aA$. Let $t\in St(\aA)$ be the only Nikodym concentration point of $\seqn{\mu_n}$. Then, $t$ is a strong Nikodym concentration point of $\seqn{\mu_n}$.
%
%    Consequently, $\aA$ does not have the Grothendieck property.
\end{lemma}
\begin{proof}
    Fix $N\io$ and $A\in\aA$ such that $t\in\clopen{A}_\aA$. Since $t\not\in\clopen{A^c}_\aA$ and $\seqn{\mu_n}$ has only $t$ as its Nikodym concentration point, Lemma \ref{lemma:Ncp_exists} implies that $\sup_{m\io}\big\|\mu_m\rstr A^c\big\|<\infty$. As $\sup_{n\io}\big\|\mu_n\rstr A\big\|=\infty$, by Lemma \ref{lemma:Ncp_large_nhbds} we get a clopen $U\sub\clopen{A}_\aA\setminus\{t\}$ and $n\io$ such that
    \[\big|\wh{\mu}_n(U)\big|>N\cdot\Big(\sup_{m\io}\big\|\mu_m\rstr A^c\big\|+1\Big)\ge N\cdot\Big(\big\|\mu_n\rstr A^c\big\|+1\Big),\]
    which proves that $t$ is a strong Nikodym concentration point of $\seqn{\mu_n}$.
%
%    The second statement follows from Corollary \ref{cor:no_nik_no_gr} (or Theorem \ref{thm:main_no_nik_no_gr}).
\end{proof}

\begin{corollary}\label{cor:finitely_many_Ncp}
Assume that $\seqn{\mu_n}$ is an anti-Nikodym sequence of measures on $\aA$ and $t\in St(\aA)$ is its Nikodym concentration point such that for some $A\in t$ there is no Nikodym concentration point of $\seqn{\mu_n}$ in $\clopen{A}_\aA\sm\{t\}$. Then, $\seqn{\mu_n\rstr A}$ is an anti-Nikodym sequence on $\aA$ with the unique strong Nikodym concentration point $t$.
\end{corollary}

In the following series of lemmas we will show that if a Boolean algebra $\aA$ carries an anti-Nikodym sequence of measures with a strong Nikodym concentration point, then $\aA$ carries an anti-Nikodym sequences of measures with pairwise disjoint supports (see Theorem \ref{thm:sNcp_disj_supp}). We start with the following standard result.

\begin{lemma}\label{lemma:small_nhbds}
Let $\mu$ be a measure on $\aA$. For every $\eps>0$ and $t\in St(\aA)$ there is $A\iA$ such that $t\in \clopen{A}_\aA$ and $|\wh{\mu}|\big(\clopen{A}_\aA\sm\{t\}\big)<\eps$.
\end{lemma}
\begin{proof}
Assume not, that is, there are $\eps>0$ and $t\in St(\aA)$ such that for every $A\iA$ with $t\in\clopen{A}_\aA$ we have $|\wh{\mu}|\big(\clopen{A}_\aA\sm\{t\}\big)\ge\eps$. Let $N\io$ be such that $N\cdot\eps/2>\|\wh{\mu}\|$. By the regularity of $|\wh{\mu}|$, there is an antichain $A_1,\ldots,A_N\in\aA$ such that $t\not\in\clopen{A_i}_\aA$ and $|\wh{\mu}|\big(\clopen{A_i}_\aA\big)>\eps/2$ for $i=1,\ldots,N$. Then,
\[\|\wh{\mu}\|\ge\sum_{i=1}^N|\wh{\mu}|\big(\clopen{A_i}_\aA\big)>N\cdot\eps/2>\|\wh{\mu}\|,\]
a contradiction.
\end{proof}

\begin{lemma}\label{lemma:aN_disjointification}
Let $\seqn{\mu_n}$ be an anti-Nikodym sequence of measures on $\aA$ and $t\in St(\aA)$ its strong Nikodym concentration point. Let $\seqn{a_n}$ be a sequence of positive real numbers. Then, there are an anti-Nikodym sequence $\seqn{\nu_n}$ of measures on $\aA$ and an antichain $\seqn{B_n}$ in $\aA$ such that for every $n\io$ we have:
\begin{itemize}
	\item $t\not\in\clopen{B_n}_\aA$,
	\item $\supp\big(\wh{\nu}_n\big)\sub\clopen{B_n}_\aA\cup\{t\}$,
	\item $\big\|\wh{\nu}_n\rstr St(\aA)\sm\{t\}\big\|>a_n$.
\end{itemize}
\end{lemma}
\begin{proof}
Let $\seqn{M_n}$ be an increasing sequence of positive real numbers such that $M_0>1$ and
\[\tag{$*$}\lim_{n\to\infty}a_n\cdot M_n=\lim_{n\to\infty}M_n=\infty.\]

Put $A_0=1_\aA$. Since $t$ is a strong Nikodym concentration point of $\seqn{\mu_n}$, there are $n_0\io$ and a clopen set $U_0\sub\clopen{A_0}_\aA\sm\{t\}$ such that
\[\big|\wh{\mu}_{n_0}\big(U_0\big)\big|>a_0\cdot M_0^2.\]
By Lemma \ref{lemma:small_nhbds}, there is $A_1\in\aA$ with $t\in\clopen{A_1}_\aA\sub\clopen{A_0}_\aA\sm U_0$ and such that
\[\big|\wh{\mu}_{n_0}\big|\big(\clopen{A_1}_\aA\sm\{t\}\big)<1.\]
Similarly, there are $n_1>n_0$ and a clopen set $U_1\sub\clopen{A_1}_\aA\sm\{t\}$ such that
\[\big|\wh{\mu}_{n_1}\big(U_1\big)\big|>a_1\cdot M_1^2\cdot\Big(\big\|\mu_{n_1}\rstr A_1^c\big\|+1\Big).\]
By Lemma \ref{lemma:small_nhbds}, there is $A_2\in\aA$ with $t\in\clopen{A_2}_\aA\sub\clopen{A_1}_\aA\sm U_1$ and such that
\[\big|\wh{\mu}_{n_1}\big|\big(\clopen{A_2}_\aA\sm\{t\}\big)<1/2.\]
Again, there are $n_2>n_1$ and a clopen set $U_2\sub\clopen{A_2}_\aA\sm\{t\}$ such that
\[\big|\wh{\mu}_{n_2}\big(U_2\big)\big|>a_2\cdot M_2^2\cdot\Big(\big\|\mu_{n_2}\rstr A_2^c\big\|+1\Big).\]
By Lemma \ref{lemma:small_nhbds}, there is $A_3\in\aA$ with $t\in\clopen{A_3}_\aA\sub\clopen{A_2}_\aA\sm U_2$ and such that
\[\big|\wh{\mu}_{n_2}\big|\big(\clopen{A_3}_\aA\sm\{t\}\big)<1/3.\]

We continue in this manner until we get a descending sequence $\seqk{A_k}$ in $\aA$, a sequence $\seqk{U_k}$ of pairwise disjoint clopen subsets of $St(\aA)$, and a strictly increasing sequence $\seqk{n_k}$ of natural numbers such that for every $k\io$ we have:
\begin{itemize}
%	\item $\seqk{A_k}$ is descending,
	\item $t\in\clopen{A_k}_\aA$,
	\item $U_k\sub\clopen{A_k\sm A_{k+1}}_\aA$,
%	\item $n_i<n_{i+1}$ for every $i\io$,
%	\item $\sup_n\big\|\mu_n\rstr A_i^c\big\|<\infty$ for every $i\io$,
	\item $\big|\wh{\mu}_{n_k}\big(U_k\big)\big|>a_k\cdot M_k^2\cdot\Big(\big\|\mu_{n_k}\rstr A_k^c\big\|+1\Big)$,
	\item $\big\|\wh{\mu}_{n_k}\rstr\big(\clopen{A_{k+1}}_\aA\sm\{t\}\big)\big\|<1/(k+1)$.
\end{itemize}

\medskip

We will define the desired sequence $\seqk{\nu_k}$ in two steps. First, for every $k\io$ we put
\[\theta_k=\mu_{n_k}\big/\big(M_k\cdot\big\|\mu_{n_k}\rstr A_k^c\big\|+M_k\big).\]
We claim that $\seqk{\theta_k}$ is anti-Nikodym. Indeed, for every $A\iA$ and $k\io$ we have
\[\big|\theta_k(A)\big|=\big|\mu_{n_k}(A)\big|\big/\big(M_k\cdot\big\|\mu_{n_k}\rstr A_k^c\big\|+M_k\big)\le\big|\mu_{n_k}(A)\big|,\]
as $M_k>1$, and hence $\lim_{k\to\infty}\big|\theta_k(A)\big|=0$,
which proves that $\seqk{\theta_k}$ is pointwise convergent to $0$. For every $k\io$ it also holds
\[\tag{$**$}\big\|\theta_k\big\|\ge\big\|\theta_k\rstr\big(A_k\sm A_{k+1}\big)\big\|\ge\big|\wh{\theta}_k\big(U_k\big)\big|=\]
\[=\big|\wh{\mu}_{n_k}\big(U_k\big)\big|\big/\big(M_k\cdot\big\|\mu_{n_k}\rstr A_k^c\big\|+M_k\big)>a_k\cdot M_k,\]
so $\seqk{\theta_k}$ is not norm bounded by ($*$). Similarly as in ($**$), for every $k\io$ and $A\in\aA$ we also have
\[\tag{$*\!*\!*$}\big|\theta_k\big(A\wedge A_k^c\big)\big|=\big|\mu_{n_k}\big(A\wedge A_k^c\big)\big|\big/\big(M_k\cdot\big\|\mu_{n_k}\rstr A_k^c\big\|+M_k\big)\le\]
\[\le\big\|\mu_{n_k}\rstr A_k^c\big\|\big/\big(M_k\cdot\big\|\mu_{n_k}\rstr A_k^c\big\|+M_k\big)\le1/M_k.\]

\medskip

Second, for every $k\io$ define $\nu_k$ via its Radon extension $\wh{\nu}_k$ on $St(\aA)$ as follows:
\[\wh{\nu}_k=\wh{\theta}_k\rstr\big(\clopen{A_k\sm A_{k+1}}_\aA\cup\{t\}\big).\]
It follows immediately that for every $k\neq l\io$ we have
\[\supp\big(\wh{\nu}_k\big)\cap\supp\big(\wh{\nu}_l\big)\sub\{t\},\]
so for every $k\io$ put $B_k=A_k\sm A_{k+1}$. We need to show that $\seqk{\nu_k}$ is anti-Nikodym. Let $A\iA$ and $k\io$. Having in mind that $A_{k+1}\le A_k$ and $t\in\clopen{A_{k+1}}_\aA$, we get
\[\big|\nu_k(A)\big|=\big|\theta_k\big(A\wedge \big(A_k\sm A_{k+1}\big)\big)+\wh{\theta}_k\big(\clopen{A}_\aA\cap\{t\}\big)\big|=\]
\[=\big|\theta_k(A)-\theta_k\big(A\sm\big(A_k\sm A_{k+1}\big)\big)+\wh{\theta}_k\big(\clopen{A}_\aA\cap\{t\}\big)\big|=\]
\[=\big|\theta_k(A)-\theta_k\big(A\sm A_k\big)-\theta_k\big(A\wedge A_{k+1}\big)+\wh{\theta}_k\big(\clopen{A}_\aA\cap\{t\}\big)\big|=\]
\[=\big|\theta_k(A)-\theta_k\big(A\wedge A_k^c\big)-\wh{\theta}_k\big(\clopen{A}_\aA\cap \clopen{A_{k+1}}_\aA\big)+\wh{\theta}_k\big(\clopen{A}_\aA\cap\{t\}\big)\big|=\]
\[=\big|\theta_k(A)-\theta_k\big(A\wedge A_k^c\big)-\wh{\theta}_k\Big(\clopen{A}_\aA\cap\big(\clopen{A_{k+1}}_\aA\sm\{t\}\big)\Big)\big|\le\]
\[\le\big|\theta_k(A)\big|+\big|\theta_k\big(A\wedge A_k^c\big)\big|+\big|\wh{\theta}_k\Big(\clopen{A}_\aA\cap\big(\clopen{A_{k+1}}_\aA\sm\{t\}\big)\Big)\big|\le\]
\[\le\big|\theta_k(A)\big|+\big|\theta_k\big(A\wedge A_k^c\big)\big|+\big\|\wh{\mu}_{n_k}\rstr\big(\clopen{A_{k+1}}_\aA\sm\{t\}\big)\big\|<\]
%\[<\big|\theta_i(U)\big|+1/i+1/i=\big|\theta_i(U)\big|+2/i,\]
\[<\big|\theta_k(A)\big|+1/{M_k}+1/(k+1),\]
where the last inequality follows from ($*\!*\!*$) and the properties of the measure $\wh{\mu}_{n_k}$. Consequently, $\lim_{k\to\infty}\big|\nu_k(A)\big|=0$,
%\[\lim_{k\to\infty}\big|\nu_k(A)\big|\le\lim_{k\to\infty}\Big(\big|\theta_k(A)\big|+1/M_k+1/(k+1)\Big)=0,\]
so $\seqk{\nu_k}$ is pointwise convergent to $0$. To prove that $\seqn{\nu_n}$ is not norm bounded, take $k\io$ and note that by ($**$) we have
\[\big\|\nu_k\big\|=\big\|\wh{\theta}_k\rstr\big(\clopen{A_k\sm A_{k+1}}_\aA\cup\{t\}\big)\big\|\ge\big\|\theta_k\rstr\big(A_k\sm A_{k+1}\big)\big\|>a_k\cdot M_k,\]
which implies that $\sup_{k\io}\big\|\nu_k\big\|=\infty$. Besides, again by ($**$), for every $k\io$ we have
\[\big\|\wh{\nu}_k\rstr St(\aA)\sm\{t\}\big\|=\big\|\nu_k\rstr\big(A_k\sm A_{k+1}\big)\big\|>a_k,\]
which finishes the proof.
\end{proof}

The proof of the next lemma follows closely Step $2$ of the proof of \cite[Proposition 3.2]{Zuc25}.

\begin{lemma}\label{lemma:aN_disjointification_t}
Let $\seqn{a_n}$ be a sequence of positive real numbers such that $\lim_{n\to\infty}a_n=\infty$. Let $\seqn{\rho_n}$ be an anti-Nikodym sequence of measures on $\aA$, $t\in St(\aA)$, and $\seqn{B_n}$ an antichain in $\aA$ such that for every $n\io$ we have:
\begin{enumerate}[(a)]
	\item $t\not\in\clopen{B_n}_\aA$,
	\item $\supp\big(\wh{\rho}_n\big)\sub\clopen{B_n}_\aA\cup\{t\}$,
	\item $\big\|\wh{\rho}_n\rstr St(\aA)\sm\{t\}\big\|>a_n$.
\end{enumerate}
Then, there are an anti-Nikodym sequence $\seqk{\nu_k}$ of measures on $\aA$ and a strictly increasing sequence $\seqk{n_k}$ such that for every $k\io$ we have:
\begin{enumerate}[(A)]
	\item $\supp\big(\wh{\nu}_k\big)\sub\clopen{B_{n_{2k}}\vee B_{n_{2k+1}}}_\aA$,
	\item $\big\|\nu_k\big\|>a_{n_{2k}}$.
\end{enumerate}
\end{lemma}
\begin{proof}
We divide the proof into two cases.

\medskip

\underline{Case 1.} We have $\liminf_{n\to\infty} \big| \wh{\rho}_n(\{t\}) \big| < \infty$.

We find a subsequence $\seqk{\rho_{n_k}}$ and $\alpha\in\R$ such that $\lim_{k\to\infty}\wh{\rho}_{n_k}(\{t\})=\alpha$. %, so in particular $\seqk{\wh{\rho}_{n_k}(\{t\})}$ is a Cauchy sequence. 
By (c) and $\lim_{n\to\infty}a_n=\infty$, the sequence $\seqk{\rho_{n_k}}$ is anti-Nikodym on $\aA$. For every $k\io$ we define $\nu_k$ via its Radon extension $\wh{\nu}_k$ on $St(\aA)$ as follows:
\[\wh{\nu}_k= \big( \wh{\rho}_{n_{2k}} - \wh{\rho}_{n_{2k+1}} \big) \rstr\big(St(\aA) \sm \{t\}\big).\]
For every $k\io$ by (b) we have
\[ \supp\big(\wh{\nu}_k\big)\sub \big(\supp\big(\wh{\rho}_{n_{2k}}\big)\cup \supp\big(\wh{\rho}_{n_{2k+1}}\big) \big) \sm \{t\} \sub \clopen{B_{n_{2k}}\vee B_{n_{2k+1}}}_\aA,\]
so (A) holds. Since $\big\|\wh{\rho}_n\rstr St(\aA)\sm\{t\}\big\|>a_n$ for every $n\io$ and $\seqn{B_n}$ is an antichain in $\aA$, we get that
\[\big\|\nu_k\big\| = \big\|\wh{\nu}_k\big\| > a_{n_{2k}}+a_{n_{2k+1}} > a_{n_{2k}}\]
for all $k\io$, so (B) holds as well. In particular, $\seqk{\nu_k}$ is not norm bounded.

%It suffices to prove that $\wh{\nu}_k \big(\clopen{A}_\aA\big)\to 0$ for every $A\in\aA$ and that $\sup_{k\io}\big\|\wh{\nu}_k \big\| = \infty$.
%%for every $A\in\aA$ with $t\in\clopen{A}_\aA$.

We need to prove that $\seqk{\nu_k}$ is anti-Nikodym on $\aA$. Let $A\in\aA$. For any $l\io$ we have
\[\nu_l(A)=\wh{\nu}_l \big(\clopen{A}_\aA\big) =
\big( \wh{\rho}_{n_{2l}} - \wh{\rho}_{n_{2l+1}} \big) \big(\clopen{A}_\aA \sm \{t\} \big) =\]
\[= \big( \wh{\rho}_{n_{2l}} - \wh{\rho}_{n_{2l+1}} \big) \big(\clopen{A}_\aA \big) - \big( \wh{\rho}_{n_{2l}} - \wh{\rho}_{n_{2l+1}} \big) \big(\clopen{A}_\aA \cap \{t\} \big)=\]
\[= \big( {\rho}_{n_{2l}} (A) - {\rho}_{n_{2l+1}} (A) \big)-\Big( \wh{\rho}_{n_{2l}}\big(\clopen{A}_\aA \cap \{t\} \big) - \wh{\rho}_{n_{2l+1}}\big(\clopen{A}_\aA \cap \{t\} \big) \Big).\]
As $\seqk{\rho_{n_k}}$ is pointwise convergent to $0$ and it also holds
\[\lim_{k\to\infty}\Big(\wh{\rho}_{n_{2k}} (\{t\} ) - \wh{\rho}_{n_{2k+1}} (\{t\} )\Big)=\alpha-\alpha=0,\]
%(recall that $\seqk{\wh{\rho}_{n_k}(\{t\})}$ is Cauchy), 
we get that $\lim_{k\to\infty}\nu_k(A)=0$. It follows that $\seqk{\nu_k}$ is pointwise convergent to $0$ and so anti-Nikodym.

%Next, let $A\in\aA$ be such that $t\in\clopen{A}_\aA$.

\medskip

\underline{Case 2.} We have $\lim_{n\to\infty} \big| \wh{\rho}_n(\{t\}) \big| = \infty$.

We find a subsequence $\seqk{\wh{\rho}_{n_k}}$ such that $\seqk{\big| \wh{\rho}_{n_k}(\{t\}) \big| }$ is strictly increasing and $\big| \wh{\rho}_{n_0}(\{t\}) \big|>0$. By (c) the sequence $\seqk{\rho_{n_{2k}}}$ is still anti-Nikodym. For each $k\io$ we denote
\[ \alpha_k = \frac{\wh{\rho}_{n_{2k}}(\{t\})} {\wh{\rho}_{n_{2k+1}}(\{t\})},\]
so $0<|\alpha_k|<1$. For every $k\io$ we define $\nu_k$ via its Radon extension $\wh{\nu}_k$ on $St(\aA)$ as follows:
\[\wh{\nu}_k= \big( \wh{\rho}_{n_{2k}} - \alpha_k\cdot \wh{\rho}_{n_{2k+1}} \big) \rstr\big(St(\aA) \sm \{t\}\big).\]
Again, as previously, for every $k\io$ we have
\[ \supp\big(\wh{\nu}_k\big)\sub \big(\supp\big(\wh{\rho}_{n_{2k}}\big)\cup \supp\big(\wh{\rho}_{n_{2k+1}}\big) \big) \sm \{t\} \sub \clopen{B_{n_{2k}}\vee B_{n_{2k+1}}}_\aA,\]
so (A) holds.

Since $\big\|\wh{\rho}_n\rstr St(\aA)\sm\{t\}\big\|>a_n$ for every $n\io$ and $\seqn{B_n}$ is an antichain in $\aA$, we get that
\[\big\|\nu_k\big\| = \big\|\wh{\nu}_k\big\| > a_{n_{2k}}+ \big|\alpha_k\big|\cdot a_{n_{2k+1}} > a_{n_{2k}}\]
for all $k\io$, so (B) holds, too. Again, $\seqk{\nu_k}$ is also not norm bounded.%It suffices again to prove that $\wh{\nu}_k \big(\clopen{A}_\aA\big)\to 0$ for every $A\in\aA$ and that $\sup_{k\io}\big\|\wh{\nu}_k \big\| = \infty$.

We again need to prove that $\seqk{\nu_k}$ is pointwise convergent to $0$. Let $A\in\aA$. For every $l\io$ we have
\[\nu_l(A)=\wh{\nu}_l \big(\clopen{A}_\aA\big) =
\big( \wh{\rho}_{n_{2l}} - \alpha_l \cdot \wh{\rho}_{n_{2l+1}} \big) \big(\clopen{A}_\aA \sm \{t\} \big) =
\] \[
= \big( \wh{\rho}_{n_{2l}} - \alpha_l \cdot \wh{\rho}_{n_{2l+1}} \big) \big(\clopen{A}_\aA \big) - \big( \wh{\rho}_{n_{2l}} - \alpha_l \cdot \wh{\rho}_{n_{2l+1}} \big) \big(\clopen{A}_\aA \cap \{t\} \big). \]
As
\[\wh{\rho}_{n_{2l}} (\{t\} ) - \alpha_l \cdot \wh{\rho}_{n_{2l+1}} (\{t\} ) = 0\]
for every $l\io$, $\seqk{\rho_{n_k}}$ is an anti-Nikodym sequence, and $\seqk{\alpha_k}$ is bounded by $1$, we get that $\lim_{k\to\infty}\nu_k(A)=0$, as required. 
\end{proof}

The following theorem concludes the above lemmas.

\begin{theorem}\label{thm:sNcp_disj_supp}
If $\aA$ carries an anti-Nikodym sequence of measures which has a strong Nikodym concentration point in $St(\aA)$, then there are an anti-Nikodym sequence $\seqk{\nu_k}$ of measures on $\aA$ and an antichain $\seqk{B_k}$ in $\aA$ such that for every $k\io$ we have:
\begin{itemize}
    \item $\supp\big(\wh{\nu}_k\big)\sub\clopen{B_k}_\aA$,
    \item $\big\|\nu_k\big\|>k$.
\end{itemize}
\end{theorem}
\begin{proof}
     Combine Lemmas \ref{lemma:aN_disjointification} and \ref{lemma:aN_disjointification_t} (with $a_k=k$ for each $k\io$).
\end{proof}

%\begin{remark}\label{remark:t_aN_property}
%The property of the point $t$ assumed in the statement of Lemma \ref{lemma:aN_disjointification} implies \textit{itself} that $t$ is a Nikodym concentration point of $\seqn{\mu_n}$.
%\end{remark}

\subsection{Strong Nikodym concentration points and the Grothendieck property\label{sec:no_gr}}

In this short subsection we will quickly show how the existence of anti-Nikodym sequences of measures on a Boolean algebra $\aA$ with strong Nikodym concentration points implies that $\aA$ does not have the Grothendieck property. We start with the following folklore proposition (cf. \cite[Theorem 19.3.5]{KKLPS}).

\begin{proposition}\label{prop:aN_to_aG}
Let $\seqn{\mu_n}$ be an anti-Nikodym sequence of non-zero measures on $\aA$ and $\seqn{B_n}$ such an antichain in $\aA$ that $\supp\big(\wh{\mu}_n\big)\sub\clopen{B_n}_\aA$ for every $n\io$. %Let $\seqn{N_n}$ be an unbounded increasing sequence of positive real numbers such that $N_0>1$ and $\big\|{\mu}_n\big\|>N_n$ for every $n\io$.
Then, %there exists an anti-Grothendieck sequence $\seqn{\nu_n}$ of measures on $\aA$.
the sequence $\seqn{\mu_n/\big\|\mu_n\big\|}$ of normalized measures is an anti-Grothendieck sequence on $\aA$.
\end{proposition}
\begin{proof}
For every $n\io$ let $\nu_n=\mu_n/\big\|\mu_n\big\|$, so $\big\|\nu_n\big\|=1$. For every $A\iA$ and sufficiently large $n\io$ we have
\[\big|\nu_n(A)\big|=\big|\mu_n(A)\big|/\big\|\mu_n\big\|\le\big|\mu_n(A)\big|,\]
so $\lim_{n\to\infty}\nu_n(A)=0$. Consequently, $\seqn{\nu_n}$ is also pointwise convergent to $0$ and hence weak* convergent to $0$.

We claim that $\seqn{\nu_n}$ is not weakly convergent to $0$. %We claim that there exists $\eps>0$ such that for almost all $n\io$ we have
%\[\big\|\wh{\nu}_n\rstr St(\aA)\sm\{t\}\big\|>\eps.\]
%For the sake of contradiction assume that for every $n,N\io$ there exists $m_{n,N}>N$ such that
%\[\big\|\wh{\nu}_{m_{n,N}}\rstr St(\aA)\sm\{t\}\big\|<1/(n+1).\]
%It follows that for every $n,N\io$ we have
%\[\big|\wh{\nu}_{m_{n,N}}(\{t\})\big|>1-1/(n+1)>0,\]
%and hence
%\[\big|\wh{\nu}_{m_{n,N}}(St(\aA))\big|=\big|\wh{\nu}_{m_{n,N}}(\{t\})+\wh{\nu}_{m_{n,N}}(St(\aA)\sm\{t\})\big|\ge\]
%\[\ge\big|\wh{\nu}_{m_{n,N}}(\{t\})\big|-\big\|\wh{\nu}_{m_{n,N}}\rstr St(\aA)\sm\{t\}\big\|>1-2/(n+1).\]
%This is a contradiction, since $\lim_{n\to\infty}\big|\wh{\nu}_n(St(\aA))\big|=\lim_{n\to\infty}\big|\nu_n\big(1_\aA\big)\big|=0$.
%
For every $n\io$, since $\big\|\nu_n\big\|=1$ and $\supp\big(\wh{\nu}_n\big)\sub\clopen{B_n}_\aA$, there is a clopen set $U_n\sub\clopen{B_n}_\aA$ such that
\[\big|\wh{\nu}_n\big(U_n\big)\big|>1/3.\]
Set
\[B=\bigcup_{n\io}U_n.\]
Then, $B$ is an open set, so Borel. For each $n\io$, since
\[B\cap\supp\big(\wh{\nu}_n\big)=B\cap B_n\cap\supp\big(\wh{\nu}_n\big)=U_n\cap\supp\big(\wh{\nu}_n\big),\]%\sub\clopen{B_n}_\aA,\]
we have $\big|\wh{\nu}_n(B)\big|>1/3$, and hence $\seqn{\nu_n}$ is not weakly convergent to $0$. It follows that $\seqn{\nu_n}$ is an anti-Grothendieck sequence on $\aA$.
\end{proof}

Note that the following proposition generalizes the aforementioned well-known fact asserting that if a totally disconnected compact space $K$ contains a non-eventually constant convergent sequence, then $Clopen(K)$ does not have the Nikodym property nor the Grothendieck property (cf. Example \ref{example:strong_conv_seq}). 

\begin{proposition}\label{prop:no_nik_no_gr}
Assume that $\aA$ carries a sequence $\seqn{\mu_n}$ of measures which is pointwise convergent to $0$ and there is a point $t\in St(\aA)$ such that for every $N\io$ and $A\iA$ with $t\in\clopen{A}_\aA$ there are a clopen $U\sub\clopen{A}_\aA\sm\{t\}$ and $n\io$ for which we have
\[\big|\wh{\mu}_n(U)\big|>N\cdot\Big(\big\|\mu_n\rstr A^c\big\|+1\Big).\]
Then, $\aA$ does not have the Nikodym property nor the Grothendieck property.
\end{proposition}
\begin{proof}
$\aA$ does not have the Nikodym property, since $\seqn{\mu_n}$ is by definition anti-Nikodym. Moreover, by Definition \ref{def:strong_nik_cp}, $t$ is its strong Nikodym concentration point.

From $\seqn{\mu_n}$, with an aid of Theorem \ref{thm:sNcp_disj_supp}, we obtain a new anti-Nikodym sequence $\seqk{\nu_k}$ on $\aA$ and an antichain $\seqk{B_k}$ in $\aA$ such that for every $k\io$ we have $\supp\big(\wh{\nu}_k\big)\sub\clopen{B_k}_\aA$ and $\big\|\nu_k\big\|>k$. Proposition \ref{prop:aN_to_aG} implies that the sequence $\seqk{\nu_k/\big\|\nu_k\big\|}$ is anti-Grothendieck and so that $\aA$ does not have the Grothendieck property.
\end{proof}

Proposition \ref{prop:no_nik_no_gr} now easily yields \textbf{Theorem \ref{thm:main_no_nik_no_gr}} from Introduction, thus providing a criterion for a Boolean algebra without the Nikodym property to not have the Grothendieck property.

%\begin{theorem}\label{thm:no_nik_no_gr_2}
%    If $\aA$ does not have the Nikodym property and this is witnessed by an anti-Nikodym sequence with a strong Nikodym concentration point, then $\aA$ does not have the Grothendieck property.
%\end{theorem}

%\begin{lemma}
%    Let $\seqn{\mu_n}$ be an anti-Nikodym sequence of measures on $\aA$. Set
%    \[F=\big\{t\in St(\aA)\colon\ t\text{ is a Nikodym concentration point of }\seqn{\mu_n}\big\}.\]
%    If $x\in F$ is isolated (in the relative topology), then $x$ is a strong Nikodym concentration point of $\seqn{\mu_n}$.
%
%    Consequently, $\aA$ does not have the Grothendieck property.
%\end{lemma}
%\begin{proof}
%    Let $A\in\aA$ be such that $F\cap\clopen{A}_\aA=\{x\}$. Since $\seqn{\mu_n}$ is pointwise convergent to $0$ and norm unbounded on $A$, $x$ is not isolated in $St(\aA)$.
%
%    The second statement follows from Corollary \ref{cor:no_nik_no_gr}.
%\end{proof}

We finish this section with the following important consequence of Proposition \ref{prop:no_nik_no_gr} (or of Theorem \ref{thm:main_no_nik_no_gr}); cf. Corollary \ref{cor:sigmacompl_ideal} below.

\begin{corollary}\label{cor:one_Ncp_gr}
    If $\aA$ admits an anti-Nikodym sequence of measures having a unique Nikodym concentration point, then $\aA$ does not have the Grothendieck property.
\end{corollary}

%\begin{remark}
%    One can easily adapt the argument for Corollary \ref{cor:one_Ncp_sNcp} and obtain the following stronger result.
%    Let $\seqn{\mu_n}$ be an anti-Nikodym sequence of measures on $\aA$. Set
%    \[F=\big\{t\in St(\aA)\colon\ t\text{ is a Nikodym concentration point of }\seqn{\mu_n}\big\}.\]
%    If $x\in F$ is isolated (in the relative topology), then $x$ is a strong Nikodym concentration point of $\seqn{\mu_n}$.
%\end{remark}

\section{Boolean algebras generated by ideals\label{sec:ideals}}

We start with the following definitions, analogous to Definitions \ref{def:nikodym} and \ref{def:anti_nik_alg} for Boolean algebras.

\begin{definition}\label{def:nikodym_ring}
A ring $\rR$ of subsets of a set $X$ has the \textit{Nikodym property} if every sequence $\seqn{\mu_n}$ in $\ba(\rR)$ such that $\lim_{n\to\infty}\mu_n(A)=0$ for every $A\in\rR$ is norm bounded.
\end{definition}

Equivalently, a ring $\rR$ of subsets of a set $X$ has the Nikodym property if every sequence $\seqn{\mu_n}$ in $\ba(\rR)$ such that $\sup_{n\io}\big|\mu_n(A)\big|<\infty$ for every $A\in\rR$ is norm bounded.

\begin{definition}\label{def:anti_nik_ring}
    Let $\rR$ be a ring of subsets of a set $X$. A sequence $\seqn{\mu_n}$ of measures on $\rR$ is \textit{anti-Nikodym} if $\lim_{n\to\infty}\mu_n(A)=0$ for every $A\in\rR$ and $\sup_{n\io}\|\mu_n\|=\infty$.
\end{definition}

It immediately follows that a Boolean algebra $\aA$ has the Nikodym property if and only if the ring $Clopen(St(\aA))$ of clopen subsets of $St(\aA)$ has the Nikodym property. Also, a ring $\rR$ has the Nikodym property if and only if there are no anti-Nikodym sequences on $\rR$.

Drewnowski and Pa\'ul \cite[Remark 2.2.(a)]{DP00} observed that a ring $\rR\sub\wp(X)$ of subsets of a set $X$ has the Nikodym property if and only if the Boolean subalgebra $\wp(X)\gen{\rR}$ of $\wp(X)$ generated by $\rR$ has the Nikodym property. For the ``if'' part they briefly noted that if $X\not\in\rR$, then every measure $\mu\in\ba(\rR)$ can be easily extended to a measure $\nu\in\ba\big(\wp(X)\gen{\rR}\!\big)$, simply by setting $\nu(X\sm A)=-\mu(A)$ for every $A\in\rR$. 

We will expand the above observations in the next subsection by describing, for a Boolean algebra $\aA$ and its ideal $\iI$, an isomorphism between the Banach spaces $\ba(\iI)\oplus\R$ and $\ba(\aA\gen{\iI})$ (Proposition \ref{prop:extension_on_alg} and Theorem \ref{thm:isomorphism}). Our approach is topological, that is, we study the values of the Radon extensions of measures from the space $\ba(\aA\gen{\iI})$ on the open subset of the Stone space $St(\aA\gen{\iI})$ corresponding to the ideal $\iI$.

In Subsection \ref{sec:ideals_prop} we will use results obtained in the previous sections to show that if a Boolean algebra $\aA$ has the Nikodym property, but its subalgebra $\aA\gen{\iI}$ generated by an ideal $\iI$ does not, then $\aA\gen{\iI}$ does not have the Grothendieck property either (Proposition \ref{prop:ai_uniq_concentr_point} and Corollary \ref{cor:sigmacompl_ideal}). We will also obtain a characterization of the Nikodym property for Boolean algebras of the form $\aA\gen{\iI}$ in terms of sequences of non-negative measures (Theorem \ref{thm:ai_nikodym_char_pos}).

\subsection{Extending measures from ideals to Boolean algebras\label{sec:extending}}

Let $\aA$ be a Boolean algebra. Note that if $\fF$ is a filter in $\aA$, then $\aA\gen{\fF}=\fF\cup\fF^*$, where $\fF^*=\{1_\aA\sm A\colon A\in\fF\}$ is the ideal in $\aA$ dual to $\fF$, and $\fF$ is an ultrafilter in $\aA\gen{\fF}$. It follows that $\fF\in St(\aA\gen{\fF})$, but in order to distinguish $\fF$ as a \textit{filter} in $\aA$ (or as an \textit{ultrafilter} in $\aA\gen{\fF}$) from the \textit{point} in $St(\aA\gen{\fF})$ we will denote the latter by $p_\fF$. Of course, formally $\fF=p_\fF$. Note that $p_\fF\not\in\clopen{A}_{\aA\gen{\fF}}$ for every $A\in\fF^*$ and that
\[\tag{$\star$}St(\aA\gen{\fF})=\bigcup\big\{\clopen{A}_{\aA\gen{\fF}}\colon\ A\in\fF^*\big\}\cup\big\{p_\fF\big\}.\]
Consequently, since $\big\{\clopen{A}_{\aA\gen{\fF}}\colon\ \aA\in\fF^*\big\}$ is an increasing net of open sets (indexed by the directed set $(\fF^*,\le)$, where $\le$ is the standard order on $\aA$) and by the $\tau$-additivity of Radon measures (see \cite[Section 7.2]{BogV2}), for every measure $\mu\in\ba(\aA\gen{\fF})$ we have
\begin{align*}\tag{$+$}
\|\mu\|&=\|\wh{\mu}\|=\Big\|\wh{\mu}\rstr\bigcup_{A\in\fF^*}\clopen{A}_{\aA\gen{\fF}}\Big\|+\big|\wh{\mu}\big(\big\{p_\fF\big\}\big)\big|=|\wh{\mu}|\Big(\bigcup_{A\in\fF^*}\clopen{A}_{\aA\gen{\fF}}\Big)+\big|\wh{\mu}\big(\big\{p_\fF\big\}\big)\big|=\\
&=\lim_{A\in\fF^*}|\wh{\mu}|\big(\clopen{A}_{\aA\gen{\fF}}\big)+\big|\wh{\mu}\big(\big\{p_\fF\big\}\big)\big|=\lim_{A\in\fF^*}|\mu|(A)+\big|\wh{\mu}\big(\big\{p_\fF\big\}\big)\big|
\end{align*}
and so, in the case of non-negative $\mu$,
\[\mu\big(1_\aA\big)=\lim_{A\in\fF^*}\mu(A)+\wh{\mu}\big(\big\{p_\fF\big\}\big)\]
and
\[\wh{\mu}\big(\big\{p_\fF\big\}\big)=\lim_{A\in\fF}\mu(A).\]

The following lemma is a standard application of the regularity of Radon measures.

\begin{lemma}\label{lemma:regularity}
    Let $\aA$ be a Boolean algebra and $\mu\in\ba(\aA)$. For every Borel set $B$ and open set $U$ such that $B\sub U\sub St(\aA)$ and every $\eps>0$ there is $C\in\aA$ such that $\clopen{C}_\aA\sub U$ and $|\wh{\mu}(B)|\le|\mu(C)|+\eps$.
\end{lemma}
\begin{proof}
    Let $St(\aA)=P\cup N$ be a Hahn decomposition of the space $St(\aA)$ for the measure $\wh{\mu}$ such that $\wh{\mu}(P\cap B)\ge0$ and $\wh{\mu}(N\cap B)\le0$ for every Borel set $B\sub St(\aA)$. Let $\wh{\mu}=\mu^+-\mu^-$ be the Jordan decomposition of $\wh{\mu}$ into two non-negative Radon measures.

    Fix a Borel set $B$ and an open set $U$ such that $B\sub U\sub St(\aA)$. Fix also $\eps>0$. We may assume that $\mu^+(B)\ge\mu^-(B)$ as the other case is symmetric. By the regularity of $\mu^+$ there is a compact set $K\sub B\cap P$ such that
    \[\tag{$*$}\mu^+(B)\le\mu^+(K)+\eps/2.\]
    By the compactness of $K$ and the regularity of the variation $|\wh{\mu}|$, there is $C\in\aA$ such that $K\sub\clopen{C}_\aA\sub U$ and
    \[|\wh{\mu}|\big(\clopen{C}_\aA\big)\le|\wh{\mu}|(K)+\eps/2=\mu^+(K)+\eps/2.\]
    The above inequality yields that
    \[\mu^-\big(\clopen{C}_\aA\big)=|\wh{\mu}|\big(\clopen{C}_\aA\big)-\mu^+\big(\clopen{C}_\aA\big)\le|\wh{\mu}|\big(\clopen{C}_\aA\big)-\mu^+(K)\le\]
    \[\le\mu^+(K)+\eps/2-\mu^+(K)=\eps/2,\]
    and hence
    \[\tag{$**$}-\mu^-\big(\clopen{C}_\aA\sm B\big)\ge-\mu^-\big(\clopen{C}_\aA\big)\ge-\eps/2.\]
    Using ($*$) and ($**$), we have
    \[|\wh{\mu}(B)|=\big|\mu^+(B)-\mu^-(B)\big|=\mu^+(B)-\mu^-(B)\le\mu^+(K)+\eps/2-\mu^-(B)\le\]
    \[\le\mu^+\big(\clopen{C}_\aA\big)+\eps/2-\mu^-\big(\clopen{C}_\aA\cap B\big)=\mu^+\big(\clopen{C}_\aA\big)+\eps-\mu^-\big(\clopen{C}_\aA\cap B\big)-\eps/2\le\]
    \[\le\mu^+\big(\clopen{C}_\aA\big)+\eps-\mu^-\big(\clopen{C}_\aA\cap B\big)-\mu^-\big(\clopen{C}_\aA\sm B\big)=\]
    \[=\mu^+\big(\clopen{C}_\aA\big)+\eps-\mu^-\big(\clopen{C}_\aA\big)=\wh{\mu}\big(\clopen{C}_\aA\big)+\eps=\mu(C)+\eps\le|\mu(C)|+\eps,\]
    and the proof of the lemma is finished.
\end{proof}

\begin{lemma}\label{lemma:regularity2}
    Let $\aA$ be a Boolean algebra and $\mu\in\ba(\aA)$. For every open subset $U$ of $St(\aA)$, disjoint Borel sets $B,B'\sub U$, and $\eps>0$, there are disjoint $C,C'\in\aA$ such that $\clopen{C}_\aA,\clopen{C'}_\aA\sub U$, $|\wh{\mu}(B)|<|\mu(C)|+\eps$, and $|\wh{\mu}(B')|<|\mu(C')|+\eps$.
\end{lemma}
\begin{proof}
    Let $U$ be an open subset of $St(\aA)$, $B,B'$ disjoint Borel subsets of $U$, and $\eps>0$. By the regularity of $|\wh{\mu}|$, there are compact sets $K\sub B$, $K'\sub B'$ such that $|\wh{\mu}|(B\sm K)<\eps/2$ and $|\wh{\mu}|(B'\sm K')<\eps/2$. Since $K\cap K'=\emptyset$, by the normality of $St(\aA)$ there are disjoint open sets $V,V'\sub U$ such that $K\sub V$ and $K'\sub V'$. Using Lemma \ref{lemma:regularity} twice, we get elements $C,C'\in\aA$ such that $\clopen{C}_\aA\sub V$, $\clopen{C'}_\aA\sub V'$, as well as
    \[\tag{$*$}|\wh{\mu}(K)|\le|\mu(C)|+\eps/2\quad\text{and}\quad|\wh{\mu}(K')|\le|\mu(C')|+\eps/2.\]
    Note that $C\wedge C'=0_\aA$ and $\clopen{C}_\aA,\clopen{C'}_\aA\sub U$.    

    We have
    \[\eps/2>|\wh{\mu}|(B\sm K)\ge|\wh{\mu}(B\sm K)|=|\wh{\mu}(B)-\wh{\mu}(K)|\ge|\wh{\mu}(B)|-|\wh{\mu}(K)|,\]
    hence
    \[|\wh{\mu}(K)|\ge|\wh{\mu}(B)|-\eps/2.\]
    Combining the latter inequality with ($*$), we get
    \[|\wh{\mu}(B)|-\eps/2\le|\mu(C)|+\eps/2,\]
    and hence $|\wh{\mu}(B)|\le|\mu(C)|+\eps$. We similarly show that $|\wh{\mu}(B')|\le|\mu(C')|+\eps$.
\end{proof}

\begin{lemma}\label{lemma:norms}
    Let $\iI$ be an ideal in a Boolean algebra $\aA$. Let $\mu\in\ba(\aA\gen{\iI})$ and $\nu\in\ba(\iI)$ be such measures that $\mu$ extends $\nu$. Then,
    \[\Big\|\wh{\mu}\rstr\bigcup_{A\in\iI}\clopen{A}_{\aA\gen{\iI}}\Big\|=\|\nu\|.\]
\end{lemma}
\begin{proof}
    We have
    \[\Big\|\wh{\mu}\rstr\bigcup_{A\in\iI}\clopen{A}_{\aA\gen{\iI}}\Big\|=|\wh{\mu}|\Big(\bigcup_{A\in\iI}\clopen{A}_{\aA\gen{\iI}}\Big)\ge\]
    \[\ge\sup\big\{\big|\wh{\mu}\big(\clopen{A}_{\aA\gen{\iI}}\big)\big|+\big|\wh{\mu}\big(\clopen{A'}_{\aA\gen{\iI}}\big)\big|\colon\ A,A'\in\iI,\ A\wedge A'=0_\aA\big\}=\]
    \[=\sup\big\{|\mu(A)|+|\mu(A')|\colon\ A,A'\in\iI,\ A\wedge A'=0_\aA\big\}=\]
    \[=\sup\big\{|\nu(A)|+|\nu(A')|\colon\ A,A'\in\iI,\ A\wedge A'=0_\aA\big\}=\|\nu\|.\]

    For the reverse inequality, let $\eps>0$ and note that by Lemma \ref{lemma:regularity2} we have%the regularity of $\wh{\mu}$ for any Borel set $B\sub\bigcup_{A\in\iI}\clopen{A}_{\aA\gen{\iI}}$ there is $C\in\iI$ such that
%    \[|\wh{\mu}|\big(B\sm\clopen{C}_{\aA\gen{\iI}}\big)<\eps/2,\]%/2\quad\text{and}\quad|\wh{\mu}|\big(\clopen{C}_{\aA\gen{\iI}}\sm B\big)<\eps/2.\]
%    so we have
%    \[|\wh{\mu}|(B)=|\wh{\mu}|\big(B\cap\clopen{C}_{\aA\gen{\iI}}\big)+|\wh{\mu}|\big(B\sm\clopen{C}_{\aA\gen{\iI}}\big)\le|\wh{\mu}|\big(\clopen{C}_{\aA\gen{\iI}}\big)+\eps/2.\]
%    We thus have
    \[\Big\|\wh{\mu}\rstr\bigcup_{A\in\iI}\clopen{A}_{\aA\gen{\iI}}\Big\|=|\wh{\mu}|\Big(\bigcup_{A\in\iI}\clopen{A}_{\aA\gen{\iI}}\Big)=\]
    \[=\sup\bigg\{|\wh{\mu}(B)|+|\wh{\mu}(B')|\colon B,B'\in Bor\Big(\bigcup_{A\in\iI}\clopen{A}_{\aA\gen{\iI}}\Big),\ B\cap B'=\emptyset\bigg\}\le\]
    \[\le\sup\big\{|\mu(C)|+|\mu(C')|\colon C,C'\in\iI,\ C\wedge C'=0_\aA\big\}+2\eps=\]
    \[=\sup\big\{|\nu(C)|+|\nu(C')|\colon C,C'\in\iI,\ C\wedge C'=0_\aA\big\}+2\eps=\|\nu\|+2\eps.\]
    Since $\eps$ was arbitrary, we get
    \[\Big\|\wh{\mu}\rstr\bigcup_{A\in\iI}\clopen{A}_{\aA\gen{\iI}}\Big\|\le\|\nu\|,\]
    which finishes the proof of the equality.
\end{proof}

Using the above results, we get the following decomposition of the norms of extensions.

\begin{proposition}\label{prop:extension_norms}
    Let $\iI$ be an ideal in a Boolean algebra $\aA$. Let $\mu\in\ba(\aA\gen{\iI})$ and $\nu\in\ba(\iI)$ be such measures that $\mu$ extends $\nu$. Then,
    \[\|\mu\|=\|\nu\|+\big|\wh{\mu}\big(\big\{p_{\iI^*}\big\}\big)\big|=\lim_{A\in\iI}|\nu|(A)+\big|\wh{\mu}\big(\big\{p_{\iI^*}\big\}\big)\big|.\]
\end{proposition}
\begin{proof}
    Combining equalities ($+$) with Lemma \ref{lemma:norms}, we get that $\|\mu\|=\|\nu\|+\big|\wh{\mu}\big(\big\{p_{\iI^*}\big\}\big)\big|$. For every $A\in\iI$ we have $|\nu|(A)=|\mu|(A)$ and so, again by ($+$), $\|\nu\|=\lim_{A\in\iI}|\mu|(A)=\lim_{A\in\iI}|\nu|(A)$.
\end{proof}

The next proposition essentially extends Drewnowski and Pa\'{u}l's observation, mentioned at the beginning of this section.

\begin{proposition}\label{prop:extension_on_alg}
    Let $\aA$ be an infinite Boolean algebra and let $\iI$ be an ideal in $\aA$. Let $\alpha\in\R$. For every $\mu\in\ba(\iI)$, the function $T_\alpha(\mu)\colon\aA\gen{\iI}\to\R$, given for every $A\in\aA\gen{\iI}$ by
    \[T_\alpha(\mu)(A)=\begin{cases}
        \mu(A),&\text{ if }A\in\iI,\\
        -\mu(A^c)+\alpha,&\text{ if }A\not\in\iI,\\
    \end{cases}\]
    is a measure on $\aA\gen{\iI}$, i.e. $T_\alpha(\mu)\in\ba(\aA\gen{\iI})$, extending $\mu$ and satisfying:
    \begin{enumerate}[(i),itemsep=2mm]
        \item $\|(\mu,\alpha)\|_\infty\le\big\|T_\alpha(\mu)\big\|\le2\|(\mu,\alpha)\|_1$,
        \item $\wh{T_\alpha(\mu)}\big(\big\{p_{\iI^*}\big\}\big)=\alpha-\wh{T_\alpha(\mu)}\Big(\bigcup_{A\in\iI}\clopen{A}_{\aA\gen{\iI}}\Big)$.
    \end{enumerate}
%    and hence $\big\|T_\alpha(\mu)\big\|=\|\mu\|+|\alpha|$.

    \noindent Consequently, if $\mu$ is non-negative, then we have
    \begin{enumerate}[(i'),itemsep=2mm]\setcounter{enumi}{1}
        \item $\wh{T_\alpha(\mu)}\big(\big\{p_{\iI^*}\big\}\big)=\alpha-\lim_{A\in\iI}\mu(A)$.
    \end{enumerate}
\end{proposition}
\begin{proof}
Fix $\mu\in\ba(\iI)$. Obviously, $T_\alpha(\mu)$ is a function extending $\mu$ onto $\aA\gen{\iI}$.

We prove that $T_\alpha(\mu)$ is finitely additive. It is enough to check that for every $A\in\iI$ and $B\in\iI^*$ such that $A\wedge B=0_\aA$ we have $T_\alpha(\mu)(A\vee B)=T_\alpha(\mu)(A)+T_\alpha(\mu)(B)$. So, let $A$ and $B$ be as required. We have $A\vee B\in\iI^*$, so by  definition
\[T_\alpha(\mu)(A\vee B)=-\mu((A\vee B)^c)+\alpha.\]
Since $A\le B^c\in\iI$, it holds
\[T_\alpha(\mu)(B)=-\mu(B^c)+\alpha=-\mu(A\wedge B^c)-\mu(A^c\wedge B^c)+\alpha=-\mu(A\wedge B^c)-\mu((A\vee B)^c)+\alpha=\]
\[=-\mu(A)+T_\alpha(\mu)(A\vee B)=-T_\alpha(\mu)(A)+T_\alpha(\mu)(A\vee B),\]
hence $T_\alpha(\mu)(A\vee B)=T_\alpha(\mu)(A)+T_\alpha(\mu)(B)$.

\medskip

We have %{\color{red} Moze dodac informacje o rownowaznej def. normy miar}
\[\big\|T_\alpha(\mu)\big\|=\sup\big\{\big|T_\alpha(\mu)(A)\big|+\big|T_\alpha(\mu)(B)\big|\colon\ A,B\in\aA\gen{\iI},\ A\wedge B=0_\aA,\ A\vee B=1_\aA\big\}=\]
\[=\sup\big\{\big|T_\alpha(\mu)(A)\big|+\big|T_\alpha(\mu)(A^c)\big|\colon\ A\in\iI\big\}=\sup\big\{|\mu(A)|+|-\mu(A)+\alpha|\colon\ A\in\iI\big\}\le\]
\[\le2\sup\big\{|\mu(A)|\colon\ A\in\iI\big\}+|\alpha|\le2\|\mu\|+|\alpha|=\|(2\mu,\alpha)\|_1\le2\|(\mu,\alpha)\|_1\]
and similarly
\[\big\|T_\alpha(\mu)\big\|=\sup\big\{|\mu(A)|+|-\mu(A)+\alpha|\colon\ A\in\iI\big\}\ge\big|\mu\big(0_\aA\big)\big|+\big|-\mu\big(0_\aA\big)+\alpha\big|=|\alpha|.\]
Moreover, we also have
\[\big\|T_\alpha(\mu)\big\|=\sup\Big\{\big|T_\alpha(\mu)(A)\big|+\big|T_\alpha(\mu)(B)\big|+\big|T_\alpha(\mu)\big((A\vee B)^c\big)\big|\colon\ A,B\in\iI,\ A\wedge B=0_\aA\Big\}=\]
\[=\sup\big\{|\mu(A)|+|\mu(B)|+|-\mu(A\vee B)+\alpha|\colon\ A,B\in\iI,\ A\wedge B=0_\aA\big\}\ge\]
\[\ge\sup\big\{|\mu(A)|+|\mu(B)|\colon\ A,B\in\iI,\ A\wedge B=0_\aA\big\}=\|\mu\|,\]
hence $\big\|T_\alpha(\mu)\big\|\ge\|(\mu,\alpha)\|_\infty$. Consequently, (i) holds and so also $T_\alpha(\mu)\in\ba(\aA\gen{\iI})$.

\medskip

%From this, by Corollary \ref{cor:extension_norms}, it will follow that $\big\|T_\alpha(\mu)\big\|=\|\mu\|+|\alpha|$.{\color{red}......................}
By ($\star$), we have
\[\wh{T_\alpha(\mu)}\big(\big\{p_{\iI^*}\big\}\big)=\wh{T_\alpha(\mu)}\big(St(\aA\gen{\iI})\big)-\wh{T_\alpha(\mu)}\Big(\bigcup_{A\in\iI}\clopen{A}_{\aA\gen{\iI}}\Big)=\]
\[=T_\alpha(\mu)\big(1_\aA\big)-\wh{T_\alpha(\mu)}\Big(\bigcup_{A\in\iI}\clopen{A}_{\aA\gen{\iI}}\Big)=-\mu\big(0_\aA)+\alpha-\wh{T_\alpha(\mu)}\Big(\bigcup_{A\in\iI}\clopen{A}_{\aA\gen{\iI}}\Big)=\]
\[=\alpha-\wh{T_\alpha(\mu)}\Big(\bigcup_{A\in\iI}\clopen{A}_{\aA\gen{\iI}}\Big),\]
which proves (ii).

\medskip

Let $U=\bigcup_{A\in\iI}\clopen{A}_{\aA\gen{\iI}}$. If $\mu$ is non-negative, then the measure $\nu\in M(St(\aA\gen{\iI}))$ given by the formula 
\[\nu=\Big(\wh{T_\alpha(\mu)}\rstr U\Big)+0\cdot\delta_{p_{\iI^*}}\]
is also non-negative. To see this, assume that there is a Borel set $B\sub U$ such that $\gamma=\nu(B)<0$. By the regularity of $\nu$, there is a compact set $K\sub B$ such that $|\nu|(B\sm K)<|\gamma|/4$ as well as an open set $V\sub U$ such that $K\sub V$ and $|\nu|(V\sm K)<|\gamma|/4$. By the compactness of $K$, we can assume that $V$ is clopen, that is, $V=\clopen{A}_\aA$ for some $A\in\iI$. We have
\[|\nu(B)-\nu(V)|=\big|\nu(B)-\nu(K)+\nu(K)-\nu(V)\big|\le|\nu(B)-\nu(K)|+|\nu(K)-\nu(V)|=\]
\[=|\nu(B\sm K)|+|\nu(V\sm K)|\le|\nu|(B\sm K)+|\nu|(V\sm K)<|\gamma|/4+|\gamma|/4=|\gamma|/2.\]
As $\nu(B)=\gamma<0$ and $A\in\iI$, it follows from the above calculations that
\[0<\mu(A)=T_\alpha(\mu)(A)=\wh{T_\alpha(\mu)}\big(\clopen{A}_\aA\big)=\nu\big(\clopen{A}_\aA\big)=\nu(V)<\gamma/2<0,\]
which is clearly a contradiction.

Consequently, (ii') follows from (ii) and the $\tau$-additivity of $\nu$:
\[\wh{T_\alpha(\mu)}\big(\big\{p_{\iI^*}\big\}\big)=\alpha-\wh{T_\alpha(\mu)}\Big(\bigcup_{A\in\iI}\clopen{A}_{\aA\gen{\iI}}\Big)=\alpha-\lim_{A\in\iI}\wh{T_\alpha(\mu)}\big(\clopen{A}_\aA\big)=\]
\[=\alpha-\lim_{A\in\iI}T_\alpha(\mu)(A)=\alpha-\lim_{A\in\iI}\mu(A).\]
\end{proof}

Note that by the above proof in condition (i) we even have the stronger inequality:
\[\big\|T_\alpha(\mu)\big\|\le\|(2\mu,\alpha)\|_1.\]

%Note that the mapping $T_\cdot\colon\ba(\iI)\oplus\R\to\ba(\aA\gen{\iI})$ described in Proposition \ref{prop:extension_on_alg} is linear.

\medskip

Proposition \ref{prop:extension_on_alg} yields the aforementioned isomorphism between the Banach spaces $\ba(\iI)\oplus\R$ and $\ba(\aA\gen{\iI})$.

\begin{theorem}\label{thm:isomorphism}
    Let $\aA$ be an infinite Boolean algebra and let $\iI$ be an ideal in $\aA$. Then, the mapping
    \[\ba(\iI)\oplus\R\ni(\mu,\alpha)\longmapsto T_\alpha(\mu)\in\ba(\aA\gen{\iI}),\]
    defined in Proposition \ref{prop:extension_on_alg}, is an isomorphism between the Banach spaces $\ba(\iI)\oplus\R$ and $\ba(\aA\gen{\iI})$ with $\|T_\cdot\|\le2$.
\end{theorem}
\begin{proof}
    It is immediate that $T_\cdot$ is an injective linear mapping.

    Let $\mu\in\ba(\aA\gen{\iI})$. Set $\nu=\mu\rstr\iI$ and $\alpha=\mu\big(1_\aA\big)$. As $T_\alpha(\nu)$ extends $\nu$, we have $T_\alpha(\nu)(A)=\mu(A)$ for every $A\in\iI$. For $A\in\iI^*$ we have
    \[T_\alpha(\nu)(A)=-\nu(A^c)+\alpha=-\mu(A^c)+\mu\big(1_\aA\big)=\mu(A).\]
    It follows that $T_\alpha(\nu)=\mu$ and so $T_\cdot$ is a bijection. By Proposition \ref{prop:extension_on_alg}.(i), $T_\cdot$ is bounded with $\|T_\cdot\|\le2$ and hence, by the Open Mapping Theorem, $T_\cdot$ is an isomorphism.
\end{proof}

\subsection{Properties of ideals and Boolean algebras generated by them\label{sec:ideals_prop}}

%We will now use results obtained in the previous sections to show that if a Boolean algebra $\aA$ has the Nikodym property, but its subalgebra $\aA\gen{\iI}$ generated by an ideal $\iI$ does not, then $\aA\gen{\iI}$ does not have the Grothendieck property either (Proposition \ref{prop:ai_uniq_concentr_point} and Corollary \ref{cor:sigmacompl_ideal}). We will also obtain a characterization of the Nikodym property for Boolean algebras of the form $\aA\gen{\iI}$ in terms of sequences of non-negative measures (Theorem \ref{thm:ai_nikodym_char_pos}).
We now study the relation between the Nikodym property and the Grothendieck property of Boolean algebras of the form $\aA\gen{\iI}$ as well as provide a characterization of the former property in terms of non-negative measures. We first introduce the following auxiliary definitions, supplementing Definitions \ref{def:nikodym}, \ref{def:nikodym_ring}, \ref{def:anti_nik_ring}, and \ref{def:anti_nik_alg}.

\begin{definition}\label{def:nik_ideal}
An ideal $\iI$ in a Boolean algebra $\aA$ has the \textit{Nikodym property} if the ring $\rR(\iI)$ has the Nikodym property.
\end{definition}

\begin{definition}\label{def:anti_nik_ideal}
    Let $\iI$ be an ideal in a Boolean algebra $\aA$. A sequence $\seqn{\mu_n}$ of measures on $\iI$ is \textit{anti-Nikodym} if the sequence $\seqn{\mathring{\mu}_n}$ of the corresponding measures on $\rR(\iI)$ is anti-Nikodym.
\end{definition}

As we will need this result later, for the sake of completeness, we now reprove the aforementioned observation of Drewnowski and Pa\'{u}l \cite[Remark 2.2.(a)]{DP00}, in the case of ideals and Boolean algebras. We use the topological language developed in the previous subsection.%, which makes the argument quite natural.

\begin{theorem}[{Drewnowski--Pa\'{u}l \cite{DP00}}]\label{thm:nik_prop_ideal_algebra}
Let $\aA$ be a Boolean algebra and let $\iI$ be an ideal in $\aA$. Then, $\iI$ has the Nikodym property if and only if the Boolean algebra $\aA\gen{\iI}$ has the Nikodym property.
\end{theorem}
\begin{proof}
Suppose that $\iI$ has the Nikodym property but there is an anti-Nikodym sequence $\seqn{\mu_n}$ of measures on $\aA\gen{\iI}$. Of course, 
\[\lim_{n\to\infty}\wh{\mu}_n\big(\clopen{A}_{\aA\gen{\iI}}\big)=0\]
for every $A\in\iI$, and so
\[\lim_{n\to\infty}\big(\wh{\mu}_n\rstr\rR(\iI)\big)(A)=0\]
for every $A\in\rR(\iI)$. On the other hand, Corollary \ref{cor:antiN_antichain} implies that there exist an antichain $\seqk{A_k}$ in $\aA\gen{\iI}$ and a strictly increasing sequence $\seqk{n_k}$ satisfying
\[\big|\wh{\mu}_{n_k}\big(\clopen{A_k}_{\aA\gen{\iI}}\big)\big|=\big|\mu_{n_k}\big(A_k\big)\big|>k\] for every $k\io$. At most one of $A_k$'s is an element of the dual filter $\iI^*$, and so for almost all $k\io$ we have $A_k\in\iI$. It follows that
\[\sup_{k\io}\big\|\wh{\mu}_{n_k}\rstr\rR(\iI)\big\|=\infty,\]
and so $\seqk{\wh{\mu}_{n_k}\rstr\rR(\iI)}$ is an anti-Nikodym sequence on $\rR(\iI)$, which is a contradiction as $\rR(\iI)$ has the Nikodym property.

Assume now that $\aA\gen{\iI}$ has the Nikodym property but there exists an anti-Nikodym sequence $\seqn{\mu_n}$ on $\iI$. For each $n\io$, let $\mu_n'=T_0\big(\mu_n\big)$ be the extension of $\mu_n$ onto $\aA\gen{\iI}$ given by Proposition \ref{prop:extension_on_alg} (for $\alpha=0$). It is immediate that for each $A\in\aA\gen{\iI}$ we have
\[\lim_{n\to\infty}\mu_n'(A)=0.\]
By Proposition \ref{prop:extension_on_alg}.(i), it also holds
\[\infty=\sup_{n\io}\big\|\mu_n\big\|\le\sup_{n\io}\big\|\mu_n'\big\|.\]
Consequently, $\seqn{\mu_n'}$ is an anti-Nikodym sequence on $\aA\gen{\iI}$, which is again a contradiction.
\end{proof}

%For an ideal $\iI$ in a Boolean algebra $\aA$ note that for every $A\in\iI$ we have $\aA_A=(\aA\gen{\iI})_A$.

The following proposition, describing the relation between the Nikodym property and the Grothendieck property of Boolean algebras $\aA\gen{\iI}$, will be crucial for results in the next section.

\begin{proposition}\label{prop:ai_uniq_concentr_point}
    Let $\aA$ be a Boolean algebra and let $\iI$ be an ideal in $\aA$. Assume that
    \begin{itemize}
        \item for every $A\in\iI$ the restricted algebra $\aA_A$ has the Nikodym property,
        \item the subalgebra $\aA\gen{\iI}$ does not have the Nikodym property. 
    \end{itemize}
    Then, for any anti-Nikodym sequence $\seqn{\mu_n}$ on $\aA\gen{\iI}$, the point $p_{\iI^*}\in St(\aA\gen{\iI})$ is the only Nikodym concentration point of $\seqn{\mu_n}$. 
    
    Consequently, $\aA\gen{\iI}$ does not have the Grothendieck property either.
\end{proposition}
\begin{proof}  
    Let $\seqn{\mu_n}$ be an anti-Nikodym sequence of measures on $\aA\gen{\iI}$. By Lemma \ref{lemma:Ncp_exists}, $\seqn{\mu_n}$ has a Nikodym concentration point $t\in St(\aA\gen{\iI})$.
    
    Suppose that there is $A\in\iI\sub\aA\gen{\iI}$ such that $t\in\clopen{A}_{\aA\gen{\iI}}$. By the definition of Nikodym concentration points and the fact that $\aA_A=(\aA\gen{\iI})_A$, we have
    \[\sup_{n\io}\big\|\mu_n\rstr(\aA\gen{\iI})_A\big\|=\sup_{n\io}\big\|\mu_n\rstr\aA_A\big\|=\sup_{n\io}\big\|\mu_n\rstr A\big\|=\infty.\]
    %\[\sup_{n\io}\big\|\mu_n\rstr A\big\|=\infty,\]
    %so, trivially,
    %\[\sup_{n\io}\big\|\mu_n\rstr(\aA\gen{\iI})_A\big\|=\infty,\] 
    %too. 
    Since we obviously have $\lim_{n\to\infty}\mu_n(B)=0$ for every $B\in(\aA\gen{\iI})_A\sub\aA\gen{\iI}$ as well, we get that the sequence
    \[\seqn{\mu_n\rstr(\aA\gen{\iI})_A}\]
    is an anti-Nikodym sequence on $(\aA\gen{\iI})_A$. But this is a contradiction, as $(\aA\gen{\iI})_A=\aA_A$ and $\aA_A$ has the Nikodym property. Consequently, there is no $A\in\iI$ such that $t\in\clopen{A}_{\aA\gen{\iI}}$. By ($\star$) (from the beginning of this section), it follows that $t=p_{\iI^*}$. In particular, $p_{\iI^*}$ is the only Nikodym concentration point of $\seqn{\mu_n}$.

    The second conclusion follows from Corollary \ref{cor:one_Ncp_gr}.
\end{proof}

Recall that, by the Nikodym--And\^{o} theorem, every $\sigma$-complete Boolean algebra has the Nikodym property.

\begin{corollary}\label{cor:sigmacompl_ideal}
    %Let $\aA$ be a $\sigma$-complete Boolean algebra and $\iI$ an ideal in $\aA$. If $\aA\gen{\iI}$ does not have the Nikodym property, then it does not have the Grothendieck property either.
    Let $\aA$ be a Boolean algebra with the Nikodym property (e.g. let $\aA$ be $\sigma$-complete) and $\iI$ an ideal in $\aA$. If $\aA\gen{\iI}$ does not have the Nikodym property, then it does not have the Grothendieck property either.
\end{corollary}

We will now characterize the Nikodym property of Boolean algebras $\aA\gen{\iI}$ in terms of non-negative measures. We start with the following preparatory proposition, rephrasing \cite[Proposition 2.4]{SZ19} (or \cite[Lemma 4.1]{SZ24}) in the context of algebras $\aA\gen{\iI}$.

\begin{proposition}\label{prop:ai_nikodym_char}
Let $\iI$ be an ideal in a Boolean algebra $\aA$. Then, the following conditions are equivalent:
\begin{enumerate}[(A),itemsep=1mm]
    \item the algebra $\aA\gen{\iI}$ has the Nikodym property;
    \item the restricted algebra $\aA_A$ has the Nikodym property for every $A\in\iI$ and there is no sequence $\seqn{\mu_n}$ of measures on $\aA\gen{\iI}$ satisfying the following three conditions:
    \begin{enumerate}[(a)]
	   \item $\sup_{n\in\omega}\big\|\mu_n\big\|=\infty$,
	   \item $\lim_{n\to\infty} \mu_n\big(1_\aA\big)=0$,
	   \item $\lim_{n\to\infty}\big\|\mu_n\rstr A\big\|=0$ for every $A\in\iI$;
    \end{enumerate}
    \item the restricted algebra $\aA_A$ has the Nikodym property for every $A\in\iI$ and there is no sequence $\seqn{\mu_n}$ of measures on $\aA\gen{\iI}$ satisfying the following three conditions:
    \begin{enumerate}[(a')]
	   \item $\sup_{n\in\omega}\big\|\mu_n\big\|=\infty$,
	   \item $\sup_{n\in\omega}\big|\mu_n\big(1_\aA\big)\big|<\infty$,
	   \item $\sup_{n\in\omega}\big\|\mu_n\rstr A\big\|<\infty$ for every $A\in\iI$.
    \end{enumerate}
\end{enumerate}

\end{proposition}
\begin{proof}
(A)$\Rightarrow$(B): Suppose that $\aA\gen{\iI}$ has the Nikodym property. Trivially, for every $A\in\iI$ the Boolean algebra $\aA_A=(\aA\gen{\iI})_A$ has the Nikodym property, too. Assume however that there is a sequence $\seqn{\mu_n}$ of measures on $\aA\gen{\iI}$ satisfying conditions (a)--(c). By conditions (b) and (c), for every $A\in\iI^*$, we have
\[\lim_{n\to\infty}\big|\mu_n(A)\big|=\lim_{n\to\infty}\big|\mu_n\big(1_\aA\big)-\mu_n\big(1_\aA\sm A\big)\big|\le\lim_{n\to\infty}\big|\mu_n\big(1_\aA\big)\big|+\lim_{n\to\infty}\big|\mu_n\big(1_\aA\sm A\big)\big|\le\]
\[\le\lim_{n\to\infty}\big|\mu_n\big(1_\aA\big)\big|+\lim_{n\to\infty}\big\|\mu_n\rstr\big(1_\aA\sm A\big)\big\|=0,\]
which implies that $\seqn{\mu_n}$ is an anti-Nikodym sequence on $\aA\gen{\iI}$. Hence, $\aA\gen{\iI}$ does not have the Nikodym property, which is a contradiction.

\medskip

(B)$\Rightarrow$(C): Assume that there is a sequence $\seqn{\mu_n}$ in  $\ba(\aA\gen{\iI})$ satisfying conditions (a')--(c'). By going to a subsequence if necessary, by condition (a'), we may assume that
\[\tag{$*$}\lim_{n\to\infty}\big\|\mu_n\big\|=\infty,\]
and that each $\mu_n$ is non-zero. For each $n\io$ set:
\[\nu_n=\mu_n\big/\sqrt{\big\|\mu_n\big\|}.\]
We then have $\lim_{n\to\infty}\big\|\nu_n\big\|=\infty$. Moreover, by ($*$) and condition (b'), we have
\[\lim_{n\to\infty}\big|\nu_n\big(1_\aA\big)\big|=\lim_{n\to\infty}\frac{\big|\mu_n\big(1_\aA\big)\big|}{\sqrt{\big\|\mu_n\big\|}}\le\lim_{n\to\infty}\frac{\sup_{m\io}\big|\mu_m\big(1_\aA\big)\big|}{\sqrt{\big\|\mu_n\big\|}}=0.\]
Similarly, by ($*$) and condition (c'), for every $A\in\iI$ we have
\[\lim_{n\to\infty}\big\|\nu_n\rstr A\big\|=\lim_{n\to\infty}\frac{\big\|\mu_n\rstr A\big\|}{\sqrt{\big\|\mu_n\big\|}}\le\lim_{n\to\infty}\frac{\sup_{m\io}\big\|\mu_m\rstr A\big\|}{\sqrt{\big\|\mu_n\big\|}}=0.\]
It follows that the sequence $\seqn{\nu_n}$ satisfies conditions (a)--(c), a contradiction.

%\medskip
%
%(C)$\Rightarrow$(B): Obvious.

\medskip

(C)$\Rightarrow$(A): Assume that each restricted algebra $\aA_A$ has the Nikodym property and there is no sequence of measures on $\aA\gen{\iI}$ satisfying conditions (a')--(c'), but yet $\aA\gen{\iI}$ does not have the Nikodym property. Let $\seqn{\mu_n}$ be an anti-Nikodym sequence on $\aA\gen{\iI}$. Immediately by definition, $\seqn{\mu_n}$ satisfies conditions (a') and (b'). %By Proposition \ref{prop:ai_uniq_concentr_point}, the point $p_{\iI^*}$ is the only Nikodym concentration point of $\seqn{\mu_n}$. 
For every $A\in\iI$, since $\lim_{n\to\infty}\mu_n(B)=0$ for each $B\le A$ and $\aA_A$ has the Nikodym property, we have
\[\sup_{n\io}\big\|\mu_n\rstr A\big\|=\sup_{n\io}\big\|\mu_n\rstr\aA_A\big\|<\infty,\]
%\[\sup_{n\io}\big\|\mu_n\rstr(\aA\gen{\iI})_A\big\|<\infty,\]
%and so
%\[\sup_{n\io}\big\|\mu_n\rstr A\big\|<\infty,\]
that is, $\seqn{\mu_n}$ satisfies also condition (c'), which is again a contradiction.
\end{proof}

Note that, by the Nikodym property of the restricted algebras, condition (c') in Proposition \ref{prop:ai_nikodym_char}.(c) can be exchanged for the following one:
\textit{\begin{enumerate}[(a*)]\setcounter{enumi}{2}
    \item $\sup_{n\in\omega}\big|\mu_n(A)\big|<\infty$ for every $A\in\iI$.
\end{enumerate}}

\begin{theorem}\label{thm:ai_nikodym_char_pos}
Let $\iI$ be an ideal in a Boolean algebra $\aA$. Then, the algebra $\aA\gen{\iI}$ has the Nikodym property if and only if the restricted algebra $\aA_A$ has the Nikodym property for every $A\in\iI$ and there is no (disjointly supported) sequence $\seqn{\mu_n}$ of non-negative measures on $St(\aA\gen{\iI})$ such that:
\begin{enumerate}
\item $p_{\iI^*}\not\in\supp\big(\mu_n\big)$ for every $n\io$,
\item $\sup_{n\in\omega}\big\|\mu_n\big\|=\infty$,
\item $\lim_{n\to\infty}\mu_n\big(\clopen{A}_{\aA\gen{\iI}}\big)=0$ for every $A\in\iI$.
\end{enumerate}
\end{theorem}
\begin{proof}
Assume first that each restricted algebra $\aA_A$ has the Nikodym property and there is no sequence of measures on $St(\aA\gen{\iI})$ satisfying conditions (1)--(3), but yet $\aA\gen{\iI}$ does not have the Nikodym property. By Proposition \ref{prop:ai_uniq_concentr_point}, there is an anti-Nikodym sequence $\seqn{\nu_n}$ on $\aA\gen{\iI}$ with $p_{\iI^*}$ being its only Nikodym concentration point. By Lemma \ref{lemma:one_Ncp_sNcp}, $p_{\iI^*}$ is a strong Nikodym concentration point of $\seqn{\nu_n}$, and hence, by Theorem \ref{thm:sNcp_disj_supp}, there is a disjointly supported anti-Nikodym sequence $\seqn{\theta_n}$ of non-zero measures on $\aA\gen{\iI}$ such that $p_{\iI^*}\not\in\supp\big(\wh{\theta}_n\big)$ for every $n\io$ and%. By going to a subsequence if necessary, we may assume that
\[\tag{$*$}\lim_{n\to\infty}\big\|\theta_n\big\|=\infty.\]
For each $n\io$ let us define the non-negative measure $\mu_n$ on $St(\aA\gen{\iI})$ by the formula:
\[\mu_n=\big|\wh{\theta}_n\big| \big/ \sqrt{\big\|\wh{\theta}_n\big\|}.\]
Note that $p_{\iI^*}\not\in\supp\big(\mu_n\big)$ for every $n\io$ and that it follows from ($*$) that $\sup_{n\io}\big\|\mu_n\big\|=\infty$. 
%\[\sup_{n\in\omega}\big\|\mu_n\big\|=\sup_{n\in\omega}\Big(\big\|\wh{\theta}_n\big\|\big/\sqrt{\big\|\wh{\theta}_n\big\|}\Big)=\sup_{n\in\omega}\sqrt{\big\|\wh{\theta}_n\big\|}=\sup_{n\in\omega}\sqrt{\big\|\theta_n\big\|}=\infty.\]

Let $A\in\iI$. For every $n\io$ we have
\[ \mu_n\big(\clopen{A}_{\aA\gen{\iI}}\big)=\big|\theta_n\big|(A)\big/{\sqrt{\big\|\wh{\theta}_n\big\|}}\leq \Big(\sup_{k\in\omega}\big\|\theta_k\rstr A\big\|\Big)\big/\sqrt{\big\|\wh{\theta}_n\big\|}=\Big(\sup_{k\in\omega}\big\|\theta_k\rstr \aA_A\big\|\Big)\big/\sqrt{\big\|\wh{\theta}_n\big\|},\]
hence, by ($*$) and the fact that the restricted algebra $\aA_A$ has the Nikodym property, it also holds
\[\lim_{n\to\infty}\mu_n\big(\clopen{A}_{\aA\gen{\iI}}\big)=0.\]
Consequently, the sequence $\seqn{\mu_n}$ satisfies conditions (1)--(3), which is a contradiction.

\medskip

Assume now that $\aA\gen{\iI}$ has the Nikodym property. Trivially, for every $A\in\iI$ the Boolean algebra $\aA_A$ has the Nikodym property, too. Assume however that there is a sequence $\seqn{\mu_n}$ of non-negative measures on $St(\aA\gen{\iI})$ satisfying conditions (1)--(3). For each $n\io$ let us define the measure $\nu_n$ on $\aA\gen{\iI}$ via its Radon extension $\wh{\nu}_n$ on $St(\aA\gen{\iI})$ as follows:
\[\wh{\nu}_n = \mu_n-\mu_n\big(St(\aA\gen{\iI})\big)\cdot\delta_{p_{\iI^*}}.\]

The sequence $\seqn{\nu_n}$ satisfies conditions (a')--(c') of Proposition \ref{prop:ai_nikodym_char}.(c). Indeed, by condition (1) for every $n\io$ we have 
\[\big\|\nu_n\big\|=\big\|\wh{\nu}_n\big\|=\big\|\mu_n\big\|+\big|\mu_n\big(St(\aA\gen{\iI})\big)\big|\ge\big\|\mu_n\big\|,\]
so by condition (2) we get
\[\sup_{n\io}\big\|\nu_n\big\|\ge\sup_{n\io}\big\|\mu_n\big\|=\infty,\]
hence condition (a') is satisfied. Condition (b') holds, too, as
\[\nu_n\big(1_\aA\big)=\wh{\nu}_n\big(St(\aA\gen{\iI})\big)=\mu_n\big(St(\aA\gen{\iI})\big)-\mu_n\big(St(\aA\gen{\iI})\big)\cdot1=0,\]
so $\sup_{n\io}\big|\nu_n\big(1_\aA\big)\big|<\infty$. Finally, by condition (3), for every $A\in\iI$ and $B\le A$ we have
\[\lim_{n\to\infty}\nu_n(B)=\lim_{n\to\infty}\mu_n\big(\clopen{B}_{\aA\gen{\iI}}\big)=0,\]
hence, by the fact that $\aA_A$ has the Nikodym property, we get that
\[\sup_{n\io}\big\|\nu_n\rstr A\big\|=\sup_{n\io}\big\|\nu_n\rstr\aA_A\big\|<\infty,\]
%\[\sup_{n\io}\big\|\nu_n\rstr(\aA\gen{\iI})_A\big\|<\infty,\]
%so
%\[\sup_{n\io}\big\|\nu_n\rstr A\big\|<\infty,\]
that is, condition (c') is satisfied as well.

Consequently, by Proposition \ref{prop:ai_nikodym_char}, $\aA\gen{\iI}$ does not have the Nikodym property, which is again a contradiction.
\end{proof}

\section{Definable Boolean algebras with the Nikodym property and without the Grothendieck property\label{sec:definable}}

In this section, exploiting the results from the previous one, we will study the Nikodym property for Boolean algebras of the form $\wo\gen{\fF}$, where $\fF$ is a free filter on $\omega$. As a result, we will obtain a large class of simple Boolean subalgebras of $\wo$ having the Nikodym property but not the Grothendieck property. Those subalgebras, when considered as subsets of the Cantor space $\Cantor$, will also have an interesting additional feature: they will belong to a low level of the Borel hierarchy of subsets of $\Cantor$.

\medskip

From now on, all considered filters on $\omega$ are assumed to be \textit{free}. Let $\fF$ be a filter on $\omega$. Following \cite[Section 3]{MS24}, by $\aA_{\fF}$ we shortly denote the Boolean algebra $\wo\gen{\fF}$. Therefore,
\[\aA_{\fF}=\fF\cup\fF^*=\big\{A\in\wo\colon\ A\in \fF\text{ or }A^c\in \fF\big\}.\]
%endowed with the standard set-theoretic operations. Trivially, $\aA_F$ is a Boolean subalgebra of $\wo$, 
Trivially, $\aA_{\fF}$ contains the ideal $Fin$ and $\fF$ is an ultrafilter on $\aA_{\fF}$. Also, we have $\aA_{\fF}=\wo$ if and only if $\fF$ is an ultrafilter on $\wo$. 

Again following \cite[Section 3]{MS24}, we put $S_{\fF}=St\big(\aA_\fF\big)$. For every $A\in\aA_{\fF}$ we will shortly write $\clopen{A}_{\fF}=\clopen{A}_{\aA_{\fF}}$. Note that $S_{\fF}$ contains a countable discrete dense subspace consisting of isolated points which we can naturally associate with $\omega$, and therefore we can put $S_{\fF}^*=S_{\fF}\sm\omega$. For $A\in\aA_{\fF}$ we also write $\clopen{A}_{\fF}^*=\clopen{A}_{\fF}\sm\omega$. Note that $\tilde{\fF}=\big\{A^\bullet\colon\ A\in\fF\big\}$ is a filter in $\wo/Fin$ such that $\aA_{\fF}/Fin=\tilde{\fF}\cup\tilde{\fF}^*$. One can easily show that $S_{\fF}^*$ is homeomorphic to the Stone space $St\big(\aA_{\fF}/Fin\big)$ (identifying the points $p_\fF$ and $p_{\tilde{\fF}}$), or, equivalently, that the Boolean algebra of clopen subsets of $S_{\fF}^*$ is isomorphic to $\aA_{\fF}/Fin$. For every $A\in\wo$ we will also simply write $A$ for the corresponding subset of $\omega\sub S_{\fF}$. If $A\in\aA_{\fF}$, then $\ol{A}^{S_{\fF}}=\clopen{A}_{\fF}$, and conversely, if $A\in\wo$ is such that $\ol{A}^{S_{\fF}}$ is clopen, then $A\in\aA_{\fF}$ (and hence again $\ol{A}^{S_{\fF}}=\clopen{A}_{\fF}$). 
%There is also a special unique point $p_F\in S_F$ such that for every $A\in\aA_F$, $p_F\in\clopen{A}_F$ if and only if $A\in F$. Formally, of course, $F=p_F$, but for the sake of precision and to focus attention we will use the symbol $F$ when we will talk about the \textit{filter} on $\omega$, and $p_F$ when we will relate to the \textit{point} in $S_F$. 
Note that for every infinite set $A\in\fF^*$ we have $\big(\aA_\fF\big)_{A}=\wp(A)$ (resp. $\big(\aA_\fF\big)_A/[A]^{<\omega}=\wp(A)/[A]^{<\omega}$) and so the clopen $\clopen{A}_{\fF}$ (resp. $\clopen{A}_{\fF}^*$) is homeomorphic to $\bo$ (resp. to $\omega^*$); consequently the Boolean algebras $\big(\aA_\fF\big)_{A}$ and $\big(\aA_\fF\big)_{A}/[A]^{<\omega}$ have the Nikodym property (by Seever's theorem \cite{See68}). 
%Note that every point $x\in S_F^*\sm\big\{p_F\big\}$ has a clopen neighborhood $U$ in $S_F$ (resp. in $S_F^*$) not containing $p_F$ and homeomorphic to $\bo$ (resp. to $\omega^*$).

Recall that $p_\fF$ denotes the point in $S_\fF^*$ corresponding to the ultrafilter $\fF$ in the algebra $\aA_\fF$. Let us define the following set:
\[N_{\fF}=\omega\cup\big\{p_{\fF}\big\}.\]
The topology of $N_{\fF}$ inherited from $S_{\fF}$ can be described as follows: every point of $\omega$ is isolated in $N_{\fF}$ (as it is isolated in $S_{\fF}$) and a local open base of $p_\fF$ in $N_\fF$ consists of all (clopen) sets of the form $A\cup\big\{p_{\fF}\big\}$, where $A\in\fF$. Note that every open neighborhood of $p_{\fF}$ in $N_{\fF}$ is a clopen subset of $N_{\fF}$.

For basic properties of spaces $S_{\fF}$ and $N_{\fF}$, see \cite[Section 3]{MS24}. The following proposition lists the most important ones.

\begin{proposition}\label{prop:sf_properties}
    Let $\fF$ be a filter on $\omega$.
    \begin{enumerate}
        \item If $\mathfrak{F}$ denotes the subset of $\bo$ consisting of all ultrafilters $x\in\bo$ such that $\fF\sub x$, then $\mathfrak{F}$ is closed in $\bo$ and the mapping $\varphi\colon\bo/\mathfrak{F}\to S_{\fF}$ given for every $x\in\bo$ by the formula $\varphi\big([x]_{\mathfrak{F}}\big)=x\cap\aA_\fF$, where $[x]_{\mathfrak{F}}$ denotes the equivalence class of $x$ in the quotient space $\bo/\mathfrak{F}$, is a homeomorphism.
        \item The following are equivalent:
        \begin{itemize}
            \item $\fF=Fr$,
            \item $S_{\fF}^*=\big\{p_{\fF}\big\}$,
            \item $N_{\fF}=S_{\fF}$.
        \end{itemize}
        \item $S_{\fF}\cong\bo$ if and only if $\fF$ is an ultrafilter on $\wo$.
        \item The following are equivalent:
        \begin{itemize}
            \item $\fF$ is a P-filter on $\omega$,
	    \item $\fF$ is a maximal P-filter on $\aA_{\fF}$,
	    \item $p_\fF$ is a P-point in $S_{\fF}^*$.
        \end{itemize}
        \item the space $N_{\fF}$ is $C^*$-embedded in $S_{\fF}$, and $S_{\fF}=\beta\big(N_{\fF}\big)$, i.e. $S_{\fF}$ is the \v{C}ech--Stone compactification of $N_{\fF}$.%\footnote{\color{red}napisac co to znaczy \textit{jest}!!!!!!!!!!!!!!!!!!!!!!!!!!!!!!!!!!}.
    \end{enumerate}
\end{proposition}

%{\color{red} Chyba niepotrzebne! :}
%\begin{proposition}\label{prop:sf_an_char}
%Let $\fF$ be a filter on $\omega$ and let $\aA=\aA_{\fF}$ or $\aA=\aA_{\fF}/Fin$. Then, $\aA$ has the Nikodym property if and only if there exists a sequence $\seqn{\mu_n}$ of measures on $St(\aA)$ satisfying the following three conditions:
%\begin{enumerate}
%	\item $\sup_{n\in\omega}\big\|\mu_n\big\|=\infty$,
%	\item $\lim_{n\to\infty} \mu_n(St(\aA))=0$,
%	\item if $\aA=\aA_{\fF}$, then $\sup_{n\in\omega}\big\|\mu_n\rstr\clopen{\omega\sm A}_{\fF}\big\|< \infty$ for every $A\in\fF$, or,\\ if $\aA=\aA_{\fF}/Fin$, then $\sup_{n\in\omega}\big\|\mu_n\rstr\clopen{\omega\sm A}_{\fF}^*\big\|<\infty$ for every $A\in\fF$.
%\end{enumerate}
%\end{proposition}

The following corollary is an immediate consequence of Theorem \ref{thm:ai_nikodym_char_pos} for Boolean algebras $\aA_\fF$ and $\aA_\fF/Fin$, we will use it frequently throughout this section.

\begin{corollary}\label{cor:sf_an_char_pos}
Let $\fF$ be a filter on $\omega$ and let $\aA=\aA_{\fF}$ or $\aA=\aA_{\fF}/Fin$. Then, $\aA$ has the Nikodym property if and only if there is no (disjointly supported) sequence $\seqn{\mu_n}$ of non-negative measures on $St(\aA)$ such that:
\begin{enumerate}
\item $p_{\fF}\not\in\supp\big(\mu_n\big)$ for every $n\io$,
\item $\sup_{n\in\omega}\mu_n(St(\aA))=\infty$,
\item if $\aA=\aA_{\fF}$, then $\lim_{n\to\infty}\mu_n\big(\clopen{\omega\sm A}_{\fF}\big)=0$ for every $A\in\fF$, or,\\ if $\aA=\aA_{\fF}/Fin$, then $\lim_{n\to\infty}\mu_n\big(\clopen{\omega\sm A}_{\fF}^*\big)=0$ for every $A\in\fF$.
\end{enumerate}
\end{corollary}

The corollary and its variant for spaces $N_\fF$, presented in \cite[Theorem 4.1]{Zuc25}, are crucial for proving the following characterization of the Nikodym property of Boolean algebras $\aA_\fF$ as being a conjunction of the Nikodym property of the algebras $\aA_\fF/Fin$ and of the spaces $N_\fF$.

\begin{theorem}\label{thm:sf_nik_nf_sfs}
Let $\fF$ be a filter on $\omega$. Then, the algebra $\aA_{\fF}$ has the Nikodym property if and only if the algebra $\aA_{\fF}/Fin$ has the Nikodym property and the space $N_{\fF}$ has the finitely supported Nikodym property.
\end{theorem}
\begin{proof}
Assume first that $\aA_\fF$ has the Nikodym property. Since the property is preserved by taking quotients (see \cite[Proposition 2.11.(b)]{Sch82}), the algebra $\aA_\fF/Fin$ has it as well. If the space $N_{\fF}$ did not have the finitely supported Nikodym property, then $\aA_{\fF}$ would not have the Nikodym property by \cite[Theorem 8.2]{Zuc25}, as $S_{\fF}$ contains a homeomorphic copy of $N_{\fF}$. Hence, $N_\fF$ has the required property, too.

Suppose now that $\aA_{\fF}$ does not have the Nikodym property. Let $\seqn{\mu_n}$ be a sequence of non-negative measures on $S_{\fF}$ like in Corollary \ref{cor:sf_an_char_pos} (for $\aA=\aA_\fF$). If $\sup_{n\io} \mu_n\big(S_{\fF}^*\big)=\infty$, then for each $n\io$ we define the measure $\nu_n=\mu_n\rstr S_{\fF}^*$, and it is not difficult to see that $\seqn{\nu_n}$ is a sequence of non-negative measures on $S_{\fF}^*$ like in Corollary \ref{cor:sf_an_char_pos} (for $\aA=\aA_\fF/Fin$), which means that the algebra $\aA_\fF/Fin$ does not have the Nikodym property either.

Thus, let us assume that $\sup_{n\io} \mu_n\big(S_{\fF}^*\big)<\infty$, which simply implies that $\sup_{n\io} \mu_n(\omega)=\infty$.
If for each $n\io$ we set $\theta_n=\mu_n\rstr\omega$, then $\theta_n$ is a non-negative measure on $S_{\fF}$ such that %$\supp\big(\theta_n\big)\sub\omega$, 
$\sup_{n\io}\theta_n(\omega)=\infty$ and for every $A\in\fF$ we have
\[\lim_{n\to\infty}\theta_n(\omega\sm A)=\lim_{n\to\infty}\mu_n\big(\clopen{\omega\sm A}_{\fF}\big)=0.\]

By the continuity of the measures, for every $n\io$ there exists a finite set $F_n\sub\omega$ such that
\[\tag{$*$}\theta_n\big(\omega\sm F_n\big) < 1/n.\]
For each $n\io$ let us set $\nu_n=\theta_n\rstr F_n$. We claim that the finitely supported sequence $\seqn{\nu_n}$ of non-negative measures on $N_\fF$ satisfies conditions (1)--(3) from \cite[Theorem 4.1]{Zuc25}. First, for every $n\io$ we have $\supp\big(\nu_n\big)\sub\omega$ and, by ($*$),
\[\nu_n(\omega) = \theta\big(F_n\big)= \theta_n(\omega) - \theta_n\big(\omega\sm F_n\big) > \theta_n(\omega) - 1/n,\]
so $\sup_{n\io}\nu_n(\omega)=\infty$.

Next, for every $A\in\fF$ we have $\lim_{n\to\infty}\nu_n(\omega\sm A)=0$, as $\lim_{n\to\infty}\theta_n(\omega\sm A)=0$ and, again by ($*$),
\[\lim_{n\to\infty}\big|\nu_n(\omega\sm A) - \theta_n(\omega\sm A)\big|\leq\lim_{n\to\infty}\big\|\nu_n - \theta_n\big\|=\lim_{n\to\infty}\big\|\theta_n\rstr \big(\omega\sm F_n\big)\big\|=\lim_{n\to\infty}\theta_n\big(\omega\sm F_n\big)=0. \]
Therefore, by \cite[Theorem 4.1]{Zuc25}, the space $N_{\fF}$ does not have the finitely supported Nikodym property.
\end{proof}

%As a consequence of Corollary \ref{cor:sf_an_char_pos} and Theorem \ref{thm:sf_nik_nf_sfs} we get the following useful result.
%
%\begin{proposition}\label{prop:nikodym_katetov}
%    Let $\iI$ and $\jJ$ be ideals on $\omega$. Assume that $\iI$ has the Nikodym property. If $\iI\le_K\jJ$ (in particular, if $\iI\sub\jJ$), then $\jJ$ has the Nikodym property, too.
%\end{proposition}
%\begin{proof}
%    Suppose for the sake of contradiction that $\jJ$ does not have the Nikodym property, so by Theorem \ref{thm:sf_nik_nf_sfs} the algebra $\aA_{\jJ^*}/Fin$ does not have the Nikodym property or the space $N_{\jJ^*}$ does not have the finitely supported Nikodym property. If the latter happens, then by \cite[Proposition 5.2]{Zuc25} the space $N_{\iI^*}$ does not have the finitely supported Nikodym property as well, which is impossible, again by Theorem \ref{thm:sf_nik_nf_sfs}. It follows that $N_{\jJ^*}$ has the finitely supported Nikodym property.
%    
%    We get then that $\aA_{\jJ^*}/Fin$ does not have the Nikodym property. Let thus $\seqn{\mu_n}$ be a sequence of non-negative disjointly supported measures on $St\big(A_{\jJ^*}/Fin\big)$ as in Corollary \ref{cor:sf_an_char_pos}. 
%\end{proof}

Note that an ideal $\iI$ on $\omega$ has the Nikodym property as an ideal in the Boolean algebra $\wo$ (i.e. in the sense of Definition \ref{def:nik_ideal}) if and only if $\iI$ has the Nikodym property as a ring of subsets of $\omega$ (i.e. in the sense of Definition \ref{def:nikodym_ring})---this follows from the fact that the rings $\rR(\iI)$ and $\iI$ are isomorphic. Also, recall that by Theorem \ref{thm:nik_prop_ideal_algebra} an ideal $\iI$ on $\omega$ has the Nikodym property if and only if the Boolean algebra $\aA_{\iI^*}$ has the Nikodym property. Since the Boolean algebra $\wo$ has the Nikodym property, it follows that if an ideal $\iI$ on $\omega$ is maximal (that is, $\iI^*$ is an ultrafilter and so $\aA_{\iI^*}=\wo$), then $\iI$ has the Nikodym property.

\begin{corollary}\label{cor:maximal_ideal_nik}
    If $\iI$ is a maximal ideal on $\omega$, then $\iI$ has the Nikodym property.
\end{corollary}

Let us now turn our attention to P-ideals. We start with the following remark.

\begin{example}\label{example:nonppoint_nikodym}
    In \cite[Example 5.6]{DP00} an ideal on $\omega$ which is not a P-ideal but has the Nikodym property was constructed. It immediately follows from the discussion above that every maximal ideal $\iI$ on $\omega$ whose dual filter $\iI^*$ is not a P-point in $\omega^*$ is also such an ideal. Since there are $2^\frakc$ many non-isomorphic ultrafilters on $\omega$ which are not P-points in $\omega^*$ (see, e.g., \cite[page 76]{Jec03}), we actually get $2^\frakc$ many non-isomorphic such examples of ideals.
\end{example}

In the case of a P-ideal $\iI$ on $\omega$, Theorem \ref{thm:sf_nik_nf_sfs} yields that the Nikodym property of the algebra $\aA_{\iI^*}$ is equivalent to the finitely supported Nikodym property of the space $N_{\iI^*}$ as well as to the following two properties of $\iI$, see Theorem \ref{thm:p_ideal_nik_equivalences} below.

\begin{definition}\label{def:nikodym_set}
Let $\rR$ be a ring of subsets of a set $X$. A subset $\qQ\sub\rR$ is a \textit{Nikodym set for $\ba(\rR)$} if, for every sequence $\seqk{\mu_k}$ in $\ba(\rR)$, if $\lim_{k\io}\mu_k(A)=0$ for every $A\in\qQ$, then $\sup_{k\io}\big\|\mu_k\big\|<\infty$. 
\end{definition}

\begin{definition}\label{def:web}
Let $\rR$ be a ring of subsets of a set $X$. A collection $\seq{\rR_\sigma}{\sigma\in\omega^{<\omega}}$ of subsets of $\rR$ with  $\rR=\rR_{\langle\rangle}$ and such that for every $\sigma\in\omega^{<\omega}$ the sequence $\seqn{\rR_{\sigma\concat n}}$ is increasing and satisfies $\rR_\sigma=\bigcup_{n\io}\rR_{\sigma\concat n}$ is called a \textit{web in} $\rR$.
\end{definition}

\begin{definition}\label{def:strong_web_nik}
A ring $\rR$ of subsets of a set $X$ has the \textit{strong Nikodym property} if for every increasing sequence $\seqn{\rR_n}$ of subsets of $\rR$ with $\rR=\bigcup_{n\io}\rR_n$ there is $n_0\io$ such that $\rR_{n_0}$ is a Nikodym set for $\ba(\rR)$.

A ring $\rR$ of subsets of a set $X$ has the \textit{web Nikodym property} if for every web $\seq{\rR_\sigma}{\sigma\in\omega^{<\omega}}$ in $\rR$ there is a sequence $\seqm{\sigma_m\in \omega^{m}}$ such that, for every $m\io$, $\sigma_{m+1}$ extends $\sigma_m$ and $\rR_{\sigma_m}$ is a Nikodym set for $\ba(\rR)$.
\end{definition}

Both of the properties of course immediately imply the standard Nikodym property for rings. The strong Nikodym property and the web Nikodym property were first essentially studied by Valdivia \cite{Val79} and L\'opez-Pellicer \cite{LP97}, respectively, who proved that $\sigma$-fields of sets have both the properties (see also \cite{KLP17} and \cite{LAMM}; cf. \cite[Chapter 21]{KKLPS}). Later, Valdivia \cite{Val13} showed that the field $\mathscr{J}_n$ of Jordan measurable subsets of the $n$-dimensional cube $[0,1]^n$ has the strong Nikodym property, which was further extended by L\'opez-Alfonso \cite{LA16} to the web Nikodym property (consequently, both results strengthen Schachermayer's theorem \cite{Sch82}, mentioned in Introduction). Ferrando \cite{Fer02} proved that the ideal $\zZ$ has the latter property, too. For further results on the web Nikodym property for ideals, see also \cite{FLALP}.

We translate Definition \ref{def:strong_web_nik} in a natural way to the class of Boolean algebras and ideals.

\begin{definition}\label{def:strong_web_nik_alg}
    Let $\aA$ be a Boolean algebra. $\aA$ has the \textit{strong Nikodym property} (resp. the \textit{web Nikodym property}) if the ring $Clopen(St(\aA))$ of all clopen subsets of $St(\aA)$ has the strong Nikodym property (resp. the web Nikodym property). 
    
    An ideal $\iI$ in $\aA$ has the \textit{strong Nikodym property} (resp. the \textit{web Nikodym property}) if the ring $\rR(\iI)$ of clopen subsets of $St(\aA)$ has the strong Nikodym property (resp. the web Nikodym property). 
\end{definition}

\begin{lemma}\label{lemma:nikodym_set}
    Let $\iI$ be an ideal on $\omega$ such that $\aA_{\iI^*}$ has the Nikodym property. Let $\qQ\sub\rR(\iI)$ be non-empty and set $\qQ^*=\big\{S_{\iI^*}\sm A\colon A\in\qQ\big\}$. Then, the following statements are equivalent:
    \begin{enumerate}[(a)]
    \item $\qQ$ is Nikodym for $\ba(\rR(\iI))$,
    \item $\qQ\cup\big\{S_{\iI^*}\!\big\}$ is Nikodym for $\ba\!\big(Clopen\big(S_{\iI^*}\big)\big)$,
    \item $\qQ\cup\qQ^*$ is Nikodym for $\ba\!\big(Clopen\big(S_{\iI^*}\big)\big)$.
    \end{enumerate}
\end{lemma}
\begin{proof}
    (a)$\Rightarrow$(b): Assume that $\qQ\sub\rR(\iI)$ is a Nikodym set for $\ba(\rR(\iI))$. Let $\seqn{\mu_n}$ be a sequence in $\ba\!\big(Clopen\big(S_{\iI^*}\big)\big)$ such that $\lim_{n\to\infty}\mu_n(A)=0$ for every $A\in\qQ\cup\big\{S_{\iI^*}\!\big\}$. It follows that
    \[\sup_{n\io}\big\|\mu_n\rstr\rR(\iI)\big\|<\infty,\]
    hence for every $A\in\rR(\iI)$ we have $\sup_{n\io}\big|\mu_n(A)\big|<\infty$, and so
    \[\sup_{n\io}\big|\mu_n(A^c)\big|\le\sup_{n\io}\big|\mu_n\big(S_{\iI^*}\big)\big|+\sup_{n\io}\big|\mu_n(A)\big|<\infty.\]
    %{\color{red} Uzywamy tu rownowaznej definicji NP, czy ona gdzies jest sformulowana na poczatku?} 
    As $\aA_{\iI^*}$ has the Nikodym property, we get that $\sup_{n\io}\big\|\mu_n\big\|<\infty$. This proves that the set $\qQ\cup\big\{S_{\iI^*}\!\big\}$ is Nikodym for $\ba\!\big(Clopen\big(S_{\iI^*}\big)\big)$. %This proves implication (1)$\Rightarrow$(2).

    (b)$\Rightarrow$(c): Let $\seqn{\mu_n}\sub\ba\!\big(Clopen\big(S_{\iI^*}\big)\big)$ be such that $\lim_{n\to\infty}\mu_n(A)=0$ for every $A\in\qQ\cup\qQ^*$. Let $A\in\qQ$. Then,
    \[\lim_{n\to\infty}\mu_n\big(S_{\iI^*}\big)=\lim_{n\to\infty}\big(\mu_n(A)+\mu_n\big(S_{\iI^*}\sm A\big)\big)=0.\]
    As $\qQ\cup\big\{S_{\iI^*}\big\}$ is Nikodym for $\ba\!\big(Clopen\big(S_{\iI^*}\big)\big)$, we get that $\sup_{n\to\infty}\big\|\mu_n\big\|<\infty$. Consequently, $\qQ\cup\qQ^*$ is Nikodym for $\ba\!\big(Clopen\big(S_{\iI^*}\big)\big)$, too.

    (c)$\Rightarrow$(a) The implication follows by a similar argument as in the proof of Theorem \ref{thm:nik_prop_ideal_algebra}, using the extension operator $T_0\colon\ba(\rR(\iI))\to\ba\!\big(Clopen\big(S_{\iI^*}\big)\big)$.
\end{proof}

\begin{theorem}\label{thm:p_ideal_nik_equivalences}
    Let $\iI$ be a P-ideal on $\omega$. Then, the following conditions are equivalent:
    \begin{enumerate}
        \item $\aA_{\iI^*}$ has the Nikodym property,
        \item $\aA_{\iI^*}$ has the strong Nikodym property,
        \item $\aA_{\iI^*}$ has the web Nikodym property,
        \item $\iI$ has the Nikodym property,
        \item $\iI$ has the strong Nikodym property,
        \item $\iI$ has the web Nikodym property,
        \item $N_{\iI^*}$ has the finitely supported Nikodym property.
    \end{enumerate}
\end{theorem}
\begin{proof}
Implications (3)$\Rightarrow$(2)$\Rightarrow$(1) and (6)$\Rightarrow$(5)$\Rightarrow$(4) are immediate.

Equivalence (1)$\Leftrightarrow$(4) follows from Theorem \ref{thm:nik_prop_ideal_algebra}.

Using the nomenclature of \cite{FLALP}, an ideal $\jJ$ on $\omega$ is a P-ideal if and only if the family $\fso$ is \textit{$\jJ$-singular}. Consequently, by \cite[Theorem 3.3]{FLALP}, if the ideal $\iI$ has the Nikodym property, then $\iI$ has the web Nikodym property, i.e. implication (4)$\Rightarrow$(6) holds.

We show implication (6)$\Rightarrow$(3). Let $\seq{\rR_\sigma}{\sigma\in\omega^{<\omega}}$ be a web in $Clopen\big(S_{\iI^*}\big)$ (with $\rR_{\langle\rangle}=Clopen\big(S_{\iI^*}\big)$). Since the sequences $\seqn{\rR_{\sigma\concat n}}$ are increasing, without loss of generality we may assume that for each $\sigma\in\omega^{<\omega}$ we have $S_{\iI^*}\in\rR_\sigma$. For each $\sigma\in\omega^{<\omega}$ set $\rR_\sigma'=\rR_\sigma\cap\rR(\iI)$. It follows that the collection $\seq{\rR_\sigma'}{\sigma\in\omega^{<\omega}}$ is a web in $\rR(\iI)$ (with $\rR_{\langle\rangle}'=\rR(\iI)$). As $\iI$ has the web Nikodym property, and so by definition the (isomorphic) ring $\rR(\iI)$ has the web Nikodym property, too, there is a sequence $\seqm{\sigma_m\in\omega^m}$ such that, for every $m\io$, $\sigma_{m+1}$ extends $\sigma_m$ and the set $\rR_{\sigma_m}'$ is Nikodym for $\ba(\rR(\iI))$. Note that by already established implications (6)$\Rightarrow$(4)$\Rightarrow$(1) the algebra $\aA_{\iI^*}$ has the Nikodym property, so by implication (a)$\Rightarrow$(b) of Lemma \ref{lemma:nikodym_set} for each $m\io$ the set $\rR_{\sigma_m}$ is Nikodym for $\ba\!\big(Clopen\big(S_{\iI^*}\big)\big)$, as $S_{\iI^*}\in\rR_{\sigma_m}$. Consequently, the ring $Clopen\big(S_{\iI^*}\big)$ has the web Nikodym property and so the algebra $\aA_{\iI^*}$ does, too.

We now prove that implication (3)$\Rightarrow$(6) holds. Let $\seq{\rR_\sigma}{\sigma\in\omega^{<\omega}}$ be a web in $\rR(\iI)$ (with $\rR_{\langle\rangle}=\rR(\iI)$); again, without loss of generality we may assume that $\emptyset\in\rR_\sigma$ for every $\sigma\in\omega^{<\omega}$. For every $\sigma\in\omega^{<\omega}$ set $\rR_\sigma^*=\big\{S_{\iI^*}\sm A\colon A\in\rR_\sigma\big\}$ and $\sS_\sigma=\rR_\sigma\cup\rR_\sigma^*$. Then, $\seq{\sS_\sigma}{\sigma\in\omega^{<\omega}}$ is a web in $Clopen\big(S_{\iI^*}\big)$ (with $\sS_{\langle\rangle}=Clopen\big(S_{\iI^*}\big)$). As $\aA_{\iI^*}$ has the web Nikodym property, there is a sequence $\seqm{\sigma_m\in\omega^m}$ such that, for every $m\io$, $\sigma_{m+1}$ extends $\sigma_m$ and the set $\sS_{\sigma_m}$ is Nikodym for $\ba\big(Clopen\big(S_{\iI^*}\big)\big)$. By implication (c)$\Rightarrow$(a) of Lemma \ref{lemma:nikodym_set}, for each $m\io$ the set $\rR_{\sigma_m}$ is Nikodym for $\ba(\rR(\iI))$, which proves that $\iI$ has the web Nikodym property as well.

Implication (1)$\Rightarrow$(7) follows from Theorem \ref{thm:sf_nik_nf_sfs}. 

For implication (7)$\Rightarrow$(1), assume that $\aA_{\iI^*}$ does not have the Nikodym property. It follows by (7) and Theorem \ref{thm:sf_nik_nf_sfs} that the algebra $\aA_{\iI^*}/Fin$ does not have the Nikodym property, so let $\seqn{\mu_n}$ be an anti-Nikodym sequence on $\aA_{\iI^*}/Fin$. As $\iI$ is a P-ideal, Proposition $\ref{prop:sf_properties}.(4)$ implies that $p_{\iI^*}$ is a P-point in the space $S_{\iI^*}^*$. Thus, by \cite[Proposition 4.5]{Sob19}, $p_{\iI^*}$ cannot be a Nikodym concentration point of $\seqn{\mu_n}$. Consequently, $\seqn{\mu_n}$ has a concentration point $t\in S_{\iI^*}^*$ different from $p_{\iI^*}$. Let $A\in\iI$ be such that $t\in\clopen{A}_{\iI^*}^*$. Note that for the equivalence class $A^\bullet$ of $A$ in the quotient algebra $\aA_{\iI^*}/Fin$ we have $\big(\aA_{\iI^*}/Fin\big)_{A^\bullet}=\wp(A)/[A]^{<\omega}$. Consequently, the sequence $\seqn{\mu_n\rstr\wp(A)/[A]^{<\omega}}$ is an anti-Nikodym sequence on $\wp(A)/[A]^{<\omega}$, which is impossible as the latter algebra has the Nikodym property by the Seever theorem \cite{See68}.
\end{proof}

Theorem \ref{thm:p_ideal_nik_equivalences} will be frequently used in the sequel; for its particular applications to the class of so-called density ideals on $\omega$, see Corollaries \ref{cor:not_lek_z_web_nik} and \ref{cor:density_nik_strong_web}.

\begin{example}\label{example:nonppoint_web_nikodym}
    A careful reader has surely noticed that the argument for implication (3)$\Rightarrow$(6) is superfluous in the proof of Theorem \ref{thm:p_ideal_nik_equivalences}, as this implication follows from combinations of other ones, whose proofs are also given. We provided it however deliberately in order to show that the assumption that the considered ideal $\iI$ is a P-ideal is actually not required for the validity of equivalences (3)$\Leftrightarrow$(6), (1)$\Leftrightarrow$(4), and (2)$\Leftrightarrow$(5) (which can be proved in a similar but easier way as (3)$\Leftrightarrow$(6)). Based on that, we obtain, e.g., that every maximal ideal $\iI$ on $\omega$ has the web Nikodym property, since for such $\iI$ we have $\aA_{\iI^*}=\wo$ and the $\sigma$-field $\wo$ has the latter property (by L\'opez-Pellicer \cite{LP97}), strengthening Corollary \ref{cor:maximal_ideal_nik}. In particular, by an argument as in Example \ref{example:nonppoint_nikodym}, we get that there are $2^\frakc$ many non-isomorphic maximal ideals with the web Nikodym property (and which are not P-ideals).
\end{example}

\begin{example}
    However, in general, Theorem \ref{thm:p_ideal_nik_equivalences} does not hold anymore if one drops the assumption that $\iI$ is a P-ideal. Namely, as described in \cite[Sections 4 and 6.3]{MS24}, there are free filters $\fF$ on $\omega$ such that the spaces $S_\fF^*$ contain non-eventually constant convergent sequences, but the spaces $N_\fF$ do not have the bounded Josefson--Nissenzweig property (recall Definition \ref{def:bjnp}). Consequently, for those filters $\fF$, the algebras $\aA_\fF/Fin$, $\aA_\fF$, and the ideals $\fF^*$ do not have the Nikodym property, but, by \cite[Proposition 6.1]{Zuc25}, the spaces $N_\fF$ do have the finitely supported Nikodym property. A similar example is also given in Remark \ref{rem:lacunary}.
\end{example}

By \cite[Proposition 5.2]{Zuc25}, given two filters $\fF$ and $\gG$ on $\omega$ such that $\fF\le_K\gG$, if $N_\gG$ does not have the finitely supported Nikodym property, then $N_\fF$ does not have it either. Our last example shows that this does not hold for the general Nikodym property of ideals.

\begin{example}
    Using again the constructions presented in \cite[Sections 4 and 6.3]{MS24} for each ultrafilter $\fF$ on $\omega$ one can obtain a filter $\gG$ such that $\fF\le_K\gG$ and $S_\gG^*$ contains a non-trivial convergent sequence (simply set $\fF_n=\fF$ for each $n\io$ in the construction). As for each ultrafilter $\fF$ we have $\aA_\fF=\wo$, we consequently get that there are ideals $\iI$ ($=\fF^*$) and $\jJ$ ($=\gG^*$) on $\omega$ such that $\iI\le_K\jJ$ and $\iI$ has the Nikodym property, but $\jJ$ does not.
\end{example}

\subsection{Submeasures and standard classes of ideals on $\omega$\label{sec:nonpath}}

We will now recall several standard definitions concerning submeasures on $\omega$ and ideals associated with them. For basic information concerning those objects, we refer the reader to \cite{BNFP15}, \cite{Far00}, and \cite{Hru17}.

A function $\varphi\colon\wo\to[0,\infty]$ is a \textit{submeasure} on $\omega$ if $\varphi(\emptyset)=0$, $\varphi(\{n\})<\infty$ for every $n\in\omega$, $\varphi(X)\leq\varphi(Y)$ whenever $X\sub Y\sub\omega$, and $\varphi(X\cup Y)\leq\varphi(X)+\varphi(Y)$ for every $X, Y\sub\omega$. A submeasure $\varphi$ on $\omega$ is \textit{lower semicontinuous} (in short, \textit{lsc}) if $\varphi(A)=\lim_{n\to\infty}\varphi(A\cap [0,n])$ for every $A\sub\omega$. %In particular, every non-negative measure $\mu$ on $\omega$ is an lsc submeasure. 
A submeasure $\varphi$ on $\omega$ is \textit{bounded} if $\varphi(\omega)<\infty$, otherwise $\varphi$ is \textit{unbounded}. Note that every non-negative measure $\mu$ on $\omega$ is a bounded finitely additive submeasure.

We consider the following three standard ideals associated with an (lsc) submeasure $\varphi$ on $\omega$:
\[\zZ(\varphi)=\big\{A\sub\omega\colon\ \varphi(A)=0\big\},\]
\[\fin(\varphi)=\big\{A\sub\omega\colon\ \varphi(A)<\infty\big\},\]
\[\exh(\varphi)=\big\{A\sub\omega\colon\ \lim_{n\to\infty} \varphi\big(A\sm [0,n]\big)=0\big\}.\]
Trivially,
\[Fin\cup\zZ(\varphi) \sub \exh(\varphi) \sub \fin(\varphi).\]
It is also not difficult to show that if $\varphi$ is lsc, then $\fin(\varphi)$ is an $\F_{\sigma}$ ideal and $\exh(\varphi)$ is an $\F_{\sigma\delta}$ P-ideal (when, as usual, thought of as subsets of the Cantor space $\Cantor$). The following results of Mazur and Solecki characterize $\F_{\sigma}$ ideals and analytic P-ideals on $\omega$ in terms of submeasures.

\begin{theorem}[Mazur \cite{Maz91}]\label{thm:mazur_submeasures}
Let $\iI$ be an ideal on $\omega$. Then, $\iI$ is an $\F_{\sigma}$ ideal if and only if there is an lsc submeasure $\varphi$ such that $\iI=\fin(\varphi)$.
\end{theorem}

\begin{theorem}[Solecki \cite{Sol99}]\label{thm:solecki_submeasures}
Let $\iI$ be an ideal on $\omega$. Then, $\iI$ is an analytic P-ideal if and only if there is an lsc submeasure $\varphi$ such that $\iI=\exh(\varphi)$.
\end{theorem}

%{\color{red} To jest po prostu nieprawda, przeciez ideal maksymalny ma (N) z tw. 4.10, a nie jest zawarty w ideale klasy $F_\sigma$... Problemem jest to ze te miary nie musza byc lsc}
%
%Immediately by Theorem \ref{thm:mazur_submeasures} we get that every ideal on $\omega$ without the Nikodym property is contained in an $\F_\sigma$ ideal without the Nikodym property, as if $\seqn{\varphi_n}$ is an anti-Nikodym sequence of measures on an ideal $\iI$ then 
%
%\[\iI\sub\fin\Big(\sup_{n\io}|\varphi_n|\Big).\]

For an lsc submeasure $\varphi$ on $\omega$, by the \textit{core} of $\varphi$, denoted $\varphi^\bullet$ (see \cite[Section 1]{DL08ii}), we mean the submeasure defined by
\[\varphi^\bullet(A)=\lim_{n\to\infty} \varphi\big(A\setminus [0,n]\big)\]
for every $A\sub\omega$ (note that $\varphi^\bullet(A)$ is well-defined by the monotonicity of the submeasure $\varphi$). We trivially have
\[\exh(\varphi)=\big\{A\sub\omega\colon\ \varphi^\bullet(A)=0\big\}=\zZ(\varphi^\bullet).\]
Let us note that $\varphi^\bullet(\cdot)$ was denoted by $\|\cdot\|_{\varphi}$ in \cite{FMRS07}.

\medskip

We will now consider various classes of ideals associated with submeasures. For submeasures $\varphi$ and $\psi$ on $\omega$ we write $\psi\leq\varphi$ if $\psi(A)\leq\varphi(A)$ for every $A\sub\omega$. Following Farah \cite[page 21]{Far00}, we call a submeasure $\varphi$ \textit{non-pathological} if for every $A\sub\omega$ we have:
\[ \varphi(A)=\sup \big\{\mu(A)\colon\ \mu\textrm{ is a measure on } \omega \textrm{ such that } \mu\leq\varphi\big\}.\]
Trivially, every non-negative measure on $\omega$ is non-pathological. An ideal $\iI$ on $\omega$ is \textit{non-pathological} if there is a non-pathological lsc submeasure $\varphi$ on $\omega$ such that $\iI=\exh(\varphi)$.

\medskip

The family of density submeasures is an important subclass of non-pathological lsc submeasures. We say that a submeasure $\varphi$ on $\omega$ is a \textit{density submeasure} if there exists a sequence $\seqn{\mu_n}$ of finitely supported non-negative measures on $\omega$ with pairwise disjoint supports contained in $\omega$ and
such that 
\[\varphi = \sup_{n\in\omega}\mu_n.\]
Note that $\varphi^\bullet=\limsup_{n\to\infty}\mu_n$. An ideal $\iI$ on $\omega$ is a \textit{density ideal} if there is a density submeasure $\varphi=\sup_{n\io}\mu_n$ such that $\iI=\exh(\varphi)$; in this case we have $X\in\exh(\varphi)$ if and only if 
$\lim_{n\to\infty}\mu_n(X)=0$.

A prototypical ideal for the class of density ideals is the \textit{(asymptotic) density zero
ideal} $\zZ = \exh(\varphi_d)$, where the \textit{asymptotic density submeasur}e $\varphi_d$ is defined for every $A\sub\omega$ by
\[\varphi_d(A)=\sup_{n\io}\frac{\big|A\cap [2^n,2^{n+1})\big|}{2^n}.\]
It is folklore that
\[\zZ=\bigg\{A\sub\omega\colon\ \lim_{n\to\infty}\frac{|A\cap[0,n)|}{n}=0\bigg\}.\]

\erdos--Ulam ideals are a generalization of the asymptotic density zero ideal, introduced by Just and Krawczyk \cite{JK84}.  %(see e.g. \cite[Section 1.13]{Far00}).
We say that $f\colon\omega\to\R^+$ is an \textit{\erdos--Ulam function} if we have
\[\sum_{n\in\omega} f(n) = \infty\quad\textrm{and}\quad\lim_{n\to\infty} \frac{f(n)}{\sum_{i\leq n} f(i)} = 0.\]
The \textit{\erdos--Ulam ideal} associated with an \erdos--Ulam function $f$ is defined as
\[\mathcal{EU}_f=\bigg\{A\sub\omega\colon\ \lim_{n\to\infty}\frac{\sum_{i\in A\cap[0,n)}f(i)}{\sum_{i\in[0,n)}f(i)}=0\bigg\}.\]
%\exh(\varphi_f)$, where the submeasure $\varphi_f$ is defined by setting
%\[\varphi_f(A) = \frac{{\sum_{{i\leq n, i\in A}} f(i)}}{{\sum_{i\leq n} f(i)}}\]
%for every $A\sub\omega$.
Every \erdos--Ulam ideal is a density ideal (see \cite[Section 1.13]{Far00}). Note that by Tryba \cite[Corollary 3.9]{Try21} an ideal isomorphic to an \erdos--Ulam ideal need not be itself an \erdos--Ulam ideal.

By the above-mentioned folklore equality, $\zZ$ is an \erdos--Ulam ideal. For the function $f\colon\omega\to\R^+$ given by $f(n)=1/(n+1)$ for every $n\io$, we also define the \textit{logarithmic density zero ideal} $\zZ_{\log}=\mathcal{EU}_f$.

A submeasure $\varphi$ on $\omega$ is a \textit{generalized density submeasure} if there exist a partition $\seqn{I_n}$ of $\omega$ into finite sets and a sequence $\seqn{\varphi_n}$ of submeasures on $\omega$ such that 
\[\varphi=\sup_{n\io}\varphi_n\]
and for any $n\io$ and $A\sub\omega$ we have $\varphi_n(A)=\varphi_n\big(A\cap I_n\big)$. In this case we again have $\varphi^\bullet=\limsup_{n\to\infty}\varphi_n$ and it is easy to see that the submeasure $\sup_{n\in\omega}\varphi_n$ is non-pathological if and only if each submeasure $\varphi_n$ is non-pathological. An ideal $\iI$ on $\omega$ is a \textit{generalized density ideal} if
$\iI=\exh(\varphi)$ for some generalized density submeasure $\varphi$ on $\omega$. Of course, every density submeasure (ideal) is a generalized density submeasure (ideal).
 
\medskip

Summable ideals constitute yet another important class of non-pathological ideals. Given $f\colon\omega\to[0,\infty)$ such that $\sum_{n\io}f(n)=\infty$, the \textit{summable ideal} corresponding to $f$ is the ideal
\[\iI_f=\bigg\{A\sub\omega\colon\ \sum_{n\in A} f(n) < \infty\bigg\}.\]
It is easy to see that $\iI_f = \exh\big(\mu_f\big)=\fin\big(\mu_f\big)$, where $\mu_f\colon\wo\to[0,\infty]$ is an unbounded finitely additive submeasure defined for every $A\sub\omega$ by
\[ \mu_f (A) = \sum_{n\in A} f(n).\]
It follows that each summable ideal is an $\F_\sigma$ P-ideal.

\medskip

We will also need two other standard ideals. The first one is $\tr(\nN)$, the \textit{trace of the null ideal}---an ideal on $\seqtwo$ defined as
\[\tr(\nN)= \Big\{A\sub\seqtwo\colon \lambda\big(\big\{x\in\Cantor\colon \exists^\infty n\in\omega \:
(x\restriction n \in A)\big\}\big)=0 \Big\},\]
where $\lambda$ denotes the standard product measure on $\Cantor$. The ideal $\tr(\nN)$, when treated as an ideal on $\omega$, is also of the form $\exh(\varphi)$ for some non-pathological submeasure $\varphi$ on $\omega$ (see e.g. \cite[page 1272]{BNFP15}). For more information on $\tr(\nN)$, see \cite{HHH07}.

Finally, we will need the ideal $\conv$, defined as the ideal on $\Q\cap [0, 1]$ generated by sequences in $\Q\cap [0, 1]$ convergent in $[0, 1]$. This ideal is of different type from the ideals associated with submeasures, as it is an $\F_{\sigma\delta\sigma}$ ideal. For more information on $\conv$, see \cite{Hru17}.

%\subsection{The Nikodym property of the algebra $\aA_{\fF}$ vs. properties of the space $N_{\fF}$}

%\begin{lemma}
%Let $\aA$ be a Boolean algebra. If $\seqn{\mu_n}$ is an anti-Nikodym sequence of measures on $St(\aA)$ with countable supports, i.e. such that $\|\mu_n\|\to\infty$, $\mu_n(\clopen{A}_\aA)\to 0$ for every $A\in\aA$ and $\supp(\mu_n)$ is countable for every $n\io$, then there is an anti-Nikodym sequence $\seqn{\nu_n}$ on $St(\aA)$ such that $\supp(\nu_n)$ is a finite subset of $\supp(\mu_n)$ for all $n\io$.
%\end{lemma}
%\begin{proof}
%Let $\seqn{\mu_n}$ be an anti-Nikodym sequence of measures on $St(\aA)$ with countable supports. By the continuity of measures, for every $n\io$ there exists a finite set $F_n\sub\supp{\mu_n}$ such that $\|\mu_n\rstr F_n^c\| < 1/n$. For each $n\io$ let us define a measure $\nu_n=\mu_n\rstr F_n$. We claim that $\seqn{\nu_n}$ is also an anti-Nikodym sequence on $St(\aA)$. First, we have
%\[\|\nu_n\|\geq \|\mu_n\| - \|\mu_n\rstr F_n^c\| > \|\mu_n\| - 1/n,\textrm{ and so }\|\nu_n\|\to\infty.\]
%Next, for every $A\in\aA$ we have  $\nu_n(\clopen{A}_\aA)\to 0$, as $\mu_n(\clopen{A}_\aA)\to 0$ and \[ \big|\nu_n(\clopen{A}_\aA) - \mu_n(\clopen{A}_\aA)\big|\leq\|\nu_n - \mu_n\|=\|\mu_n\rstr F_n^c\|<1/n \to 0. \]
%\end{proof}

\medskip

One of the first results connecting the Nikodym property with any of the above defined classes of ideals on $\omega$ was the result of Drewnowski, Florencio, and Pa\'ul \cite[Proposition 6]{DFP94} asserting that the asymptotic density zero ideal $\zZ$ has the Nikodym property. Later, the same authors showed in \cite[Corollary 2.3]{DFP96} that the Nikodym property implies either of the following equivalent properties of ideals.

\begin{definition}\label{def:psp_asp}
    An ideal $\iI$ on $\omega$ has the \textit{Positive Summability Property} (in short, \textit{(PSP)}), if, for every sequence $\seqn{x_n}$ of non-negative real numbers, we have $\sum_{n\io}x_n<\infty$ provided that $\sum_{n\in A}x_n<\infty$ for every $A\in\iI$.

    An ideal $\iI$ on $\omega$ has the \textit{Absolute Summability Property} (in short, \textit{(ASP)}), if, for every sequence $\seqn{x_n}$ of real numbers, we have $\sum_{n\io}\big|x_n\big|<\infty$ provided that $\sum_{n\in A}x_n<\infty$ for every $A\in\iI$.
\end{definition}

The equivalence of (PSP) and (ASP) was observed by Sember and Freedman \cite[Proposition 1]{SF81}. Note that, trivially, an ideal $\iI$ on $\omega$ has (PSP) if and only if $\iI$ cannot be extended to a summable ideal. In particular, by aforementioned \cite[Corollary 2.3]{DFP96}, no summable ideal has the Nikodym property (cf. Theorem \ref{thm:zuch} and Section \ref{sec:tukey}).

By Drewnowski and Pa\'ul \cite[Theorem 3.5]{DP00}, the Nikodym property coincides with (PSP) and (ASP) in the class of P-ideals (see Section \ref{sec:tukey} for a generalization of this fact). If an ideal $\iI$ with (ASP) is not a P-ideal, then $\iI$ need not have the Nikodym property---for suitable examples, see \cite[Section 4]{DP00} (or Remark \ref{rem:lacunary}).%---it was proved in \cite[Section 4]{DP00} that $Fin\otimes Fin$ and the ideal $\lL$ generated by all lacunary subsets of $\omega$ both have (ASP) but not the Nikodym property.

Note that every ideal $\iI$ with (PSP) (so, in particular, with the Nikodym property) is \textit{tall}, that is, for any infinite set $A\in\wo$ there is an infinite set $B\in\iI$ such that $B\sub A$.

\medskip

The relations between the above defined standard classes of non-pathological ideals and the finitely supported Nikodym property of the spaces $N_\fF$ were studied by the second author in \cite{Zuc25}. The following theorem summarizes those of his results, which are the most significant for the context of the present work.% (cf. also \cite[Theorem 5.5]{Zuc25} and Corollary \ref{cor:tot_bnded_full_char}).

\begin{theorem}[\.{Z}uchowski]\label{thm:zuch}% {\cite{Zuc25}}
Let $\iI$ be an ideal on $\omega$.
\begin{enumerate}
    \item $N_{\iI^*}$ does not have the finitely supported Nikodym property if and only if there exists an unbounded non-pathological lsc submeasure $\varphi$ on $\omega$ such that $\iI\sub\exh(\varphi)$ if and only if there exists an unbounded density submeasure $\varphi$ on $\omega$ such that $\iI\sub\exh(\varphi)$ (\cite[Theorem 4.3]{Zuc25}).
    \item If $N_{\iI^*}$ does not have the finitely supported Nikodym property, then $\iI\le_K\zZ$ (\cite[Corollary 5.4]{Zuc25}).
    \item If $\iI$ is a density ideal, then $N_{\iI^*}$ has the finitely supported Nikodym property if and only if $\iI$ is isomorphic to an \erdos--Ulam ideal (\cite[Theorem 5.5]{Zuc25}). %if and only if $\iI\equiv_K\zZ$ 
\end{enumerate}
\end{theorem}

Let us note that \cite[Theorem 5.5]{Zuc25} lists several additional interesting conditions equivalent for a density ideal to be isomorphic to an \erdos--Ulam ideal, e.g. being Kat\v{e}tov equivalent to $\zZ$. The latter in particular implies that if $\iI$ is a density ideal and $N_{\iI^*}$ does not have the finitely supported Nikodym property, then $\iI<_K\zZ$.

Combining Theorems \ref{thm:p_ideal_nik_equivalences} and \ref{thm:zuch}.(2), we get the following corollary.

\begin{corollary}\label{cor:not_lek_z_web_nik}
    If $\iI$ is a P-ideal on $\omega$ such that $\iI\not\le_K\zZ$, then $\iI$ has the web Nikodym property.
\end{corollary}
%\begin{proof}
%    Let $\iI$ be a P-ideal on $\omega$ with $\iI\not\le_K\zZ$. By \cite[Corollary 5.4]{Zuc25}, $N_{\iI^*}$ has the finitely supported Nikodym property, and so, by Theorem \ref{thm:p_ideal_nik_equivalences}, $\iI$ has the web Nikodym property.
%\end{proof}

We also obtain the following important result, being a variant of Drewnowski, Florencio, and Pa\'ul's \cite[Example 4.3]{DFP96} implying that \erdos--Ulam ideals belonging to a special class (which includes $\zZ$) have the Nikodym property (see also \cite[Theorem 6.8]{DP00}). 

\begin{theorem}\label{thm:dens_nik_erdos_ulam}
Let $\iI$ be a density ideal on $\omega$. Then, $\aA_{\iI^*}$ has the Nikodym property if and only if $\iI$ is isomorphic to an \erdos--Ulam ideal.
\end{theorem}
\begin{proof}
If $\aA_{\iI^*}$ has the Nikodym property, then by Theorem \ref{thm:sf_nik_nf_sfs} $N_{\iI^*}$ has the finitely supported Nikodym property, but, as $\iI$ is a density ideal, this may only happen if $\iI$ is isomorphic to an \erdos--Ulam ideal, by Theorem \ref{thm:zuch}.(3). 

Assume then that $\iI$ is isomorphic to an \erdos--Ulam ideal. %Assume that $\aA_{\iI^*}/Fin$ does not have the Nikodym property and so let $\seqn{\mu_n}$ be an anti-Nikodym sequence on $\aA_{\iI^*}/Fin$. As $\iI$ is in particular a density ideal, it is a P-ideal, and so, by Proposition $\ref{prop:sf_properties}.(4)$, $p_{\iI^*}$ is a P-point in the space $S_{\iI^*}$. Thus, by \cite[Proposition 4.5]{Sob19}, $p_{\iI^*}$ cannot be a Nikodym concentration point of $\seqn{\mu_n}$. Consequently, $\seqn{\mu_n}$ has a concentration point $t\in S_{\iI^*}^*$ different than $p_{\iI^*}$. Let $A\in\iI^*$ be such that $t\in\clopen{A}_{\iI^*}^*$. Note that for the equivalence class $A^\bullet$ of $A$ in the quotient algebra $\aA_{\iI^*}/Fin$ we have $\big(\aA_{\iI^*}/Fin\big)_{A^\bullet}=\wp(A)/[A]^{<\omega}$. Consequently, the sequence $\seqn{\mu_n\rstr\wp(A)/[A]^{<\omega}}$ is an anti-Nikodym sequence on $\wp(A)/[A]^{<\omega}$, which is impossible as the latter algebra has the Nikodym property by the Seever theorem \cite{See68}. Consequently, the algebra $\aA_{\iI^*}/Fin$ has the Nikodym property.
%    
%Next, 
Again, by Theorem \ref{thm:zuch}.(3), the space $N_{\iI^*}$ has the finitely supported Nikodym property. As $\iI$ is a P-ideal, Theorem \ref{thm:p_ideal_nik_equivalences} yields that the algebra $\aA_{\iI^*}$ has the Nikodym property.
\end{proof}

Theorems \ref{thm:p_ideal_nik_equivalences} and \ref{thm:dens_nik_erdos_ulam} immediately yield together the following result (see also Corollary \ref{cor:density_nik_strong_web_totally bounded}).

\begin{corollary}\label{cor:density_nik_strong_web}
    Let $\iI$ be a density ideal on $\omega$. Then, the following conditions are equivalent:
    \begin{enumerate}
%        \item $\aA_{\iI^*}$ has the Nikodym property,
        \item $\iI$ has the Nikodym property,
        \item $\iI$ has the strong Nikodym property,
        \item $\iI$ has the web Nikodym property,
        \item $\iI$ is isomorphic to an \erdos--Ulam ideal.
    \end{enumerate}
\end{corollary}

\begin{remark}\label{rem:c_many_density}
    Note that, e.g. by Kwela \textit{et al.} \cite[Theorem 3]{KPST}, there are $\frakc$ many Kat\v{e}tov incomparable density ideals, hence there are $\frakc$ many Kat\v{e}tov incomparable density ideals which are not isomorphic to \erdos--Ulam ideals (see Farah \cite[Lemma 1.13.10]{Far00}). Consequently, by Corollary \ref{cor:density_nik_strong_web}, there are $\frakc$ many non-isomorphic density ideals without the Nikodym property. 

    On the other hand, by Kwela \cite[Propositions 5 and 6]{Kwe18} there are $\frakc$ many pairwise non-isomorphic \erdos--Ulam ideals, and hence by Corollary \ref{cor:density_nik_strong_web} $\frakc$ many pairwise non-isomorphic density ideals with the web Nikodym property.
\end{remark}

%{\color{red}tutaj przyklady nowych ideałów z web Nikodym property, np. uniform density ideal albo jakieś z summability!!!!!!!!!!!!!!! (p. Stuart)
%
%- czyli chodzi o to żeby napisać przykłady Exh-ów które mają własność Nikodyma, np. trn(N) ? NO TAK :)} 

\begin{example}\label{example:new_web_nikodym}
The preceding results give us several new examples of ideals on $\omega$ with the web Nikodym property. The first one is $\tr(\nN)$, which is a non-pathological analytic P-ideal such that $N_{\tr(\nN)^*}$ has the finitely supported Nikodym property (see \cite[page 728]{Zuc25}), and so $\tr(\nN)$ has the web Nikodym property by Theorem \ref{thm:p_ideal_nik_equivalences}. By Corollary \ref{cor:density_nik_strong_web}, the logarithmic density zero ideal $\zZ_{\log}$ has the web Nikodym property, too. Finally, a large class of ideals with the web Nikodym property is given in Section \ref{sec:hypergraph}, see Remark \ref{rem:hypergraph_web_nikodym}.
\end{example}

%\begin{proposition}\label{prop:erdos_ulam_nik_prop}
%If $\iI$ is an \erdos--Ulam ideal, then $\aA_{\iI^*}$ has the Nikodym property. In particular, $\aA_{\zZ^*}$ has the Nikodym property.
%\end{proposition}
%\begin{proof}
%Let $\iI$ be an \erdos--Ulam ideal. Assume that $\aA_{\iI^*}/Fin$ does not have the Nikodym property and so let $\seqn{\mu_n}$ be an anti-Nikodym sequence on $\aA_{\iI^*}/Fin$. As $\iI$ is in particular a density ideal, it is a P-ideal, and so, by Proposition $\ref{prop:sf_properties}.(4)$, $p_{\iI^*}$ is a P-point in the space $S_{\iI^*}$. Thus, by \cite[Proposition 4.5]{Sob19}, $p_{\iI^*}$ cannot be a Nikodym concentration point of $\seqn{\mu_n}$. Consequently, $\seqn{\mu_n}$ has a concentration point $t\in S_{\iI^*}^*$ different than $p_{\iI^*}$. Let $A\in\iI^*$ be such that $t\in\clopen{A}_{\iI^*}^*$. Note that for the equivalence class $A^\bullet$ of $A$ in the quotient algebra $\aA_{\iI^*}/Fin$ we have $\big(\aA_{\iI^*}/Fin\big)_{A^\bullet}=\wp(A)/[A]^{<\omega}$. Consequently, the sequence $\seqn{\mu_n\rstr\wp(A)/[A]^{<\omega}}$ is an anti-Nikodym sequence on $\wp(A)/[A]^{<\omega}$, which is impossible as the latter algebra has the Nikodym property by the Seever theorem. Consequently, the algebra $\aA_{\iI^*}/Fin$ has the Nikodym property.
%
%Next, by \cite[Theorem 5.5]{Zuc25}, the space $N_{\iI^*}$ has the finitely supported Nikodym property. Therefore, by Theorem \ref{thm:sf_nik_nf_sfs}, the algebra $\aA_{\iI^*}$ has the Nikodym property.
%\end{proof}

The following example, appealing to the class $\aA\nN$ of ideals (see Definition \ref{def:class_an}), supplements Example \ref{example:strong_only} and expands Remark \ref{rem:c_many_density}.

\begin{example}\label{example:wo_no_nikodym} (Supplement to Example \ref{example:strong_only}). Let us consider the algebra $\aA_{\iI^*}$ for some $\iI\in\aA\nN$. %By \cite[Corollary 8.3]{Zuc25}, it does not have the Nikodym property (since $N_\fF$ homeomorphically embeds into $St\big(\aA_{\iI^*}\big)$). 
By Theorem \ref{thm:sf_nik_nf_sfs}, $\aA_{\iI^*}$ does not have the Nikodym property. On the other hand, for every $A\in\iI$ the restricted algebra $\wo_A=\wp(A)$ is $\sigma$-complete, so it has the Nikodym property. Therefore, by Proposition \ref{prop:ai_uniq_concentr_point}, every anti-Nikodym sequence on $\aA_{\iI^*}$ has a unique Nikodym concentration point, namely $p_{\iI^*}$, which is then necessarily strong.

Recall that in \cite{Zuc25} it was shown that $\iI\in\aA\nN$ provided that $\iI=\exh(\varphi)$ where $\varphi$ is a non-pathological lsc submeasure on $\omega$ satisfying $\varphi(\omega)=\infty$; thus, e.g., all summable ideals are in $\aA\nN$. Consequently, the class $\aA\nN$ is large, that is, it contains $2^\frakc$ many non-isomorphic non-maximal ideals (cf. \cite[Corollary 7.4]{Zuc25}). By Lemma \ref{lemma:sf_homeo_to_filt_iso} below, it follows that there are $2^\frakc$ many non-isomorphic Boolean subalgebras of $\wo$ without the Nikodym property and carrying only anti-Nikodym sequences with unique Nikodym concentration points. 
\end{example}

As we stated in Example \ref{example:strong_conv_seq}, if the Stone space $St(\aA)$ of a Boolean algebra $\aA$ contains a non-eventually constant convergent sequence, then $\aA$ admits an anti-Nikodym sequence of measures with a strong Nikodym concentration point. Since the space $N_{Fr}$ is clearly homeomorphic to any space consisting of a non-eventually constant convergent sequence together with its limit, the example may be rephrased as follows: if $N_{Fr}$ homeomorphically embeds into $St(\aA)$, then $\aA$ admits an anti-Nikodym sequence of measures with a strong Nikodym concentration point. The next proposition asserts that actually the latter statement holds true for any filter $\fF$ with $\fF^*\in\aA\nN$ and the space $N_{\fF}$ (cf. \cite[Theorem 8.2]{Zuc25}).

\begin{proposition}\label{prop:nf_strong}
Let $\iI\in\aA\nN$. Let $\aA$ be a Boolean algebra such that $N_{\iI^*}$ homeomorphically embeds into $St(\aA)$. Then, $\aA$ admits an anti-Nikodym sequence of measures with a unique strong Nikodym concentration point.
\end{proposition}
\begin{proof}
    Let $h\colon N_{\iI^*}\to St(\aA)$ be a homeomorphism onto the image $X=h\big[N_{\iI^*}\big]$. Set $x=h\big(p_{\iI^*}\big)$ and $Y=X\sm\{x\}=h[\omega]$. By \cite[Theorem 4.1]{Zuc25}, there exists a disjointly supported sequence $\seqn{\mu_n}$ of finitely supported non-negative measures on $St(\aA)$ such that
    \begin{enumerate}[(a)]
        \item $\supp\big(\mu_n\big)\sub Y$ for every $n\io$,
        \item $\lim_{n\to\infty}\mu_n(St(\aA))=\infty$, and
        \item $\lim_{n\to\infty}\mu_n(St(\aA)\sm U)=0$ for every clopen neighborhood $U$ of $x$.
    \end{enumerate}

    For each $n\io$ define the measure $\nu_n$ via its Radon extension as
    \[\wh{\nu}_n=\mu_n(St(\aA))\cdot\delta_x-\mu_n.\]
    We claim that the sequence $\seqn{\nu_n}$ is anti-Nikodym on $\aA$ and $x$ is its unique (strong) Nikodym concentration point. First, note that by (b) we have $\sup_{n\io}\big\|\nu_n\big\|=\infty$. Second, by (c), for every $A\in\aA$, if $A\not\in x$, then we have
    \[\lim_{n\to\infty}\wh{\nu}_n\big(\clopen{A}_\aA\big)=0-\lim_{n\to\infty}\mu_n\big(\clopen{A}_\aA\big)=0,\]
    and if $A\in x$, then it holds 
    \[\lim_{n\to\infty}\wh{\nu}_n\big(\clopen{A}_\aA\big)=\lim_{n\to\infty}\Big(\mu_n(St(\aA))-\mu_n\big(\clopen{A}_\aA\big)\Big)=\lim_{n\to\infty}\mu_n\big(St(\aA)\sm\clopen{A}_\aA\big)=0.\]
    Consequently, $\seqn{\nu_n}$ is indeed anti-Nikodym.

    We now show that $x$ is a unique strong Nikodym concentration point of $\seqn{\nu_n}$. To achieve this, we prove that $x$ is the only Nikodym concentration point of $\seqn{\nu_n}$ and then appeal to Lemma \ref{lemma:one_Ncp_sNcp}. By (a) and (b), for any $A\in x$ we have
    \[\sup_{n\io}\big\|\nu_n\rstr A\big\|=\sup_{n\io}\Big(\big\|\wh{\nu}_n\rstr\clopen{A}_\aA\sm\{x\}\big\|+\big|\mu_n(St(\aA))\big|\Big)\ge\sup_{n\io}\mu_n(St(\aA))=\infty,\]
    so $x$ is a Nikodym concentration point of $\seqn{\nu_n}$. If, on the other hand, $y\in St(\aA)\sm\{x\}$, then there is $B\in y$ such that $B\not\in x$. By (c), we then have
    \[\sup_{n\io}\big\|\nu_n\rstr B\big\|=\sup_{n\io}\mu_n\big(\clopen{B}_\aA\big)<\infty,\]
    so $y$ is not a Nikodym concentration point of $\seqn{\nu_n}$.
%    \[\alpha=\sup_{n\io}\mu_n\big(St(\aA)\sm\clopen{A}_\aA\big)=\sup_{n\io}\big\|\nu_n\rstr A^c\big\|<\infty.\]
%    Then, by (b), there is $n\io$ such that
%    \[\mu_n(St(\aA))>N\cdot(\alpha+1)+\alpha,\]
%    and so we have
%    \[\mu_n\big(\clopen{A}_\aA\big)>N\cdot(\alpha+1)\ge N\cdot\Big(\big\|\nu_n\rstr A^c\big\|+1\Big).\]
%    Since $\mu_n$ is finitely supported, (a) implies that there is $B\in\aA$, $B\le A$, such that $x\not\in\clopen{B}_\aA$ and $\mu_n(B)=\mu_n(A)$. Thus, $\clopen{B}_\aA\sub\clopen{A}_\aA\sm\{x\}$ and
%    \[\mu_n\big(\clopen{B}_\aA\big)>N\cdot\Big(\big\|\nu_n\rstr A^c\big\|+1\Big),\]
%    which finishes the proof.
\end{proof}

By \cite[Corollary 7.4]{Zuc25} we immediately obtain the following result extending Example \ref{example:strong_conv_seq}. Note that the cardinal number $2^\frakc$ is here maximal possible.

\begin{corollary}\label{cor:2c_many_strong}
    There exists a family $\xX$ of $2^\frakc$ many pairwise non-homeomorphic countable infinite spaces with exactly one non-isolated point such that, for any $X\in\xX$, if $X$ homeomorphically embeds into the Stone space $St(\aA)$ of a Boolean algebra $\aA$, then $\aA$ admits an anti-Nikodym sequence of measures with a unique strong Nikodym concentration point.
\end{corollary}

%{\color{red}tutaj paragraf o tym, że $\mathcal{AN}$ daje nam $2^\frakc$ typów różnych strong Nikodym concentration point -- związek z pytaniem \ref{ques:type_sncp}}

%\begin{remark}\label{rem:p_filter_nik_prop}
%\end{remark}

\subsection{Large families of algebras with the Nikodym property and without the Grothendieck property\label{sec:definable_large}}

In this section we construct two large families $\hH_1$ and $\hH_2$ of Boolean algebras with the Nikodym property and without the Grothendieck property. The family $\hH_1$ consists of $\frakc$ algebras of the form $\aA_{\fF}$ for a Borel (in fact, $\mathbb{F}_{\sigma\delta}$) filter $\fF$ on $\omega$, and the family $\hH_2$ consists of $2^\frakc$ algebras of the form $\aA_{\fF}$ for some non-analytic filter $\fF$ on $\omega$. These constructions supplement the results of Schachermayer \cite{Sch82}, Graves and Wheeler \cite{GW83}, Valdivia \cite{Val13}, and Drewnowski, Florencio, and Pa\'ul \cite[Remark 4]{DFP94}.

We start with the following lemma.

\begin{lemma}\label{lemma:sf_homeo_to_filt_iso}
%Let $F$ and $G$ be filters on $\omega$ that are not ultrafilters. If the spaces $S_F$ and $S_G$ are homeomorphic, then $F$ and $G$ are isomorphic filters.
Let $\fF$ and $\gG$ be filters on $\omega$ which are not ultrafilters. Assume that $\varphi\colon S_{\fF}\to S_{\gG}$ is a homeomorphism. Then, $\varphi\rstr N_{\fF}$ is a homeomorphism of the spaces $N_{\fF}$ and $N_{\gG}$ and the filters $\fF$ and $\gG$ are isomorphic.
\end{lemma}
\begin{proof}
By $\uU_z$ let us denote the neighborhood filter of a point $z$ in a topological space $Z$.

First, note that since $\varphi$ is a homeomorphism and the points in $\omega$ are the only isolated points in both $S_{\fF}$ and $S_{\gG}$, we have $\varphi[\omega]=\omega$ and $\varphi\big[S_\fF^*\big]=S_\gG^*$. It follows that $\varphi\rstr\omega\colon\omega\to\omega$ is a bijection.

Second, for each points $x\in S_{\fF}^*$ and $y\in S_{\gG}^*$ we have $x=\big\{U\cap\omega\colon U\in\uU_x\big\}$ and $y=\big\{V\cap\omega\colon V\in\uU_y\big\}$. Moreover, as $\varphi$ is a homeomorphism, for every $x\in S_{\fF}^*$ it holds $\uU_{\varphi(x)}=\big\{\varphi[U]\colon U\in\uU_x\big\}$. Consequently, for every $x\in S_{\fF}^*$ we have
\[\tag{$*$}\varphi(x)=\big\{V\cap\omega\colon V\in\uU_{\varphi(x)}\big\}=\big\{\varphi[U]\cap\omega\colon U\in\uU_x\big\}=\big\{\varphi[U\cap\omega]\colon U\in\uU_x\big\}=\]
\[=\{\varphi[A]\colon A\in x\}.\]

Third, observe that any points $x\in S_{\fF}^*\sm\big\{p_{\fF}\big\}$ and $y\in S_{\gG}^*\sm\big\{p_{\gG}\big\}$ extend to unique ultrafilters in $\wo$, whereas by the assumption both $p_{\fF}=\fF$ and $p_{\gG}=\gG$ are non-maximal filters in $\wo$. As $\varphi\rstr\omega$ is a bijection, ($*$) implies that if $x\in S_{\fF}^*\sm\big\{p_{\fF}\big\}$, then $\varphi(x)$ also extends to a unique ultrafilter in $\wo$, and so, consequently, $\varphi(x)\in S_{\gG}^*\sm\big\{p_\gG\big\}$. It follows that $\varphi\big(p_{\fF}\big)=p_{\gG}$, and so $\varphi\restriction N_{\fF}$ is a homeomorphism between $N_{\fF}$ and $N_{\gG}$. Thus, by \cite[Proposition 3.10.(2)]{MS24}, the filters $\fF$ and $\gG$ are isomorphic.
\end{proof}

Note that, for a filter $\fF$ on $\omega$, the Boolean algebra $\aA_\fF=\fF\cup\fF^*$ belongs to a Borel class $\Gamma$ (resp. is analytic) provided that $\fF$ belongs to the Borel class $\Gamma$ (resp. is analytic).

\begin{theorem}\label{thm:continuum}
There is a family $\big\{\fF_\alpha\colon \alpha<\frakc\big\}$ of $\mathbb{F}_{\sigma\delta}$ filters on $\omega$ such that the family $\hH_1=\big\{\aA_{\fF_\alpha}\colon \alpha<\frakc\big\}$ consists of pairwise non-isomorphic Boolean algebras, each of them having the Nikodym property but not the Grothendieck property.
\end{theorem}
\begin{proof}
By the results from Kwela's \cite[Section 4]{Kwe18}, there is a family $\big\{\iI_\alpha\colon\alpha<\frakc\big\}$ of pairwise non-isomorphic \erdos--Ulam ideals. Every \erdos--Ulam ideal, as a density ideal, is $\mathbb{F}_{\sigma\delta}$. Let $\fF_\alpha=\iI_\alpha^*$ for every $\alpha<\frakc$, and set $\hH_1=\big\{\aA_{\fF_\alpha}\colon\alpha<\frakc\big\}$. By Lemma \ref{lemma:sf_homeo_to_filt_iso}, the spaces $S_{\fF_\alpha}$, $\alpha<\frakc$, are pairwise non-homeomorphic, and so the Boolean algebras $\aA_{\fF_\alpha}$, $\alpha<\frakc$, are pairwise non-isomorphic. Next, by \cite[Theorem C]{MS24}, for all $\alpha<\frakc$ the space $N_{\fF_\alpha}$ has the bounded Josefson--Nissenzweig property, and so, by \cite[Corollary D]{MS24}, the algebra $\aA_{\fF_\alpha}$ does not have the Grothendieck property. Finally, by Theorem \ref{thm:dens_nik_erdos_ulam}, the algebra $\aA_{\fF_\alpha}$ has the Nikodym property for every $\alpha<\frakc$.
\end{proof}

\begin{remark}
Let us note that in the above proof we could use as well the family of ideals given by Theorem \ref{thm:hypergraph_nikodym_not_tot_bnd} instead of the one from \cite[Section 4]{Kwe18}, since it also consists of analytic P-ideals (hence, $\mathbb{F}_{\sigma\delta}$ ideals) with the Nikodym property and for which \cite[Theorem C  and Corollary D]{MS24} are applicable.
\end{remark}

\medskip

For a pair of topological spaces $X$, $Y$ and points $x\in X$, $y\in Y$, by $(X\sqcup Y)/\{x,y\}$ we mean the topological disjoint union $X\sqcup Y$ of $X$ and $Y$ with the points $x$ and $y$ identified as a one point. Both of the spaces $X$ and $Y$ are naturally identified inside $(X\sqcup Y)/\{x,y\}$.

Let $\fF_1$ and $\fF_2$ be filters on $\omega$. Fix a partition $(\Omega_1,\Omega_2)$ of $\omega$ into two infinite sets and bijections $b_i\colon\omega\to \Omega_i$, $i=1,2$. Then, we define the filter $\fF_1\oplus \fF_2$ on $\omega$ as follows: for every $A\in\wo$, $A\in \fF_1\oplus \fF_2$ if and only if $b_i^{-1}\big[A\cap \Omega_i\big]\in \fF_i$ for $i=1,2$. For each $i=1,2$ and the subspace $N_{\fF_i}'=\Omega_i\cup\big\{p_{\fF_1\oplus\fF_2}\big\}$ of $N_{\fF_1\oplus\fF_2}$, the bijection $b_i$ gives rise to the natural homeomorphism $h_i\colon N_{\fF_i}\to N_{\fF_i}'$ with $h_i\big(p_{\fF_i}\big)=p_{\fF_1\oplus\fF_2}$, identifying the spaces $N_{\fF_i}$ and $N_{\fF_i}'$. Consequently, the space $N_{\fF_1\oplus \fF_2}$ can be naturally identified with $\big(N_{\fF_1}\sqcup N_{\fF_2}\big)/\big\{p_{\fF_1},p_{\fF_2}\big\}$, with the points $p_{\fF_1}$ and $p_{\fF_2}$ glued together as the point $p_{\fF_1\oplus \fF_2}$. The following lemma extends this identification to the spaces $S_{\fF_1\oplus \fF_2}$ and $\big(S_{\fF_1}\sqcup S_{\fF_2}\big)/\big\{p_{\fF_1},p_{\fF_2}\big\}$.

\begin{lemma}\label{lemma:sf_sum}
    For every filters $\fF_1$ and $\fF_2$ on $\omega$, the space $S_{\fF_1\oplus \fF_2}$ can be naturally identified with the quotient space $\big(S_{\fF_1}\sqcup S_{\fF_2}\big)/\big\{p_{\fF_1},p_{\fF_2}\big\}$.
\end{lemma}
\begin{proof}
    We first show that
    \[\tag{$*$}\overline{N_{\fF_1}'}^{S_{\fF_1\oplus \fF_2}}\cap\overline{N_{\fF_2}'}^{S_{\fF_1\oplus \fF_2}}=\big\{p_{\fF_1\oplus \fF_2}\big\}.\] 
    Of course, we have
    \[\big\{p_{\fF_1\oplus \fF_2}\big\}=N_{\fF_1}'\cap N_{\fF_2}'\sub\overline{N_{\fF_1}'}^{S_{\fF_1\oplus \fF_2}}\cap\overline{N_{\fF_2}'}^{S_{\fF_1\oplus \fF_2}}.\]
    So, assume there is
    \[x\in\overline{N_{\fF_1}'}^{S_{\fF_1\oplus \fF_2}}\cap\overline{N_{\fF_2}'}^{S_{\fF_1\oplus \fF_2}}\]
    such that $x\neq p_{\fF_1\oplus \fF_2}$. 
    Then, there is a clopen subset $U$ of $S_{\fF_1\oplus\fF_2}$ such that $x\in U$ and $p_{\fF_1\oplus\fF_2}\not\in U$. Let $A_1=\Omega_1\cap U$ and $A_2=\Omega_2\cap U$. It follows that $U\cap\omega=A_1\cup A_2$, so $A_1\cup A_2\in\aA_{\fF_1\oplus\fF_2}$ and
    \[U=\overline{A_1\cup A_2}^{S_{\fF_1\oplus\fF_2}}=\clopen{A_1\cup A_2}_{\fF_1\oplus\fF_2}.\]
    Moreover, as for each $i=1,2$ the set $A_i$ is clopen in $N_{\fF_i}'$ and $p_{\fF_1\oplus\fF_2}\not\in A_i$, we have $b_i^{-1}\big[A_i\big]\in\fF_i^*$, and hence $A_i\in\aA_{\fF_1\oplus\fF_2}$. Also, as $A_1\cap A_2\sub \Omega_1\cap \Omega_2=\emptyset$, we get that 
    \[\clopen{A_1}_{\fF_1\oplus\fF_2}\cap\clopen{A_2}_{\fF_1\oplus\fF_2}=\emptyset.\]
    Since
    \[x\in U=\clopen{A_1\cup A_2}_{\fF_1\oplus\fF_2}=\clopen{A_1}_{\fF_1\oplus\fF_2}\cup\clopen{A_2}_{\fF_1\oplus\fF_2},\]
    we obtain that either $x\in\clopen{A_1}_{\fF_1\oplus\fF_2}$ and $x\not\in\clopen{A_2}_{\fF_1\oplus\fF_2}$, or $x\in\clopen{A_2}_{\fF_1\oplus\fF_2}$ and $x\not\in\clopen{A_1}_{\fF_1\oplus\fF_2}$. However, note then that in the first case we have $x\not\in\overline{N_{\fF_2}'}^{S_{\fF_1\oplus \fF_2}}$, as $\clopen{A_1}_{\fF_1\oplus\fF_2}\cap N_{\fF_2}'=\emptyset$, and in the second case it holds $x\not\in\overline{N_{\fF_1}'}^{S_{\fF_1\oplus \fF_2}}$, since $\clopen{A_2}_{\fF_1\oplus\fF_2}\cap N_{\fF_1}'=\emptyset$. In either case, we get a contradiction, which proves that ($*$) indeed holds.
    %As $x\neq p_{\fF_1\oplus\fF_2}$, for each $i\in\{1,2\}$ there is $A_i\sub \Omega_i$ such that $b_i^{-1}\big[A_i\big]\in \fF_i^*$ and $x\in\clopen{A_i}_{\fF_1\oplus\fF_2}$. Note that the sets $A_1$ and $A_2$ are closed in $N_{\fF_1\oplus\fF_2}$. Moreover, $A_1\cap A_2\sub \Omega_1\cap \Omega_2=\emptyset$, so they are also disjoint. Consequently, since by Proposition \ref{prop:sf_properties}.(5) the space $N_{\fF_1\oplus \fF_2}$ is $C^*$-embedded in $S_{\fF_1\oplus \fF_2}$, the Urysohn Extension Theorem (\cite[Theorem 1.17]{GJ60}) implies that
    %\[\overline{A_1}^{S_{\fF_1\oplus \fF_2}}\cap\overline{A_2}^{S_{\fF_1\oplus \fF_2}}=\emptyset,\]
    %which is a contradiction as the latter intersection contains $x$. Consequently, ($*$) holds true.

    Second, we trivially have
    \[\tag{$**$}S_{\fF_1\oplus \fF_2}=\overline{N_{\fF_1\oplus \fF_2}}^{S_{\fF_1\oplus \fF_2}}=\overline{N_{\fF_1}'\cup N_{\fF_2}'}^{S_{\fF_1\oplus \fF_2}}=\overline{N_{\fF_1}'}^{S_{\fF_1\oplus \fF_2}}\cup\overline{N_{\fF_2}'}^{S_{\fF_1\oplus \fF_2}}.\]

    Finally, for every $i\in\{1,2\}$ we have
    \[\tag{$*\!*\!*$}\overline{N_{\fF_i}'}^{S_{\fF_1\oplus \fF_2}}=\beta\big(N_{\fF_i}'\big).\]
    Indeed, fix $i\neq j\in\{1,2\}$ and note that every $f\in C_b\big(N_{\fF_i}'\big)$ extends to the function $f'\in C_b\big(N_{\fF_1\oplus \fF_2}\big)$ defined for every $n\in \Omega_{j}$ by $f'(n)=f\big(p_{\fF_1\oplus \fF_2}\big)$. Since the space $N_{\fF_1\oplus \fF_2}$ is $C^*$-embedded in $S_{\fF_1\oplus \fF_2}$, $f'$ extends to some $f''\in C\big(S_{\fF_1\oplus \fF_2}\big)$. Consequently, $N_{\fF_i}'$ is $C^*$-embedded in $S_{\fF_1\oplus \fF_2}$ and therefore, by \cite[Section 6.9]{GJ60}, we have $\overline{N_{\fF_i}'}^{S_{\fF_1\oplus \fF_2}}=\beta\big(N_{\fF_i}'\big)$.

    Now, combining ($*$), ($**$), and ($*\!*\!*$), we get the natural identification of the space $S_{\fF_1\oplus\fF_2}$ with the quotient space
    \[\Big(\beta\big(N_{\fF_1}'\big)\sqcup\beta\big(N_{\fF_2}'\big)\Big)/\big\{p_{\fF_1\oplus\fF_2}^1,p_{\fF_1\oplus\fF_2}^2\big\},\]
    where for $i=1,2$ the point $p_{\fF_1\oplus\fF_2}^i$ is the copy of the point $p_{\fF_1\oplus\fF_2}$ in the subspace $N_{\fF_i}'$ of the disjoint union $\beta\big(N_{\fF_1}'\big)\sqcup\beta\big(N_{\fF_2}'\big)$. As for each $i=1,2$, using the homeomorphism $h_i$, we have the identification $\beta\big(N_{\fF_i}'\big)=\beta\big(N_{\fF_i}\big)=S_{\fF_i}$ such that $p_{\fF_1\oplus\fF_2}^i$ is identified with $p_{\fF_i}$, we get the required natural identification of the spaces $\big(S_{\fF_1}\sqcup S_{\fF_2}\big)/\big\{p_{\fF_1},p_{\fF_2}\big\}$ and $S_{\fF_1\oplus \fF_2}$.
\end{proof}

\begin{theorem}\label{thm:nikodym_sum}
    %Let $\fF_1$ and $\fF_2$ be two filters on $\omega$ such that both Boolean algebras $\aA_{\fF_1}$ and $\aA_{\fF_2}$ have the Nikodym property. Then, the algebra $\aA_{\fF_1\oplus \fF_2}$ has the Nikodym property, too.
    Let $\fF_1$ and $\fF_2$ be two filters on $\omega$. Then, the algebra $\aA_{\fF_1\oplus \fF_2}$ has the Nikodym property if and only if both algebras $\aA_{\fF_1}$ and $\aA_{\fF_2}$ have the Nikodym property.
\end{theorem}
\begin{proof}
    We will use the identification described in Lemma \ref{lemma:sf_sum}. 
    
    Assume first that both $\aA_{\fF_1}$ and $\aA_{\fF_2}$ have the Nikodym property, but $\aA_{\fF_1\oplus \fF_2}$ does not. Let $\seqn{\mu_n}$ be a sequence of non-negative measures on $St\big(\aA_{\fF_1\oplus \fF_2}\big)$ as in Corollary \ref{cor:sf_an_char_pos}. By condition (2), there exists $i\in\{1,2\}$ for which we have $\sup_{n\io}\mu_n\big(S_{\fF_i}\big)=\infty$. Let $j\in\{1,2\}\sm\{i\}$. For every $n\io$, define the non-negative measure $\nu_n$ on $S_{\fF_i}$ simply by $\nu_n=\mu_n\rstr S_{\fF_i}$.

    We claim that the sequence $\seqn{\nu_n}$ of measures on $S_{\fF_i}$ too satisfies conditions (1)--(3) of Corollary  \ref{cor:sf_an_char_pos}. Indeed, (2) holds trivially. By condition (1) for $\seqn{\mu_n}$, for every $n\io$ we have $p_{\fF_1\oplus \fF_2}\not\in\supp\big(\mu_n\big)$, so $p_{\fF_i}\not\in\supp\big(\nu_n\big)$, hence condition (1) is satisfied by $\seqn{\nu_n}$ as well. By condition (3) for $\seqn{\mu_n}$, for every $A\in \fF_i$ we have
    \[\lim_{n\to\infty}\nu_n\big(\clopen{\omega\sm A}_{\fF_i}\big)=\lim_{n\to\infty}\nu_n\Big(\overline{\omega\sm A}^{S_{\fF_i}}\Big)=\lim_{n\to\infty}\mu_n\Big(\overline{\Omega_i\sm b_i[A]}^{S_{\fF_1\oplus\fF_2}}\Big)=\]
    \[=\lim_{n\to\infty}\mu_n\Big(\overline{\Omega_i\sm b_i[A]}^{S_{\fF_1\oplus\fF_2}}\cup\overline{\Omega_{j}\sm\Omega_{j}}^{S_{\fF_1\oplus\fF_2}}\Big)=\]
    \[=\lim_{n\to\infty}\mu_n\Big(\overline{\omega\sm\big(b_i[A]\cup\Omega_{j}\big)}^{S_{\fF_1\oplus\fF_2}}\Big)=\lim_{n\to\infty}\mu_n\Big(\clopen{\omega\sm\big(b_i[A]\cup b_{j}[\omega]\big)}_{\fF_1\oplus\fF_2}\Big)=0;\]
    it follows that condition (3) also holds for $\seqn{\nu_n}$. Consequently, $\aA_{\fF_i}$ does not have the Nikodym property, a contradiction.

    Assume now that $\aA_{\fF_1\oplus\fF_2}$ has the Nikodym property, but $\aA_{\fF_i}$ for some $i\in\{1,2\}$ does not, so that there exists a sequence $\seqn{\mu_n}$ of non-negative measures on $St\big(\aA_{\fF_i}\big)$ as in Corollary \ref{cor:sf_an_char_pos}. It is however immediate that $\seqn{\mu_n}$ is also a sequence on $St\big(\aA_{\fF_1\oplus\fF_2}\big)$ as in Corollary \ref{cor:sf_an_char_pos} and so $\aA_{\fF_1\oplus\fF_2}$ does not have the Nikodym property either, which is a contradiction.
\end{proof}

We are ready to construct the second aforementioned family $\hH_2$ of Boolean algebras.

\begin{theorem}\label{thm:two_to_continuum}
%There is a family $\{G_\alpha\colon \alpha<2^\frakc\}$ of pairwise non-isomorphic non-analytic filters on $\omega$ such that for each $\alpha<2^\frakc$ the Boolean algebra $\aA_{G_\alpha}$ has the Nikodym property but does not have the Grothendieck property.
There is a family $\big\{\gG_\alpha\colon \alpha<2^\frakc\big\}$ of non-analytic filters on $\omega$ such that the family $\hH_2=\big\{\aA_{\gG_\alpha}\colon \alpha<2^\frakc\big\}$ consists of pairwise non-isomorphic Boolean algebras, each of them having the Nikodym property but not the Grothendieck property.
\end{theorem}
\begin{proof}
Recall that every ultrafilter on $\omega$ is non-analytic. For each ultrafilter $\uU$ on $\omega$ let $\fF_\uU=\uU\oplus\zZ^*$. Then, for every $\uU$ the space $N_{\fF_\uU}$ has the bounded Josefson--Nissenzweig property by \cite[Theorem C and Proposition 5.17]{MS24}, and so the algebra $\aA_{\fF_\uU}$ does not have the Grothendieck property by \cite[Corollary D]{MS24}. On the other hand, by Theorem \ref{thm:nikodym_sum} each $\aA_{\fF_\uU}$ has the Nikodym property, since both $\aA_\uU=\wo$ and $\aA_{\zZ^*}$ have the Nikodym property (by the Nikodym--And\^{o} theorem and by Theorem \ref{thm:dens_nik_erdos_ulam}, respectively).

Finally, since there are $2^\frakc$ ultrafilters on $\omega$ and
each family of pairwise isomorphic filters has the cardinality at most $\frakc$, there is a family $\big\{\gG_\alpha\colon \alpha<2^\frakc\big\}$ of pairwise non-isomorphic filters of the form $\fF_\uU$. Set $\hH_2=\big\{\aA_{\gG_\alpha}\colon\alpha<2^\frakc\big\}$. By Lemma \ref{lemma:sf_homeo_to_filt_iso}, the spaces $S_{\gG_\alpha}$, $\alpha<2^\frakc$, are pairwise non-homeomorphic, and so the algebras $\aA_{\gG_\alpha}$, $\alpha<2^\frakc$, are pairwise non-isomorphic.
\end{proof}

\section{Boundedness, non-atomicity, and hypergraph ideals\label{sec:nonatomic}}

In this section we study the relations between the Nikodym property of ideals on $\omega$ and miscellaneous additional notions of boundedness and non-atomicity of ideals occurring in literature. In connection with this, we also expand and investigate the construction of so-called hypergraph ideals.

\medskip

Again, throughout this section, all filters on $\omega$ are assumed to be \textit{free}, so that all ideals on $\omega$ are assumed to contain the ideal $Fin$.

\subsection{Various boundedness properties of ideals on $\omega$}

We start with the following alike properties, basically introduced by Drewnowski, Florencio, and Pa\'ul  (\cite[Section 2]{DFP96}).

\begin{definition}[{\cite{DFP96}}]\label{def:lgbps}
    An ideal $\iI$ on $\omega$ has the \textit{Local-to-Global Boudnedness Property for submeasures} (in short, \textit{(LGBPs)}) if whenever $\varphi$ is a submeasure on $\omega$ for which  $\iI\sub\fin(\varphi)$, then $\sup_{A\in\iI}\varphi(A)<\infty$.
\end{definition}

\begin{definition}\label{def:lgbpl}
    An ideal $\iI$ on $\omega$ has the \textit{Local-to-Global Boudnedness Property for lower semicontinuous submeasures} (in short, \textit{(LGBPl)}) if whenever $\varphi$ is a lsc submeasure on $\omega$ for which  $\iI\sub\fin(\varphi)$, then $\varphi(\omega)<\infty$ (i.e. $\fin(\varphi)=\wo$).%{\footnote{\color{red} We require here that the submeasure $\varphi$ is lsc, as for every non-trivial ideal $\iI$ on $\omega$ there exists a non-lsc submeasure $\psi$ such that $\iI\sub \fin(\psi)$ and $\psi(\omega)=\infty$ ($\psi$ can be chosen to be even a measure).}}
\end{definition}

\begin{definition}[{\cite{DFP96}}]\label{def:lgbpm}
An ideal $\iI$ on $\omega$ has the \textit{Local-to-Global Boudnedness Property for measures} (in short, \textit{(LGBPm)}) if whenever $\varphi\colon\iI\to\R$ is a finitely additive set function which is pointwise bounded on $\iI_A$ (i.e. $\sup_{B\sub A}|\varphi(B)|<\infty$) for every $A\in\iI$, then $\varphi$ is a measure on $\iI$ (i.e. $\|\varphi\|<\infty$). 
\end{definition}

Intuitively, the Local-to-Global Boundedness Property for submeasures may be considered as the Nikodym property for a single submeasure. It was shown in \cite{DFP96} that an ideal $\iI$ on $\omega$ has the Nikodym property if and only if it has (LGBPm), and that if $\iI$ has (LGBPs), then it has (LGBPm). However, it was conjectured in \cite[Conjecture, page 147]{DFP96} that (LGBPm) does not imply (LGBPs). The next theorem shows that this conjecture indeed holds true (see also Remark \ref{rmk:nikodym_submeas}). It is also immediate that (LGBPs) implies (LGBPl); in Corollary \ref{cor:p_anal_lgbps_lgbpl} we show that the converse holds in the class of analytic P-ideals.

\begin{theorem}\label{thm:nikodym_lgbps}
There exists an ideal $\iI$ on $\omega$ with the Nikodym property but without the Local-to-Global Boundedness Property for (lsc) submeasures.
\end{theorem}
\begin{proof}
Let $\iI$ be either of the two examples considered in \cite[Remark 4.7]{Zuc25}. $\iI$ is then of the form $\iI = \exh(\varphi) = \fin(\varphi)$ for some lsc submeasure
$\varphi$ on $\omega$ and the space $N_{\iI^*}$ has the finitely supported
Nikodym property. Since $\omega\not\in\fin(\varphi)=\iI$, we have $\varphi(\omega)=\infty$, so $\iI$ does not have (LGBPl). As $\iI$ is in particular a P-ideal, Theorem \ref{thm:p_ideal_nik_equivalences} implies that $\iI$ has the Nikodym property.%by the proof of Theorem \ref{thm:dens_nik_erdos_ulam} it follows that the algebra $\aA_{\iI^*}/Fin$ has the Nikodym property. Thus, by Theorem \ref{thm:sf_nik_nf_sfs}, the algebra $\aA_{\iI^*}$ has the Nikodym property, and so, by Theorem \ref{thm:nik_prop_ideal_algebra}, the ideal $\iI$ has the Nikodym property.
\end{proof}

The following characterization of the Local-to-Global Boundedness Property for lsc submeasures immediately follows from Mazur's
Theorem \ref{thm:mazur_submeasures}. 
%Mazur's classical characterization of $\F_\sigma$-ideals \cite{Maz91}.

\begin{proposition}\label{prop:lgbps_char_mazur}%[{\cite{Maz91}}]
An ideal $\iI$ on $\omega$ has (LGBPl) if and only if $\iI$ cannot be extended to an $\F_\sigma$-ideal.
\end{proposition}

We will also consider the following property of analytic P-ideals introduced by Hern\'{a}ndez-Hern\'{a}ndez and Hru\v{s}\'{a}k \cite[Section 3]{HHH07} for the study of destructibility properties of ideals.

\begin{definition}[\cite{HHH07}]\label{def:tot_bnded}
An analytic P-ideal $\iI$ on $\omega$ is \textit{totally bounded} if whenever $\varphi$ is an lsc submeasure on $\omega$ for which $\iI=\exh(\varphi)$, then $\varphi(\omega)<\infty$.%\footnote{The original definition from \cite{HHH07} requires that only submeasures $\varphi$ satisfying $\iI=\exh(\varphi)$ are bounded, but the one given here is equivalent by the proof of \cite[Corollary 4.6]{Zuc25}.}
\end{definition}

It turns out that in the above definition one can relax the equality $\iI=\exh(\varphi)$ to the inclusion $\iI\sub\exh(\varphi)$\footnote{The equivalent condition given in Proposition \ref{prop:tot_bnded_equiv_def} enables us to extend the definition of totally bounded ideals beyond the class of analytic P-ideals, we will however deliberately not do this.}.

\begin{proposition}\label{prop:tot_bnded_equiv_def}
    For an analytic P-ideal $\iI$ on $\omega$, the following conditions are equivalent:
    \begin{enumerate}[(a)]
        \item $\iI$ is totally bounded,
        \item for every lsc submeasure $\varphi$ on $\omega$, if $\iI\sub\exh(\varphi)$, then $\varphi(\omega)<\infty$.
    \end{enumerate}
\end{proposition}
\begin{proof}
    The argument is literally the same as for \cite[Corollary 4.6]{Zuc25}.
\end{proof}

It was proved in \cite[Proposition 3.18]{HHH07} that the ideals $\tr(\nN)$ and $\zZ$ are totally bounded. The second author further observed in \cite[Theorem 5.5]{Zuc25} that among density ideals totally bounded ones are precisely those which are isomorphic to \erdos--Ulam ideals. Corollary \ref{cor:density_nik_strong_web} thus yields the following result (see also Corollaries \ref{cor:tot_bnded_full_char} and \ref{cor:tot_bnded_full_char2} below).

\begin{corollary}\label{cor:density_nik_strong_web_totally bounded}
    Let $\iI$ be a density ideal on $\omega$. Then, the following conditions are equivalent:
    \begin{enumerate}
%        \item $\aA_{\iI^*}$ has the Nikodym property,
        \item $\iI$ has the Nikodym property,
        \item $\iI$ has the strong Nikodym property,
        \item $\iI$ has the web Nikodym property,
        \item $\iI$ is totally bounded.
    \end{enumerate}    
\end{corollary}

Since $\exh(\varphi)\sub\fin(\varphi)$ for any lsc submeasure $\varphi$ on $\omega$, it follows that every analytic P-ideal $\iI$ with (LGBPl) is totally bounded. We will show that these notions are in fact equivalent for analytic P-ideals. For this, we first prove a strengthening of \cite[Lemma 2.6]{Sak18} (with a similar proof).

\begin{lemma}\label{lemma:fin_exh}
For every lsc submeasure $\varphi$ on $\omega$ such that $\varphi(\omega)=\infty$, there exists an lsc submeasure $\psi$ on $\omega$ satisfying $\psi(\omega)=\infty$ and $\fin(\varphi)\sub\exh(\psi)$.
\end{lemma}
\begin{proof}
Let $\varphi$ be an lsc submeasure on $\omega$ such that $\varphi(\omega)=\infty$. By induction on $n\io$, we can define a sequence $\seqn{k_n}\sub\omega$ such that $k_0 = 0$ and $\varphi\big(k_{n+1}\setminus k_n\big) \geq (n+1)^2 $ for each $n\io$.
For each $n\io$, set
$X_n = k_{n+1}\setminus k_n$ and let $\psi_n$ be the submeasure on $\omega$ defined by
\[\psi_n(A) = \varphi\big(A\cap X_n\big)/(n+1)\]
for every $A\sub\omega$. Then, we define the lsc submeasure $\psi$ on $\omega$ by setting
\[\psi(A) = \sup_{n\io} \psi_n\big(A \cap X_n\big) \]
for each $A\sub\omega$. We have $\psi(\omega)=\infty$ and $\omega\not\in\exh(\psi)$, since $\psi_n\big(X_n\big)\geq n+1$ for every $n\io$.

We will show that $\fin(\varphi)\sub\exh(\psi)$. Let $A\in\fin(\varphi)$. Then, for each $n\io$ we have
\[\varphi\big(A\cap X_n\big)\leq \varphi(A)<\infty,\]
so $\psi_n(A)\leq \varphi(A)/(n+1)$. Thus,
\[\lim_{n\to\infty} \psi(A \setminus n) = \lim_{n\to\infty}\sup_{k\ge n}\psi_k(A)\le\lim_{n\to\infty}\sup_{k\ge n}\varphi(A)/(k+1)=\lim_{n\to\infty}\varphi(A)/(n+1)=0,\]
that is, $A\in\exh(\psi)$.
\end{proof}

\begin{theorem}\label{thm:tot_bounded_lgbps}
An analytic P-ideal on $\omega$ is totally bounded if and only if it has the Local-to-Global Boundedness Property for lsc submeasures.
\end{theorem}
\begin{proof}
Let $\iI$ be an analytic P-ideal. We only have to prove that if $\iI$ does not have (LGBPl), then it is not totally bounded. Assume that $\iI\sub\fin(\varphi)$ for some lsc submeasure $\varphi$ on $\omega$ such that $\varphi(\omega)=\infty$. By Lemma \ref{lemma:fin_exh}, there is an lsc submeasure $\psi$ on $\omega$ satisfying $\psi(\omega)=\infty$ and $\fin(\varphi)\sub\exh(\psi)$. Therefore, $\iI\sub\exh(\psi)$, and so $\iI$ is not totally bounded, by Proposition \ref{prop:tot_bnded_equiv_def}.
\end{proof}

\begin{remark}
By the proof of Lemma \ref{lemma:fin_exh}, if we have a non-pathological lsc submeasure $\varphi$ such that $\varphi(\omega)=\infty$, then there is a non-pathological lsc submeasure $\psi$ satisfying $\psi(\omega)=\infty$ and $\fin(\varphi)\sub\exh(\psi)$. Therefore, by Theorem \ref{thm:zuch}.(1), we have in this case $\fin(\varphi)\in\aA\nN$.
\end{remark}

Note that Proposition \ref{prop:lgbps_char_mazur} and Theorem \ref{thm:tot_bounded_lgbps} imply together the well-known fact that no $\F_\sigma$ P-ideal is totally bounded (see \cite[page 590]{HHH07}).

We will also need the following general notion from the theory of Boolean algebras.

\begin{definition}\label{def:splitting}
A subset $\bB$ of a Boolean algebra $\aA$ is \textit{splitting} if for every non-zero $A\in\aA$ there is $B\in\bB$ such that $A\sm B\neq0_\aA$ and $A\wedge B\neq0_\aA$.
\end{definition}

It was proved by Hern\'{a}ndez-Hern\'{a}ndez and Hru\v{s}\'{a}k \cite[Lemma 3.17]{HHH07} that, for an analytic P-ideal $\iI$ on $\omega$, the existence of a countable splitting family in the quotient Boolean algebra $\wo/\iI$ implies that $\iI$ is totally bounded. We will show below that the converse holds true as well.

%\begin{corollary}\label{cor:tot_bounded_bwp}
%An analytic P-ideal on $\omega$ is totally bounded if and only if it does not have the Bolzano--Weierstrass property.
%\end{corollary}
%\begin{proof}
%By Theorem \ref{thm:tot_bounded_nikodym} and Proposition \ref{prop:lgbps_char}.
%\end{proof}

The last additional property considered in this subsection is connected with the classical Bolzano--Weierstrass theorem asserting that every bounded sequence of real numbers contains a convergent subsequence. It was first studied by Filip\'{o}w \textit{et al.} \cite[Section 2]{FMRS07}.

\begin{definition}[\cite{FMRS07}]\label{def:bwp}
An ideal $\iI$ on $\omega$ has the \textit{Bolzano--Weierstrass property} (in short, \textit{(BWP)}) if for any bounded sequence $\seqn{x_n}$ of real numbers there is $A\not\in\iI$ such that $\seq{x_n}{n\in A}$ is $\iI$-convergent, that is, there is $x\in\R$ such that $\big\{n\in A\colon \big|x_n-x\big|\ge\eps\big\}\in\iI$ for every $\eps>0$.
\end{definition}

For basic information on the property see \cite{BFMS11}, \cite{FMRS07}, or \cite{FS10}.

The following corollary collects the most important conditions equivalent to the total boundedness. Note that most equivalences were already proved by Filip\'{o}w \textit{et al.} \cite{BFMS13} and \cite{FMRS07}. %Implication (5)$\Rightarrow$(1) was also already proved by Hern\'{a}ndez-Hern\'{a}ndez and Hru\v{s}\'{a}k \cite[Lemma 3.17]{HHH07}.

\begin{corollary}\label{cor:tot_bnded_full_char}
%For every lsc submeasure $\varphi$ on $\omega$ and the ideal $\iI=\exh(\varphi)$, the following are equivalent:
For analytic P-ideal $\iI$, the following are equivalent:
\begin{enumerate}
    \item $\iI$ is totally bounded,
    \item $\iI$ has (LGBPl),
    \item $\iI$ cannot be extended to an $\F_\sigma$-ideal,
    \item $\iI$ does not have (BWP),
    \item $\wo/\iI$ has a countable splitting family,
    \item $\conv\le_K\iI$.
\end{enumerate}
\end{corollary}
\begin{proof}
Equivalences (1)$\Leftrightarrow$(2)$\Leftrightarrow$(3) follow from Theorem \ref{thm:tot_bounded_lgbps} and Proposition \ref{prop:lgbps_char_mazur}. Equivalences (3)$\Leftrightarrow$(4)$\Leftrightarrow$(5) follow from \cite[Theorems 4.2 and 5.1]{FMRS07}. Finally, equivalence (3)$\Leftrightarrow$(6) is covered by \cite[Proposition 6.5]{BFMS13}.
\end{proof}

\begin{remark}
    As the ideal $\tr(\nN)$ is totally bounded, Corollary \ref{cor:tot_bnded_full_char} immediately implies the (known) inequality $\conv\le_K\tr(\nN)$. 
\end{remark}
%{\color{red} Chciałem tu napisać że dostajemy $\conv\le_K\tr(\nN)$ jako nowy wniosek, ale to jednak wynika wprost z \cite[Lemma 3.17]{HHH07} i równoważności $(5)$ i $(7)$.\\}

In Corollary \ref{cor:tot_bnded_full_char2} we provide further ``non-atomic'' conditions for an ideal $\exh(\varphi)$ to be totally bounded. Additional conditions equivalent to the Bolzano--Weierstrass property (and so to the lack of the total boundedness) of analytic P-ideals were also given e.g. in Filip\'{o}w \textit{et al.} \cite{FMRS11} and \cite[Section 4]{FS10c} (in terms of colorings of $\omega$ and $[\omega]^2$), and in \cite{FMRS12} (in terms of convergence of functions), etc.

%Corollary \ref{cor:tot_bnded_full_char} easily yields also the following result.

%\begin{corollary} ---- ZNANE HRUSAKOWI (patrz strona 590 w HHH)!
%No $\F_\sigma$ P-ideal is totally bounded.
%\end{corollary}

%{\color{red} może dać tu jakieś dodatkowe przykłady spośród klasycznych ideałów które teraz okazują się być totally bounded??????? np. uniform density ideal, może coś z tych matrix summability etc.
%
%- nasz nowy rezultat moze moze jedynie dawac nowe wyniki o NIE byciu totally bounded; no i pomyślałem o tym i NOWY przykład ideału nie totally bounded musiałby być poza klasą AN (inaczej mieści się w moim wyniku); teraz z mojego tw. wynika że takich ideałów density nie ma, a jeśli chodzi o niepatologiczne to tego dotyczyło pytanie z mojej pracy, na które dopiero odpowiadamy ideałem hipergrafowym w tej sekcji; zaś patologicznych to znamy parę przykładów i chyba wszystkie są $F_\sigma$ P-ideałami, więc w oczywisty sposób nie są tot. bounded}

Moreover, Filip\'{o}w \textit{et al.} \cite[Theorem 4.3]{FMRS07} proved that under the assumption of the Continuum Hypothesis an analytic P-ideal $\iI$ has the Bolzano--Weierstrass property if and only if $\iI$ can be extended to a maximal P-ideal. Combining this with Corollary \ref{cor:tot_bnded_full_char}, we get yet another (consistent) characterization of totally bounded ideals.

\begin{corollary}\label{cor:ch_tot_bnded_char}
Assume the Continuum Hypothesis. Let $\iI$ be an analytic P-ideal. Then, $\iI$ is totally bounded if and only if $\iI$ cannot be extended to a maximal P-ideal.
\end{corollary}

Note that Corollary \ref{cor:ch_tot_bnded_char} requires some set-theoretic assumption, as there are models of \textsf{ZFC} without any maximal P-ideals (that is, equivalently, without P-points on $\omega$), see e.g. \cite{Wim82}.

%It was also proved in \cite[Proposition 3.4]{FMRS07} that every $\F_\sigma$-ideal has (BWP).

\subsection{Non-atomic ideals and the Strong Nested Partition Property}

Let us recall the following standard definition of strongly non-atomic submeasures and its counterpart for ideals.

\begin{definition}
A submeasure $\varphi$ on $\omega$ is said to be \textit{strongly non-atomic} if for every
$\eps > 0$ there exists a finite partition $A_1,\ldots, A_k$ of $\omega$ such that $\varphi\big(A_i\big)<\eps$ for each $i=1,\ldots,k$.
\end{definition}

Given two partitions $\pP$ and $\pP'$ of $\omega$, we write $\pP'\preceq\pP$ if for any $A\in\pP$ and $B\in\pP'$ we either have $A\cap B=\emptyset$ or $B\sub A$.

\begin{definition}\label{def:nonat_ideal}
    An ideal $\iI$ on $\omega$ is \textit{non-atomic} if there exists a $\preceq$-decreasing sequence $\seqn{\pP_n}$ of finite partitions of $\omega$ such that for every $\subseteq$-decreasing sequence $\seqn{A_n\in\pP_n}$ and every set $E\sub\omega$ such that $E\sm A_n$ is finite for every $n\io$ we have $E\in\iI$.
\end{definition}

For basic information concerning non-atomic submeasures and ideals, see \cite{BFMS11}, \cite{DL08}, \cite{DL08ii}, \cite{DL09}, \cite{FMRS07}, \cite{FS10s}. In particular, note that for a strongly non-atomic submeasure $\varphi$ on $\omega$ the ideal $\zZ(\varphi)$ is also non-atomic, but the converse may not hold---it holds provided that $\varphi=\psi^\bullet$ for some lsc submeasure $\psi$ on $\omega$ (see \cite[Section 3]{DL08ii} for details).

The relations between the Nikodym property and non-atomic submeasures and ideals were first studied by Drewnowski, Florencio, and Pa\'ul in \cite{DFP96}, where it was proved, among others, that for a strongly non-atomic submeasure $\varphi$ on $\omega$ the ideal $\zZ(\varphi)$ has (LGBPs) and so the Nikodym property (hence, in particular, the asymptotic density zero ideal $\zZ$ has both of the properties, cf. \cite[Proposition 6]{DFP94}). Let us note here that a bit earlier Pa\v{s}t\'{e}ka \cite{Pas90} proved that such ideal $\zZ(\varphi)$ has (PSP).

The following non-atomicity property of ideals, shortly called (NPP), was introduced by Stuart \cite{Stu07} (for rings of sets). We also introduce its strong variant, because it is somewhat easier to handle than (NPP) and, as shown in Proposition \ref{prop:snpp_nonatomic}, it unifies both (NPP) and non-atomicity of ideals. 
%The following notion is a strengthening of the Nested Partition Property (in short, (NPP)) defined by Stuart in \cite{Stu07} (for rings of sets). We introduce it because it is somewhat easier to handle than (NPP) and, as shown in Proposition \ref{prop:snpp_nonatomic}, it unifies both (NPP) and non-atomicity of ideals. 

\begin{definition}\label{def:npp}
    An ideal $\iI$ on $\omega$ has the \textit{Nested Partition Property} (in short, \textit{(NPP)}), if there exists a $\preceq$-decreasing sequence $\seqn{\pP_n}$ of finite partitions of $\omega$ such that for every sequence $\seqn{E_n}$ of pairwise disjoint elements of $\iI$ there is an infinite set $M\io$ such that for every $\subseteq$-decreasing sequence $\seqn{A_n\in\pP_n}$ we have
    \[\bigcup_{n\in M}\big(E_n\cap A_n\big)\in\iI.\]
    If the set $M$ can be always chosen to be $\omega$, then we say that $\iI$ has the \textit{Strong Nested Partition Property} (in short, \textit{(SNPP)}).
\end{definition}

Of course, if an ideal has (SNPP), then it has (NPP).% {\color{red}The converse however does not hold. (????)}

\begin{proposition}\label{prop:snpp_nonatomic}
If an ideal $\iI$ on $\omega$ has (SNPP), then $\iI$ is non-atomic.
\end{proposition}
\begin{proof}
Let $\seqn{\pP_n}$ be the $\preceq$-decreasing sequence of finite partitions of $\omega$ witnessing (SNPP) of $\iI$. Assume that $\seqn{A_n\in\pP_n}$ is a $\subseteq$-decreasing sequence and that $E\sub\omega$ is such that $E\setminus A_n$ is finite for every $n\io$. Moreover, without loss of generality we may assume that $\pP_0=\{\omega\}$ and $A_0=\omega$. Let $\kappa=\Big|\bigcap_{n\io}A_n\Big|$ and enumerate $\bigcap_{n\io}A_n=\big\{x_m\colon m<\kappa\big\}$. For each $m\io$ we define
\[E_m = \Big(\big\{x_m\big\}\cup\big(A_m\setminus A_{m+1}\big)\Big)\cap E\]
if $m<\kappa$, and
\[E_m = \big(A_m\setminus A_{m+1}\big)\cap E\]
if $\kappa\le m<\omega$. Then, the sets $E_m$'s are pairwise disjoint, and we have $E=\bigcup_{m\io}E_m$, since
\[\bigcup_{m<\kappa}\Big(\big\{x_m\big\}\cup\big(A_m\setminus A_{m+1}\big)\Big)\cup\bigcup_{\kappa\le m<\omega}\big(A_m\setminus A_{m+1}\big)=A_0=\omega.\]
Furthermore, for each $m\io$ the set $E_m$ is finite, since $E_m\sub\big(E\setminus A_{m+1}\big)\cup\big\{x_m\big\}$ for $m<\kappa$ and $E_m\sub\big(E\setminus A_{m+1}\big)$ for $m\ge\kappa$, thus in particular $E_m\in\iI$. For every $m\io$ we also have $E_m\sub A_m$, so $E_m\cap A_m=E_m$. Consequently, 
\[E=\bigcup_{m\io}E_m=\bigcup_{m\io}\big(E_m\cap A_m\big)\in\iI,\]
as required.
\end{proof}

The next corollary, being a supplement of Corollary \ref{cor:tot_bnded_full_char}, shows that the converse to Proposition \ref{prop:snpp_nonatomic} also holds, provided that a considered ideal is of the form $\exh(\varphi)$ for some lsc submeasure $\varphi$ on $\omega$.

\begin{corollary}\label{cor:tot_bnded_full_char2}
For every lsc submeasure $\varphi$ on $\omega$ and the ideal $\iI=\exh(\varphi)$, conditions (1)--(6) of Corollary \ref{cor:tot_bnded_full_char} are equivalent to the following ones:
\begin{enumerate}
\setcounter{enumi}{6}
    \item $\varphi^\bullet$ is strongly non-atomic,
    \item $\iI$ is non-atomic,
    \item $\iI$ has (SNPP),
    \item $\iI$ has (LGBPs).
\end{enumerate}
\end{corollary}
\begin{proof}
Equivalence (4)$\Leftrightarrow$(7) follows from \cite[Theorem 3.6]{FMRS07}. Implication (9)$\Rightarrow$(8) follows from Proposition \ref{prop:snpp_nonatomic}. Implication (8)$\Rightarrow$(7) is given by \cite[Corollary 3.2]{DL08ii}. For implication (7)$\Rightarrow$(9) see \cite[proof of Proposition 3]{Stu07}. Implication (7)$\Rightarrow$(10) follows from \cite[Proposition 4.1]{DFP96}, whereas implication (10)$\Rightarrow$(2) is trivial.
\end{proof}

Let us underline implication (2)$\Leftrightarrow$(10) of the above result as a separate corollary.

\begin{corollary}\label{cor:p_anal_lgbps_lgbpl}
    If $\iI$ is an analytic P-ideal on $\omega$, then $\iI$ has (LGBPs) if and only if $\iI$ has (LGBPl).
\end{corollary}

\begin{remark}\label{rem:lacunary}
Recall that an infinite set $\{a_0<a_1<a_2<...\}\sub\omega$ is \textit{lacunary} if $\lim_{n\to\infty}a_{n+1}-a_n=\infty$. Let $\lL$ denote the ideal generated by all lacunary subsets of $\omega$. By \cite[Proposition 2]{SF81} (see also \cite{Agn47}, \cite{Aue30}, or \cite{KN17}), $\lL$ is not contained in any summable ideal, so it has (ASP), but by \cite[Example 4.2]{DP00} it does not have the Nikodym property.

Drewnowski and \L uczak \cite[Theorem 4.1]{DL08ii} proved that if an lsc submeasure $\varphi$ on $\omega$ is such that $\lL\sub\exh(\varphi)$, then the core $\varphi^\bullet$ is strongly non-atomic. Since $\conv\le_K\lL$ (by \cite[Lemma 3.2]{KN17} and \cite[page 57]{MA09}), the combination of Corollaries \ref{cor:tot_bnded_full_char} and \ref{cor:tot_bnded_full_char2} provides a generalization of the latter fact: if an lsc submeasure $\varphi$ on $\omega$ is such that $\conv\le_K\exh(\varphi)$, then $\varphi^\bullet$ is strongly non-atomic.
\end{remark}

It was proved in \cite{Stu07} that every ideal on $\omega$ with (NPP) has the Nikodym property, the converse however does not hold---in the next section we provide an example of an ideal which has the Nikodym property but not (NPP), see Corollary \ref{cor:nikodym_no_npp}. Below we show that an ideal with (SNPP) has (LGBPl).%, which, by \cite[Proposition 2.2]{DFP96}, is also a property stronger than the Nikodym property).

\begin{proposition}\label{prop:inf_submeas_partition_scheme}
If $\varphi$ is an lsc submeasure on $\omega$ such that $\varphi(\omega)=\infty$, then for every $\preceq$-decreasing sequence $\seqn{\pP_n}$ of finite partitions of $\omega$ there exist a $\subseteq$-decreasing sequence $\seqn{A_n\in\pP_n}$ and a sequence $\seqn{F_n}$ of pairwise disjoint finite subsets of $\omega$ such that $F_n\sub A_n$ and $\varphi\big(F_n\big)\geq n+1$ for every $n\io$.
\end{proposition}
\begin{proof}
Fix an lsc submeasure $\varphi$ on $\omega$ such that $\varphi(\omega)=\infty$. Let $\seqn{\pP_n}$ be a $\preceq$-decreasing sequence of finite partitions of $\omega$. By the subadditivity of $\varphi$, there exists $A_0\in\pP_0$ such that $\varphi\big(A_0\big)=\infty$. As $\varphi$ is lsc, we can choose a finite set $F_0\sub A_0$ with $\varphi\big(F_0\big)\geq 1$. Now, since $\varphi\big(A_0\setminus F_0\big)=\infty$, there exists $A_1\in\pP_1$ such that $A_1\sub A_0$ and $\varphi\big(A_1\setminus F_0\big)=\infty$. Again, choose a finite set $F_1\sub\big(A_1\setminus F_0\big)$ with $\varphi\big(F_1\big)\geq 2$. Since $\varphi\big(A_1\setminus\big(F_0\cup F_1\big)\big)=\infty$, there exists $A_2\in\pP_2$ such that $A_2\sub A_1$ and $\varphi\big(A_2\setminus\big(F_0\cup F_1\big)\big)=\infty$. And so on.

By repeating the above argument, we get a $\subseteq$-decreasing sequence $\seqn{A_n\in\pP_n}$ and a sequence $\seqn{F_n}$ of pairwise disjoint finite subsets of $\omega$ such that $F_n\sub A_n$ and $\varphi\big(F_n\big)\geq n+1$ for every $n\io$.
\end{proof}

\begin{proposition}\label{prop:nonatomic_lgbps}
    Every non-atomic ideal on $\omega$ has (LGBPl).
\end{proposition}
\begin{proof}
    Assume that for an ideal $\iI$ on $\omega$ we have $\iI\sub\fin(\varphi)$ for some lsc submeasure $\varphi$ on $\omega$ such that $\varphi(\omega)=\infty$. To show that $\iI$ is not non-atomic, let $\seqn{\pP_n}$ be any $\preceq$-decreasing sequence of finite partitions of $\omega$. Let $\seqn{A_n\in\pP_n}$ and $\seqn{F_n}$ be sequences as in Proposition \ref{prop:inf_submeas_partition_scheme}.
    For $E=\bigcup_{n\io} F_n$, we have $\varphi(E)=\infty$, so $E\notin\iI$, even though $E\sm A_n$ is finite for every $n\io$, which proves that $\iI$ is indeed not non-atomic.%is a contradiction. 
\end{proof}

Combining Propositions \ref{prop:snpp_nonatomic} and \ref{prop:nonatomic_lgbps}, we get the following aforementioned result.

\begin{corollary}\label{cor:snpp_lgbps}
     Every ideal on $\omega$ with (SNPP) has (LGBPl).
\end{corollary}

Let us note that using a similar method one can get the following variation of Proposition \ref{prop:inf_submeas_partition_scheme} (cf. \cite[Lemma 3.17]{HHH07}); we use here the convention that, for a set $A\sub\omega$, we have $A^1=A$ and $A^{-1}=\omega\sm A$.

\begin{proposition}\label{prop:inf_submeas_scheme}
If $\varphi$ is an lsc submeasure on $\omega$ such that $\varphi(\omega)=\infty$, then for every sequence $\seqn{A_n}$ of subsets of $\omega$ there exist sequences $\seqn{\eps_n}\in\{-1,1\}^\omega$ and $\seqn{F_n}\sub Fin$ with the following properties:
\begin{enumerate}[(a)]
\item $F_n\cap F_m=\emptyset$ for every $n\neq m\io$,
\item $F_n\sub\bigcap_{k=0}^n A_n^{\eps_n}$ for every $n\io$,
\item $\varphi\big(F_n\big)\geq n$ for every $n\io$.
\end{enumerate}
\end{proposition}

\subsection{Hypergraph ideals}\label{sec:hypergraph}

In this subsection we construct a family of $\frakc$ many pairwise non-isomorphic non-pathological ideals on $\omega$ which have the Nikodym property but are not totally bounded or, even stronger, do not have (NPP). In order to do this, we exploit and expand the construction due to Alon, Drewnowski, and \L uczak \cite{ADL09} of an ideal which has the Nikodym property but is not non-atomic. This construction relies on the notion of a hypergraph ideal, presented below.

Recall that a \textit{hypergraph} is a pair $H=\big(V(H),E(H)\big)$, where $V(H)$ is a finite set and $E(H)\sub\wp(V(H))\sm\{\emptyset\}$.

\begin{definition}\label{def:hypergraph}
    A submeasure $\varphi$ on $\omega$ is \textit{hypergraph} if there exist a sequence $\seqn{H_n}$ of hypergraphs, a sequence $\seqn{G_n}$ of finite non-empty disjoint subsets of $\omega$, and for each $n\io$ a bijection $b_n\colon G_n\to V\big(H_n\big)$ such that for every $A\in\wo$ we have
    \[\varphi(A)=\sup_{n\io}\sup_{e\in E(H_n)}\frac{\big|b_n\big[A\cap G_n\big]\cap e\big|}{|e|}.\]
    In this case, we will also say that $\varphi$ is \textit{induced} by the sequence $\seqn{H_n}$ of hypergraphs.

    An ideal $\iI$ on $\omega$ is \textit{hypergraph} if there exists a hypergraph submeasure $\varphi$ on $\omega$ such that $\iI=\exh(\varphi)$.
\end{definition}

\begin{remark}\label{rmk:hypergraph}
Note that every hypergraph ideal is an analytic P-ideal. Moreover, every hypergraph ideal is a non-pathological generalized density ideal. On the other hand, every density ideal induced by a sequence of uniform measures is hypergraph.
\end{remark}

Let $H$ be a hypergraph. Recall that the \textit{chromatic number} $\chi(H)$ of $H$ is the minimal number $n\io$ for which there exists a partition $C_1,\ldots,C_n$ of $V(H)$ such that for every $e\in E(H)$ and every $1\le i\le n$ we have $e\not\sub C_i$. For a subset $W\sub V(H)$ we set $H\rstr W=\big(W,E(H)\cap\wp(W)\big)$. %$H\rstr W=(W,E(H)\rstr W)$, where $E(H)\rstr W=\{e\in E(H)\colon e\sub W\}$.
We will need the following lemma.
\begin{lemma}\label{lem:chromatic}
    Let $\seqn{H_n}$ be a sequence of hypergraphs such that $\lim_{n\to\infty}\chi\big(H_n\big)=\infty$. Let $\seqn{G_n}$ be a sequence of finite non-empty disjoint subsets of $\omega$ and for each $n\io$ let $b_n\colon G_n\to V\big(H_n\big)$ be a bijection.
    Let $A_1,\ldots,A_k$ be a finite partition of $\omega$ into infinite sets. Then, there are $1\le j_0\le k$ and $I\in\cso$ such that
    \[\lim_{\substack{n\to\infty\\n\in I}}\chi\Big(H_n\rstr b_n\big[G_n\cap A_{j_0}\big]\Big)=\infty.\]
\end{lemma}
\begin{proof}
    Assume for the sake of contradiction that for each $1\le j\le k$ there is $K_j\io$ such that
    \[\chi\Big(H_n\rstr b_n\big[G_n\cap A_j\big]\Big)\le K_j\]
    for every $n\io$. Then, for every $n\io$ we have
    \[\chi\big(H_n\big)\le\sum_{j=1}^k\chi\Big(H_n\rstr b_n\big[G_n\cap A_j\big]\Big)\le\sum_{j=1}^k K_j,\]
    so
    \[\sup_{n\io}\chi\big(H_n\big)\le\sum_{j=1}^k K_j<\infty,\]
    which is impossible.
\end{proof}
%Recall that the \textit{chromatic number} $\chi(H)$ of a hypergraph $H$ is the minimal number $n\io$ for which there exists a partition $C_1,\ldots,C_n$ of $V(H)$ such that no $e\in E(H)$ is contained in some $C_i$.

The idea behind the proof of the next proposition closely follows the first part of the proof of \cite[Theorem 2.3]{ADL09}, showing that for no submeasure $\varphi$ on $\omega$ induced by a sequence $\seqn{H_n}$ of hypergraphs such that $\lim_{n\to\infty}\chi\big(H_n\big)=\infty$ the ideal $\exh(\varphi)$ is non-atomic.

\begin{proposition}\label{prop:hypergraph_npp}
    If a hypergraph submeasure $\varphi$ on $\omega$ is induced by a sequence $\seqn{H_n}$ of hypergraphs such that $\lim_{n\to\infty}\chi\big(H_n\big)=\infty$, then the ideal $\iI=\exh(\varphi)$ does not have the Nested Partition Property.
\end{proposition}
\begin{proof}
    Let us fix sequences $\seqn{G_n}$ and $\seqn{b_n}$ as in Definition \ref{def:hypergraph}. Let also $\seqn{\pP_n}$ be a $\preceq$-decreasing sequence of finite partitions of $\omega$. We will show that $\seqn{\pP_n}$ does not satisfy the requirements of Definition \ref{def:npp}, which will prove that $\iI$ does not have (NPP).

    We start as follows. By Lemma \ref{lem:chromatic} there are $A_0\in\pP_0$ and $I_0\in\ctblsub{\omega}$ such that
    \[\lim_{\substack{n\to\infty\\n\in I_0}}\chi\Big(H_n\rstr b_n\big[G_n\cap A_0\big]\Big)=\infty.\]
    Let $n_0\in I_0$ be such that there exists $e_0\in E\Big(H_{n_0}\rstr b_{n_0}\big[G_{n_0}\cap A_0\big]\Big)$ (such $n_0$ exists, since otherwise the above limit would be $1$). %Set $E_0=b_{n_0}^{-1}\big[e_0\big]$.

    Again, by Lemma \ref{lem:chromatic}, there are  $A_1\in\pP_1$, $A_1\sub A_0$, and $I_1\in\big[I_0\sm(n_0+1)\big]^\omega$ such that
    \[\lim_{\substack{n\to\infty\\n\in I_1}}\chi\Big(H_n\rstr b_n\big[G_n\cap A_1\big]\Big)=\infty.\]
    Let $n_1\in I_1$ be such that there exists $e_1\in E\Big(H_{n_1}\rstr b_{n_1}\big[G_{n_1}\cap A_1\big]\Big)$. %Set $E_1=b_{n_1}^{-1}\big[e_1\big]$. Note that $E_0\cap E_1=\emptyset$.

    Repeat the above argument until you get an $\subseteq$-decreasing sequence $\seqk{A_k\in\pP_k}$, a strictly increasing sequence $\seqk{n_k\io}$, and a sequence $\seqk{e_k\in E\big(H_{n_k}\big)}$ such that $e_k\in E\Big(H_{n_k}\rstr b_{n_k}\big[G_{n_k}\cap A_k\big]\Big)$ for each $k\io$.

    For each $k\io$ put $B_k=b_{n_k}^{-1}\big[e_k\big]$, so $B_k\sub G_{n_k}\cap A_k$; as $B_k$ is a finite set, it belongs to $\iI$. For each $M\in\cso$ let
    \[B_M=\bigcup_{k\in M}\big(B_k\cap A_k\big)=\bigcup_{k\in M}B_k,\]
    so that $B_M\cap G_{n_k}=B_k$ for each $k\in M$.
    For every $M\in\cso$ we have
    \[\varphi^\bullet\big(B_M\big)=\limsup_{n\to\infty}\sup_{e\in E(H_n)}\frac{\Big|b_n\big[B_M\cap G_n\big]\cap e\Big|}{|e|}\ge\limsup_{\substack{k\to\infty\\k\in M}}\frac{\Big|b_{n_k}\big[B_k\cap G_{n_k}\big]\cap e_k\Big|}{|e_k|}=\]
    \[=\limsup_{\substack{k\to\infty\\k\in M}}\frac{\big|b_{n_k}\big[B_k\big]\cap e_k\big|}{\big|e_k\big|}=\limsup_{\substack{k\to\infty\\k\in M}}\frac{\big|e_k\cap e_k\big|}{\big|e_k\big|}=1,\]
    so $B_M\not\in\iI$.%{\color{red} To już niepotrzebne do braku NPP: even though $B_M\sm A_n$ is finite for every $n\io$.}
    %There is $N_0\io$ such that $\chi\big(H_n\big)>\big|\pP_0\big|$ for every $n\ge N_0$. It follows that there are $A_0\in\pP_0$ and $I_0\in\ctblsub{\omega}$ such that for every $n\in I_0$ there is $e\in E\big(H_n\big)$ with $b_{n}^{-1}[e]\sub A_0$...............
    %\[\lim_{\substack{n\to\infty\\n\in I_0}}\chi\big(H_n\rstr A_0\big)=\infty.\]
\end{proof}

\begin{corollary}\label{cor:hypergraph_tot_bnd}
    If a hypergraph submeasure $\varphi$ on $\omega$ is induced by a sequence $\seqn{H_n}$ of hypergraphs such that $\lim_{n\to\infty}\chi\big(H_n\big)=\infty$, then the ideal $\iI=\exh(\varphi)$ is not totally bounded.
\end{corollary}
\begin{proof}
Let $\varphi$ be a hypergraph submeasure satisfying the premises. By Proposition \ref{prop:hypergraph_npp} it follows that the ideal $\iI$ does not have (NPP) and so, by Corollary \ref{cor:tot_bnded_full_char2}, it is not totally bounded.
\end{proof}

\begin{theorem}\label{thm:hypergraph_nikodym_not_tot_bnd}
There is a family of $\frakc$ many pairwise non-isomorphic  hypergraph ideals which have the Nikodym property but are not totally bounded.   
\end{theorem}
\begin{proof}
Let $\seqn{H_n}$ be the sequence of Kneser hypergraphs defined by Alon, Drewnowski, and \L uczak in \cite[Section 4]{ADL09}, and let $\varphi_{ADL}$ be the hypergraph submeasure on $\omega$ induced by the sequence $\seqn{H_n}$, with sequences $\seqn{G_n}$ and $\seqn{b_n}$ as in Definition \ref{def:hypergraph}.  Set $\iI_{ADL}=\exh\big(\varphi_{ADL}\big)$. As for every $k\io$ all the edges of $H_k$ are of size $2^k$, we can recursively define a strictly increasing sequence $\seqk{n_k}$ such that
\[\tag{$*$}
|e| > k\cdot\sum_{i<k}\big|G_{n_i}\big|
\]
holds for every $e\in E\big(H_{n_k}\big)$ and $k\io$. For any $M\in\cso$ let us set $\iI_M=\exh(\varphi_M)$, where $\varphi_M$ is the hypegraph submeasure on $\omega$ defined for every $A\in\wo$ as
\[\varphi_M(A)=\sup_{\substack{k\in M}}\sup_{e\in E(H_{n_k})}\frac{\Big|b_{n_k}\big[A\cap G_{n_k}\big]\cap e\Big|}{|e|}.\]
    
By \cite[Corollary 3.4]{ADL09} we have $\chi\big(H_{n_k}\big)\ge 2^{n_k}$ for every $k\io$, and so, by Corollary \ref{cor:hypergraph_tot_bnd}, the ideal $\iI_M$ is not totally bounded for any $M\in\cso$. Next, $\iI_{ADL}$ has the Nikodym property by \cite[Theorem 2.3]{ADL09}, and so the space $N_{\iI_{ADL}^*}$ has the finitely supported Nikodym property by Theorem \ref{thm:p_ideal_nik_equivalences}. Thus, the space $N_{\iI_{M}^*}$ has the finitely supported Nikodym property for every $M\in\cso$, by \cite[Proposition 5.2]{Zuc25} and the fact that $\iI_{ADL}\sub\iI_M$. Therefore, $\iI_{M}$ has the Nikodym property for every $M\in\cso$, again by Theorem \ref{thm:p_ideal_nik_equivalences}.

Let $\big\{M_\alpha\colon\alpha<\frakc\big\}$ be a family of infinite almost disjoint subsets of $\omega$, that is, the intersection $M_\alpha\cap M_\beta$ is finite for every $\alpha\neq\beta<\frakc$. We will show that the ideals $\big\{\iI_{M_\alpha}\colon\alpha<\frakc\big\}$ are pairwise non-isomorphic. Fix $\alpha\neq\beta<\frakc$ and a bijection $\phi\colon\omega\to\omega$. Let $\seqi{m_i}$ be an increasing enumeration of $M_\alpha\setminus M_\beta$. Let us fix an arbitrary edge $e_k\in E\big(H_{n_k}\big)$ for any $k\io$. By ($*$), for every $k\geq 2$ there exists $A_k\sub G_{n_k}$ such that $b_{n_k}\big[A_k\big]\sub e_k$, $\big|A_k\big|\geq\big|e_k\big|/2$, and 
\[\tag{$**$}\phi\big[A_k\big]\cap\bigcup_{i<k} G_{n_i} = \emptyset.\]

Let $A=\bigcup_{i\io} A_{m_i}$. We have $A\notin\iI_{M_\alpha}$, because%as for every $i\geq 2$ we have 
\[\varphi_{M_\alpha}^\bullet(A)\ge\limsup_{i\to\infty}\frac{\Big|b_{n_{m_i}}\Big[A\cap G_{n_{m_i}}\Big]\cap e_{m_i}\Big|}{\big|e_{m_i}\big|}=\limsup_{i\to\infty}\frac{\Big|b_{n_{m_i}}\big[A_{m_i}\big]\cap e_{m_i}\Big|}{\big|e_{m_i}\big|}\geq 1/2.\] 
We will show that $\phi[A]\in\iI_{M_\beta}$, which will conclude the proof. By ($**$), for every $k\geq 2$ we have
\[\phi[A]\cap G_{n_k}=\bigcup_{i\io}\big(\phi\big[A_{m_i}\big]\cap G_{n_k}\big) \sub \bigcup_{\substack{i\io\\m_i\leq k}} \phi\big[A_{m_i}\big].\]
Thus, for any $k\in M_\beta\sm M_\alpha$ such that $k\geq 2$ we get
\[\tag{$*\!*\!*$}\phi[A]\cap G_{n_k} \sub \bigcup_{i<k} \phi\big[A_i\big].\]
By ($*$), for every $k\ge2$ and $e\in E\big(H_{n_k}\big)$, we have 
\[\frac{\sum_{i<k}\big|A_i\big|}{|e|}\le\frac{\sum_{i<k}\big|G_{n_i}\big|}{|e|} < 1/k,\]
and so, for almost all $k\in M_\beta$, by ($*\!*\!*$) it holds
\[\sup_{e\in E(H_{n_k})}\frac{\Big|b_{n_k}\big[\phi[A]\cap G_{n_k}\big]\cap e\Big|}{|e|}\le\sup_{e\in E(H_{n_k})}\frac{\Big|b_{n_k}\big[\phi[A]\cap G_{n_k}\big]\Big|}{|e|}\le\]
\[\le\sup_{e\in E(H_{n_k})}\frac{\sum_{i<k}\big|A_i\big|}{|e|}<1/k.\]
Consequently, we get that $\phi[A]\in\iI_{M_\beta}$, as required.
\end{proof}

As a corollary, we get a positive answer to the second author's \cite[Question 9.1]{Zuc25}.

\begin{corollary}\label{cor:nikodym_not_tot_bnd}
There exists a non-pathological lsc submeasure $\varphi$ such that the ideal $\exh(\varphi)$ is not totally bounded and the space $N_{\exh(\varphi)^*}$ has the finitely supported Nikodym property.
\end{corollary}
\begin{proof}
 Let $M\in\cso$ be arbitrary and let $\varphi=\varphi_M$ be the corresponding submeasure defined in the proof of Theorem \ref{thm:hypergraph_nikodym_not_tot_bnd} (we can here of course also consider the submeasure $\varphi_{ADL}$). By Remark \ref{rmk:hypergraph}, $\varphi$ is a non-pathological lsc submeasure. Moreover, the ideal $\iI_M=\exh(\varphi)$ is not totally bounded but has the Nikodym property. Consequently, the space $N_{\exh(\varphi)^*}$ has the finitely supported Nikodym property by Theorem \ref{thm:p_ideal_nik_equivalences}.
\end{proof}

\begin{remark}\label{rmk:nikodym_submeas}
For every $M\in\cso$ the ideal $\iI_M$ from the proof of Theorem \ref{thm:hypergraph_nikodym_not_tot_bnd} has the Nikodym property and does not have the Local-to-Global Boundedness Property for (lsc) submeasures as it is not totally bounded. This provides yet another class of examples confirming the conjecture of Drewnowski, Florencio, and Pa\'ul \cite[Conjecture, page 147]{DFP96}; cf. Theorem \ref{thm:nikodym_lgbps}.
\end{remark}

\begin{remark}
    Let us also note that, for each $M\in\cso$, even though the core $\varphi^\bullet_M$ of the submeasure $\varphi_M$  from the proof of Theorem \ref{thm:hypergraph_nikodym_not_tot_bnd} is not strongly non-atomic (otherwise, the ideal $\iI_M$ would be totally bounded, by Corollary \ref{cor:tot_bnded_full_char2}), it still presents some degree of non-atomicity. Namely, $\varphi_M^\bullet$ is \textit{convexly non-atomic}, that is, for every set $A\in\wo$ and number $\alpha\in\big[0,\varphi_M^\bullet(A)\big]$ there is a set $B\sub A$ such that $\varphi_M^\bullet(B)=\alpha$, see \cite[Proposition 5.2 and Theorem 6.2]{DP00} and \cite[Section 3]{BL13} for details (note here that every strongly non-atomic submeasure is automatically convexly non-atomic, see \cite[page 494]{DP00}).
\end{remark}

\begin{remark}\label{rem:hypergraph_web_nikodym}
Recall that each hypergraph ideal is a P-ideal, therefore if a hypergraph ideal $\iI$ has the Nikodym property, then by Theorem \ref{thm:p_ideal_nik_equivalences} $\iI$ has the web Nikodym property. In particular, all the ideals considered in (the proof of) Theorem \ref{thm:hypergraph_nikodym_not_tot_bnd} have the web Nikodym property.
\end{remark}

Stuart \cite[page 153]{Stu07} asked if every hereditary ring of sets with the Nikodym property has the Nested Partition Property. Theorem \ref{thm:hypergraph_nikodym_not_tot_bnd} (and its proof) yields a negative answer to this question, since for every $M\in\cso$ the ideal $\iI_M$ has the Nikodym property, but it does not have (NPP) by Proposition \ref{prop:hypergraph_npp}; the same concerns the ideal $\iI_{ADL}$.

\begin{corollary}\label{cor:nikodym_no_npp}
    There exists an ideal on $\omega$ with the Nikodym property but without (NPP).
\end{corollary}

\section{The Tukey type of $(\aA\nN,\le_K)$\label{sec:tukey}}

As we mentioned earlier, Drewnowski and Pa\'ul \cite[Theorem 3.5]{DP00} proved that the Nikodym property coincides with the Positive Summability Property in the class of P-ideals on $\omega$. On the other hand, Theorem \ref{thm:p_ideal_nik_equivalences} asserts that in this class the Nikodym property of an ideal $\iI$ is equivalent for the space $N_{\iI^*}$ to have the finitely supported Nikodym property, i.e. $\iI\not\in\mathcal{AN}$. Combining these two results, we get that, given a P-ideal $\iI$ on $\omega$, $\iI$ has (PSP) if and only if $\iI\not\in\mathcal{AN}$. In this section we generalize this equivalence to the class of \textit{all} ideals on $\omega$. This result actually already follows from Filip\'ow--Szuca's \cite[Theorem 3.3]{FS10} and the proof of their \cite[Proposition 3.7]{FS10} (cf. also \cite[Theorem 11.2]{FT25}), however, for the sake of completeness, we present here a very short proof of this fact, based on the arguments used to prove \cite[Proposition 2.1 and Theorem 3.5]{DP00} (our method is thus more measure-theoretic in spirit and hence fits the more traditional research on the Nikodym property). As a corollary, we get that the order $(\aA\nN,\le_K)$, i.e. the set $\aA\nN$ together with the Kat\v{e}tov ordering $\le_K$, is Tukey equivalent to the standard order $(\omega^\omega,\le^*)$.

The aforementioned generalization, also closely related to Theorem \ref{thm:zuch}, reads precisely as follows.

\begin{theorem}\label{thm:an_summable}
For every density submeasure $\varphi$ on $\omega$ such that $\varphi(\omega)=\infty$, the ideal $\exh(\varphi)$ is contained in some summable ideal.

Consequently, for every ideal $\iI$ on $\omega$, we have $\iI\in\aA\nN$ if and only if there exists a summable ideal $\jJ$ such that $\iI\sub\jJ$.
\end{theorem}

The latter statement in Theorem \ref{thm:an_summable} can be of course also shortly expressed as follows: for every ideal $\iI$ on $\omega$, $\iI\not\in\aA\nN$ if and only if $\iI$ has (PSP).

For the proof we will need the following lemma. %Their ideas follow closely the proofs of ..
%{\color{blue}Instead of talking about charges or unbounded measures, we can simply speak about additive real-valued functions on ideals!}

\begin{lemma}\label{lemma:aN_unbounded_measure}
For every density submeasure $\varphi$ on $\omega$ such that $\varphi(\omega)=\infty$ there exist a finitely additive set function $\rho\colon\exh(\varphi)\to [0,\infty)$ and a sequence $\seqk{A_k}$ of pairwise disjoint finite subsets of $\omega$ such that $\rho(A_k)>k$ for every $k\io$.
\end{lemma}
\begin{proof}
Let $\seqn{\mu_n}$ be a sequence of disjointly supported finitely supported non-negative measures on $\omega$, with supports contained in $\omega$, and such that $\varphi=\sup_{n\io}\mu_n$. Since $\varphi(\omega)=\infty$, there exists an increasing sequence $\seqk{n_k}$ such that
\[\big\|\mu_{n_k}\big\|=\mu_{n_k}(\omega)>k\cdot 2^k\]
for every $k\io$. We define a finitely additive set function $\rho\colon\exh(\varphi)\to[0,\infty)$ by setting
\[ \rho(A) = \sum_{k=1}^{\infty} 2^{-k}\cdot \mu_{n_k}(A)\]
for each $A\in\exh(\varphi)$; note that the series in this definition converges because every $A\in\exh(\varphi)$ satisfies $\lim_{k\to\infty}\mu_{n_k}(A)=0$. For every $k\io$, setting $A_k=\supp\big(\mu_{n_k}\big)$, we have $\rho\big(A_k\big)>k$.
\end{proof}

\begin{proof}[Proof of Theorem \ref{thm:an_summable}]
Let $\varphi$ be a density submeasure on $\omega$ such that $\varphi(\omega)=\infty$, and let $\rho$ and $\seqk{A_k}$ be as in Lemma \ref{lemma:aN_unbounded_measure}. We define the function $f\colon\omega\to\R$ by setting $f(n)=\rho(\{n\})$ for each $n\io$. Since $\sum_{n\in A_k }f(n)=\rho\big(A_k\big)>k$ for every $k\io$, we have $\sum_{n\io}f(n)=\infty$. Moreover, since $\rho$ is finitely additive and non-negative, and hence monotone, we have
\[\sum_{n\in A}f(n)\leq\rho(A)<\infty\]
for every $A\in\exh(\varphi)$. Therefore, the ideal
\[\iI_f=\Big\{A\sub\omega\colon\sum_{n\in A} f(n) < \infty\Big\}\]
is a summable ideal such that $\exh(\varphi)\sub\iI_f$.

\medskip

For the second part of the statement, let $\iI\in\aA\nN$ and note that by Theorem \ref{thm:zuch}.(1) and the first part there is a summable ideal $\jJ$ containing $\iI$. The converse, on the other hand, follows from the observation that the class $\aA\nN$ is closed under taking subideals (by \cite[Proposition 5.2]{Zuc25}) and that every summable ideal is in $\aA\nN$ (see Example \ref{example:wo_no_nikodym}). 
\end{proof}

\begin{remark}
Theorem \ref{thm:an_summable} together with Corollary \ref{cor:tot_bnded_full_char} give yet another reason that no totally bounded analytic P-ideal on $\omega$ can be in the class $\aA\nN$ (see \cite[Corollary 4.6]{Zuc25}), as every summable ideal is $\F_\sigma$ and no totally bounded ideal is contained in an $\F_\sigma$ ideal.
%    Let us note that, by comparison with the characterization of totally bounded ideals from Corollary \ref{cor:tot_bnded_full_char}, it gives another way of showing that a totally bounded ideal cannot be in the class $\aA\nN$ (\cite[Corollary 4.6]{Zuc25}), since every summable ideal is $\F_\sigma$.
\end{remark}

Theorem \ref{thm:an_summable} yields two important corollaries. The first one strengthens Theorem \ref{thm:zuch}.(2), which asserts that for each ideal $\iI\in\aA\nN$ we have $\iI\le_K\zZ$, and extends the remark after Theorem \ref{thm:zuch}, stating that for a density ideal $\iI\in\aA\nN$ we have $\iI<_K\zZ$. Recall that $\tr(\nN)$ is a totally bounded non-pathological ideal on $\omega$ (by \cite[Proposition 3.18]{HHH07}) as well as we have $\tr(\nN)<_K\zZ$ (see e.g. \cite[page 728]{Zuc25}).

\begin{corollary}\label{cor:tr}
For every $\iI\in\aA\nN$ we have $\iI\leq_K\tr(\nN)$ and so $\iI<_K\zZ$.
\end{corollary}
\begin{proof}
By \cite[Proposition 3.5]{HHH07}, for every summable ideal $\jJ$ we have $\jJ\le_K\tr(\nN)$. The conclusion now follows from Theorem \ref{thm:an_summable} and the above discussion.
\end{proof}

The second corollary is a positive answer to the second author's \cite[Question 9.4]{Zuc25}. For this, we first need to recall the definitions regarding the Tukey order. 

Let $(P, \leq_P)$ and $(Q, \leq_Q)$ be partial orders. A function $f\colon P\to Q$ is said to be a \textit{Tukey reduction} if for every $q_0\in Q$ there exists $p_0\in P$ such that, for every $p\in P$, the following implication holds:
\[f(p)\leq_Q q_0 \: \Longrightarrow \: p \leq_P p_0, \]
i.e. if the preimages by $f$ of bounded subsets of $Q$ are bounded subsets of $P$. We write $P\preceq_T Q$ if there exists a Tukey reduction from $P$ to $Q$. We say that $P$ and $Q$ are \textit{Tukey equivalent}, which we denote by $P\equiv_T Q$, if $P\preceq_T Q$ and $Q\preceq_T P$. Recently, several authors considered also the Galois--Tukey order $\preceq_{GT}$, which is stronger than $\preceq_T$ in the sense that $P\preceq_{GT} Q$ implies $P\preceq_T Q$; see, e.g., \cite{HLZ} for definitions and basic information. Again, we write $P\equiv_{GT} Q$, if $P\preceq_{GT} Q$ and $Q\preceq_{GT} P$. %We refer the reader to \cite[Section 6]{Hru11} for basic information and further references concerning the Tukey ordering of ideals. {\color{red} To cytowanie chyba nic nie wnosi tutaj!! Tam jest mowa typach Tukeyowskich pojedynczych idealow, a tutaj o typie calej klasy idealow.}

\begin{corollary}\label{cor:tukey}
$(\aA\nN,\leq_K)\equiv_T(\oo, \leq^*)$.
\end{corollary}
\begin{proof}
Let $\Sigma$ denotes the class of all summable ideals on $\omega$. By Theorem \ref{thm:an_summable}, the order $(\Sigma,\leq_K)$ is a cofinal subset of the order $(\aA\nN,\leq_K)$, therefore we have $(\Sigma,\leq_K)\equiv_T(\aA\nN,\leq_K)$. The result now follows from He--Li--Zhang's \cite[Theorem 5.6]{HLZ}, asserting that $(\Sigma,\leq_K)\equiv_{GT}(\oo, \leq^*)$.
\end{proof}

\section{Open questions\label{sec:questions}}

In the final section of the paper we state several problems. The first two concern Nikodym concentration points, studied in Section \ref{sec:anti}.

In Example \ref{example:strong_conv_seq}, Proposition \ref{prop:nf_strong}, and Corollary \ref{cor:2c_many_strong}, using the class $\aA\nN$ of ideals on $\omega$, we described a large class of countable spaces $X$ such that if $x$ is a (unique) non-isolated point of $X$ and $X$ homeomorphically embeds into the Stone space $St(\aA)$ of a Boolean algebra $\aA$, then the image of $x$ in $St(\aA)$ is a strong Nikodym concentration point of some anti-Nikodym sequence of measures on $\aA$. This description can be naturally translated into the language of filters of neighborhoods of points in Stone spaces in the following way: given a point $y$ in the Stone space $St(\aA)$ of a Boolean algebra $\aA$, if there is a countable discrete set $Y\sub St(\aA)$ such that $y\in\ol{Y}\sm Y$ and the filter $\big\{Y\cap\clopen{A}_\aA\colon\ A\in y\big\}$ on $Y$ is isomorphic to some filter $\fF$ on $\omega$ such that $\fF^*\in\aA\nN$, then $y$ is a strong Nikodym concentration point of some anti-Nikodym sequence of measures on $\aA$. We are interested in other properties of points in Stone spaces implying that they are strong Nikodym concentration points and thus we ask the following general question.

\begin{question}\label{ques:type_sncp}
    What properties of points in the Stone spaces $St(\aA)$ of Boolean algebras $\aA$ imply that they are strong Nikodym concentration points of some anti-Nikodym sequences on $\aA$?
\end{question}

Note that the above issue seems to be closely connected to yet another general problem asking what topological properties of points in compact spaces imply that they (do not) belong to the supports of non-atomic measures, see \cite{BNR23}.

In Example \ref{example:strong_consistently_no} we mentioned several consistent constructions of Boolean algebras without the Nikodym property which is witnessed only by anti-Nikodym sequence with no strong Nikodym concentration points. Since we are not aware of any \textsf{ZFC} example of such an algebra, we pose the following question.

\begin{question}\label{ques:no_strong_cp}
    Does there exist in \textup{\textsf{ZFC}} a Boolean algebra $\aA$ without the Nikodym property and such that every anti-Nikodym sequence of measures on $\aA$ has no strong Nikodym concentration points in $St(\aA)$?
\end{question}

There exist $\F_{\sigma}$ P-filters $\fF$ on $\omega$ for which the Boolean algebras $\aA_\fF$ have the Nikodym property, for example one can take filters constructed in \cite[Theorem 4.12]{FT19} or in \cite[Example 6.15]{MS24}, and appeal to Theorem \ref{thm:p_ideal_nik_equivalences}. %However, we cannot decide if these algebras have the Grothendieck property by the methods used in this paper, as these filters are constructed using pathological submeasures.} {\color{blue}But we know that they do have the Grothendieck property :-)}
Similarly, the filter $\zZ^*$, dual to the asymptotic density ideal $\zZ$, is an $\F_{\sigma\delta}$ non-$\F_\sigma$ P-filter also such that the algebra $\aA_{\zZ^*}$ has the Nikodym property. We do not know however of any constructions of Borel filters $\fF$ on $\omega$ which do not belong to the Borel class $\F_{\sigma\delta}$ and for which the Boolean algebras $\aA_\fF$ have the Nikodym property.

\begin{question}
    Does there exist a non-$\F_{\sigma\delta}$ Borel filter $\fF$ on $\omega$ for which the Boolean algebra $\aA_\fF$ has the Nikodym property?% (and has/does not have the Grothendieck property)? 
\end{question}

Corollary \ref{cor:p_anal_lgbps_lgbpl} asserts that properties (LGBPs) and (LGBPl) are equivalent in the class of analytic P-ideals on $\omega$. We are however not aware of any ideal on $\omega$ outside of this class which has only (LGBPl).

\begin{question}
    Does there exist an ideal on $\omega$ with (LGBPl) but without (LGBPs)?
\end{question}

In Section \ref{sec:hypergraph} we studied the class of hypergraph ideals. In particular, in Theorem \ref{thm:hypergraph_nikodym_not_tot_bnd} we showed that there exists $\frakc$ many pairwise non-isomorphic hypergraph ideals which have the Nikodym property but are not totally bounded. We then ask if the non-isomorphicity may be strengthened to the Kat\v{e}tov incomparability.

\begin{question}\label{ques:incomp_hypergraph}
Does there exist a family of $\frakc$ many pairwise Kat\v{e}tov incomparable hypergraph ideals which have the Nikodym property and are not totally bounded?   
%{\color{red} And without the assumption of being not totally bounded?}
\end{question}

A weaker and seemingly easier variant of Question \ref{ques:incomp_hypergraph} may read as follows.

\begin{question}
Does there exist a family of $\frakc$ many pairwise Kat\v{e}tov incomparable non-pathological ideals which have the Nikodym property and are not totally bounded?   
%{\color{red} And without the assumption of being not totally bounded?}
\end{question}

We finish the paper with the following question connected with Corollary \ref{cor:tukey}. Let $\mathbb{AN}$ denote the class of all ideals $\iI$ on $\omega$ which do not have the Nikodym property (so, e.g., $\mathcal{AN}\sub\mathbb{AN}$).

%{\color{red} Dobre pytanie! Bo niewiele wiemy o klasie $(\mathbb{AN}\sm\mathcal{AN}$, gdyż do "zabijania własności Nikodyma" używaliśmy zwykle P-ideałów z klasy $\mathcal{AN}$
\begin{question}
What are the Tukey types of the ordered sets $(\mathbb{AN},\le_K)$ and $(\mathbb{AN}\sm\mathcal{AN},\le_K)$?
\end{question}

%{\color{red}Note that, by the remark after Theorem \ref{thm:mazur_submeasures}, the class of $\F_\sigma$ ideals without the Nikodym property ordered by $\leq_K$ is cofinal in the order $(\mathbb{AN},\le_K)$, and so these orders are Tukey equivalent. The same remark appeals to the order $(\mathbb{AN}\sm\mathcal{AN},\le_K)$.}

\end{document}